\documentclass[11pt,reqno]{amsart}

\usepackage{
  mathabx,
  amsmath,
  amsfonts,
  amssymb,
  amsthm,
  amscd,
  graphicx,
  mathtools,    
  enumitem,
  mathdots,     
  txfonts,
  stackrel
}
\usepackage[all]{xy}
\usepackage[usenames,dvipsnames]{xcolor}
\usepackage{todonotes}
\usepackage{tikz}
\usetikzlibrary{decorations.markings}
\usetikzlibrary{cd}
\usetikzlibrary{patterns}
\usetikzlibrary{shapes}
\usepackage{todonotes}

\makeatletter
\@namedef{subjclassname@2020}{%
  \textup{2020} Mathematics Subject Classification}
\makeatother

\usepackage{bbm}

\DeclareFontFamily{OT1}{pzc}{}
\DeclareFontShape{OT1}{pzc}{m}{it}{<-> s * [1.10] pzcmi7t}{}
\DeclareMathAlphabet{\mathpzc}{OT1}{pzc}{m}{it}

\usepackage[colorlinks=true, linkcolor=blue, citecolor=purple, urlcolor=blue, breaklinks=true]{hyperref}

\usepackage{cleveref}

\crefname{definition}{Definition}{Definitions}
\crefname{example}{Example}{Examples}
\crefname{lemma}{Lemma}{Lemmas}
\crefname{corollary}{Corollary}{Corollaries}
\crefname{theorem}{Theorem}{Theorems}
\crefname{remark}{Remark}{Remarks}
\crefname{equation}{}{}
\crefname{enumi}{}{}
\crefname{section}{Section}{Section}

\tikzset{anchorbase/.style={outer sep=auto,baseline={([yshift=-0.5ex]current bounding box.center)}}}
\renewcommand{\ngeq}{{\mathchoice
{\displaystyle\geq\hspace{-3.5mm}/\hspace{1mm}}
{\textstyle\geq\hspace{-3.5mm}/\hspace{1mm}}
{\scriptstyle\geq\hspace{-1.8mm}/\hspace{.6mm}}
{\scriptscriptstyle\geq\hspace{-1.5mm}/\hspace{.6mm}}}}

\tikzset{anchorbase/.style={outer sep=auto,baseline={([yshift=-0.5ex]current bounding box.center)}}}
\renewcommand{\ng}{{\mathchoice
{\displaystyle>\hspace{-3.5mm}/\hspace{1mm}}
{\textstyle>\hspace{-3.5mm}/\hspace{1mm}}
{\scriptstyle>\hspace{-1.8mm}/\hspace{.6mm}}
{\scriptscriptstyle>\hspace{-1.5mm}/\hspace{.6mm}}}}
\newcommand{\nl}{{\mathchoice
{\displaystyle< \hspace{-3.2mm}/\hspace{.9mm}}
{\textstyle< \hspace{-3.2mm}/\hspace{.9mm}}
{\scriptstyle< \hspace{-1.6mm}/\hspace{.6mm}}
{\scriptscriptstyle< \hspace{-1.4mm}/\hspace{.6mm}}}}

\def\supp{\operatorname{supp}}

\newcommand\Z{\mathbb{Z}}
\newcommand\Q{\mathbb{Q}}

\newcommand\N{\mathbb{N}}
\newcommand\kk{\Bbbk}

\def\op{{\operatorname{op}}}

\def\eps{\varepsilon}

\def\ob{{\operatorname{ob\,}}}

\def\S{\mathbf{S}}
\def\B{\mathbf{B}}
\def\I{\mathbf{I}}
\def\X{\mathrm{X}}
\def\Y{\mathrm{Y}}
\def\H{\mathrm{H}}

\newcommand\GVec{\mathpzc{GVec}}

\def\lround{(\!(}
\def\rround{)\!)}

\DeclareMathOperator{\End}{End}
\DeclareMathOperator{\Ext}{Ext}
\DeclareMathOperator{\Tor}{Tor}
\DeclareMathOperator{\Hom}{Hom}
\def\soc{\operatorname{soc}}
\def\head{\operatorname{cosoc}}
\def\rad{\operatorname{rad}}
\DeclareMathOperator{\id}{id}

\DeclareMathOperator{\im}{im}
\DeclareMathOperator{\coker}{coker}

\def\taudual{{\scriptscriptstyle\bigcirc\hspace{-1.35mm}\tau\hspace{.1mm}}}

\def\gmodpuredelta{\operatorname{\!-gmod}_{\pureDelta}}
\def\gmodpurenabla{\operatorname{\!-gmod}_{\purenabla}}
\def\gmodpuredeltabar{\operatorname{\!-gmod}_{\bar\pureDelta}}
\def\gmodpurenablabar{\operatorname{\!-gmod}_{\bar\purenabla}}
\def\gmoddelta{\operatorname{\!-gmod}_\Delta}
\def\gmodnabla{\operatorname{\!-gmod}_\nabla}
\def\gmoddeltabar{\operatorname{\!-gmod}_{\bar\Delta}}
\def\gmodnablabar{\operatorname{\!-gmod}_{\bar\nabla}}
\def\gmod{\operatorname{\!-gmod}}
\def\domg{\operatorname{gmod-\!}}
\def\gmodlfd{\operatorname{\!-lfdmod}}
\def\dfldomg{\operatorname{lfdmod-\!}}
\def\pgmod{\operatorname{\!-pgmod}}
\def\domgp{\operatorname{pgmod-\!}}
\def\igmod{\operatorname{\!-igmod}}
\def\domgi{\operatorname{igmod-\!}}

\newcommand{\pureDelta}
{\mathchoice   {\displaystyle\Delta\hspace{-2.48mm}\raisebox{.24pt}{$\blacktriangle$}\hspace{.2mm}}     {\textstyle\Delta\hspace{-2.48mm}\raisebox{.24pt}{$\blacktriangle$}\hspace{.2mm}}{\scriptstyle\Delta\hspace{-2.05mm}\raisebox{.1pt}{$\blacktriangle$}\hspace{-.2mm}}{\scriptscriptstyle\blacktriangle}}

\newcommand{\purenabla}
{\mathchoice   {\displaystyle\nabla\hspace{-2.44mm}\raisebox{1pt}{$\blacktriangledown$}}     {\textstyle\nabla\hspace{-2.44mm}\raisebox{1pt}{$\blacktriangledown$}}{\scriptstyle\nabla\hspace{-2.15mm}\raisebox{.1pt}{$\blacktriangledown$}\hspace{-.25mm}}{\scriptscriptstyle\blacktriangledown}}

\newtheorem{theorem}{Theorem}[section]
\newtheorem{lemma}[theorem]{Lemma}
\newtheorem*{lemma*}{Lemma}
\newtheorem{corollary}[theorem]{Corollary}

\theoremstyle{definition}
\newtheorem{definition}[theorem]{Definition}
\newtheorem{remark}[theorem]{Remark}

\numberwithin{equation}{section}
\allowdisplaybreaks

\begin{document}

\title[Graded triangular bases]{Graded triangular bases}
\author{Jonathan Brundan}
\address{
  Department of Mathematics \\
  University of Oregon \\
  Eugene, OR, USA
}
\urladdr{\href{https://pages.uoregon.edu/brundan}{https://pages.uoregon.edu/brundan}, \textrm{\textit{ORCiD}:} \href{https://orcid.org/0009-0009-2793-216X}{0009-0009-2793-216X}}
\email{brundan@uoregon.edu}

\thanks{Research supported in part by NSF grant
  DMS-2101783.}


\begin{abstract}
This article develops a practical technique for studying representations of $\kk$-linear categories arising in the categorification of quantum groups.
We work in terms of locally unital algebras which are $\Z$-graded with graded pieces that are 
finite-dimensional and bounded below,
developing a theory of {\em graded triangular bases} for such algebras.
The definition is a graded extension of the notion of triangular basis as formulated in \cite{BS}. However, in the general graded setting, finitely generated projective 
modules often fail to be Noetherian,
so that existing results from the study of highest weight categories are not directly applicable. Nevertheless, we show that there is still a good theory of {\em standard modules}. In motivating examples arising from Kac-Moody 2-categories, these modules categorify the PBW bases for
the modified forms of 
quantum groups constructed by Wang.
\end{abstract}

\maketitle

\vspace{-8mm}

\section{Introduction}\label{intro}

Recently, Wang \cite{PBW} has introduced PBW
bases for 
the modified forms of quantum groups. Similar bases exist also for
iquantum groups.
This article arose from attempts to understand the categorification of these bases. 
Quantum groups are categorified by
the {\em Kac-Moody 2-categories}
of Khovanov and Lauda \cite{KL3} and Rouquier \cite{Rou}.
From this perspective, Wang's PBW bases come from certain {\em standard modules} for the morphism categories of these 2-categories.
More precisely, 
standard modules categorify Wang's fused canonical basis and a variation,
called {\em pure} standard modules, categorify his PBW basis in all finite types.
This will be explained in forthcoming work.
Another example in a similar spirit is
developed in \cite{BWW2}, where we show
that the split Grothendieck ring of the monoidal category of finitely generated graded projective modules for the {\em nil-Brauer category} from \cite{BWWbasis} is isomorphic to the split {iquantum group} 
of rank one.
Again, this iquantum group has a PBW basis
which is categorified by standard modules.

The motivating examples just mentioned
are small graded $\kk$-linear categories over a field $\kk$. The goal
of this article is to develop the algebraic tools 
needed to construct the standard modules for these categories
in the first place.
We are inclined to replace the $\kk$-linear category
in question with its
{\em path algebra} $A$. This is a 
locally unital graded associative algebra
$$
A = \bigoplus_{i,j \in \I} 1_i A 1_j
$$
equipped with a distinguished family of mutually orthogonal homogeneous idempotents
$1_i\:(i \in \I)$ arising from the identity endomorphisms of the objects of the underlying $\kk$-linear category.
The spaces $1_i A 1_j$ are usually infinite-dimensional graded vector spaces, but they are {\em locally finite-dimensional}, i.e.,
the degree $d$ component $1_i A_d 1_j$
is finite-dimensional
for all $d \in\Z$.
Moreover, the grading is {\em bounded below}
in the sense that for each $i,j \in \I$
there exists $N_{i,j} \in \Z$
such that $1_i A_d 1_j = 0$
for all $d < N_{i,j}$.

\begin{definition}\label{raspberries}
Let $A = \bigoplus_{i,j \in \I} 1_i A 1_j$ 
be a locally unital graded algebra that is
locally finite-dimensional and bounded below.
A {\em graded triangular basis} for $A$
is following additional data:                
\begin{itemize}
\item A subset $\S \subseteq \I$ indexing {\em special idempotents}
  $\{1_s\:|\:s\in \S\}$.
\item A {\em lower finite} poset $(\Lambda,\leq)$,
meaning that $\{\mu \in \Lambda\:|\:\mu \leq \lambda\}$ is finite for each $\lambda \in \Lambda$.
\item
A function 
$\partial:\S\rightarrow \Lambda, s \mapsto \dot s$ with finite (possibly empty) fibers $\S_\lambda := \partial^{-1}(\lambda)$ for each $\lambda \in \Lambda$.
\item
Homogeneous sets $\X(i,s) \subset 1_i A 1_s$, 
$\H(s,t) \subset 1_s A 1_t$, $\Y(t,j) \subset 1_t A 1_j$ 
for $i,j\in \I$ and $s,t \in \S$.
\end{itemize}
For $s,t \in \S$, let 
$\X(s) := \bigcup_{i \in \I} \X(i,s)$ and
$\Y(t) := \bigcup_{j \in \I} \Y(t,j)$.
The axioms are as follows:
\begin{enumerate}
\item[(A1)] The products $x h y$ for  $(x,h,y) \in \bigcup_{s,t \in \S} \X(s) \times \H(s,t) \times \Y(t)$
  give a basis for $A$.
  \item[(A2)] For each $s \in \S$,
$\X(s,s)=\Y(s,s) = \{1_s\}$.
\item[(A3)]  For $s,t \in \S$ with $s\neq t$,
$\X(s,t)\neq\varnothing\Rightarrow
\dot s > \dot t$, 
$\H(s,t)\neq\varnothing\Rightarrow\dot s = \dot t$, and
$\Y(s,t)\neq\varnothing\Rightarrow
\dot s < \dot t$.
\item[(A4)] For each $i \in \I-\S$, the
sets $\{s \in \S\:|\:\X(i,s)\neq\varnothing\}$
and 
$\{s \in \S\:|\:\Y(s,i)\neq\varnothing\}$
are both finite\footnote{The final axiom is seldom needed; it is applied in \cref{wherefinalaxiomisneeded}.}.
    \end{enumerate}
  \end{definition}

The setup in 
\cref{raspberries}, incorporating three index sets $\I, \S$ and $\Lambda$, is designed to be sufficiently flexible to be applicable directly to the various examples ``in nature''. From a theoretical perspective, one can always reduce to the special case that $\I = \S = \Lambda$, which simplifies the definition; this is discussed further at the start of \cref{sfirst}. See also \cref{nice}, which introduces two particularly well-behaved
special cases in which the set $\S$
also parametrizes the isomorphism classes of irreducible graded left $A$-modules.
One of these special cases, in which $\S = \Lambda$ and the function $\partial$ is the identity,
gives a general definition of a {\em based affine quasi-hereditary algebra}.

The history behind \cref{raspberries} will be 
discussed later in the introduction. 
We just note for now that it is almost exactly the same
as the definition of triangular basis given in \cite[Def.~5.26]{BS},
and that is equivalent to the
definition of weakly triangular decomposition in \cite{GRS}.
The main difference is that we are now in a graded setting, so that the assumption made in \cite{BS,GRS} that each $1_i A 1_j$ is finite-dimensional can be weakened.
We have also reversed the partial order compared to \cite{BS} since it seems more sensible to work in terms of lowest weight rather than highest weight modules in the sort of diagrammatical examples that we are interested in; this is the same convention as in \cite{EL} and \cite{SS}.

When $A$ has a graded triangular basis, 
the category 
$A\gmod$ of (locally unital) 
graded left $A$-modules has 
properties which are reminiscent of various
Abelian categories appearing in Lie theory. 
Here is a brief summary of the 
results developed in the main body of the text:
\begin{itemize}
\item
For each $\lambda \in \Lambda$,
let $e_\lambda := \sum_{s \in \S_\lambda} 1_s$. The {\em $\lambda$-weight space} of a graded left $A$-module $V$ is 
the subspace $e_\lambda V$.
Let $A_{\geq \lambda}$
be the quotient of $A$ by the two-sided ideal generated by all $e_\mu\:(\mu\ngeq\lambda)$.
Then let $A_\lambda := \bar e_\lambda A_{\geq \lambda} \bar e_\lambda$, where $\bar e_\lambda$ is the canonical image of $e_\lambda$ in $A_{\geq \lambda}$.
These are {\em unital} graded algebras which are locally finite-dimensional and bounded below; in the motivating examples coming from Kac-Moody 2-categories they are some quiver Hecke algebras.
\item
The algebras $A_\lambda\:(\lambda \in \Lambda)$ play the role of ``Cartan subalgebra" in a sort of 
lowest weight theory:
if $V$ is any graded left $A$-module and $\lambda$ is a minimal weight of $V$,
there is a naturally induced action of $A_\lambda$ on the $\lambda$-weight space 
$e_\lambda V$.
There are also exact
functors
$j^\lambda_!:A_\lambda\gmod \rightarrow
A\gmod$ 
and 
$j^\lambda_*:A_\lambda\gmod \rightarrow
A\gmod$, which are left and right adjoints
of the idempotent truncation functor
$j^\lambda:A_{\geq\lambda}\gmod \rightarrow A_\lambda\gmod, V \mapsto \bar e_\lambda V$;
see \cref{xmas}.
We call these the {\em standardization}
and {\em costandardization functors},
respectively, following the terminology of \cite{LW}.
\item
Fix also a set $\B = \coprod_{\lambda \in \Lambda} \B_\lambda$ such that $\B_\lambda$ indexes a complete set of irreducible graded left $A_\lambda$-modules $L_\lambda(b)\:(b \in \B_\lambda)$ up to isomorphism and degree shift; these modules are (globally) finite-dimensional since $A_\lambda$ 
is unital. Also let $P_\lambda(b)$
and $I_\lambda(b)$ be a projective cover and an injective hull of $L_\lambda(b)$ in $A_\lambda\gmod$, respectively.
For $b \in \B_\lambda$ we define
{\em standard modules}
$\Delta(b) := j^\lambda_! P_\lambda(b)$,
{\em proper standard modules}
$\bar\Delta(b) := j^\lambda_! L_\lambda(b)$,
{\em costandard modules}
$\nabla(b) := j^\lambda_* I_\lambda(b)$
and
{\em proper costandard modules}
$\bar\nabla(b) := j^\lambda_* L_\lambda(b)$.
We show that $L(b) := \head \bar\Delta(b)
= \soc\bar\nabla(b)$ is irreducible, and
 the modules $L(b)\:(b \in \B)$ give
 a complete set of
irreducible graded left $A$-modules up to isomorphism and degree shift; see \cref{irrclass}.
To keep track of all of these modules, it is helpful to note that there are canonical homomorphisms
$$
P(b)\twoheadrightarrow 
\Delta(b)\twoheadrightarrow  \bar\Delta(b)\twoheadrightarrow L(b) \hookrightarrow \bar\nabla(b)
\hookrightarrow \nabla(b)
\hookrightarrow I(b).
$$
\item
Let $P(b)$ be a projective cover and $I(b)$
be an injective hull of $L(b)$ in $A\gmod$.
We show that $P(b)$ has a {\em $\Delta$-flag}
and $I(b)$ has a {\em $\nabla$-flag}
with multiplicities
satisfying an analog of the BGG reciprocity formula; see \cref{bgginj,bggproj}.
We also introduce
$\bar\Delta$-flags and $\bar\nabla$-flags,
and establish the familiar homological criteria for all of these types of ``good filtrations"; see \cref{citizens,citizenss} (with finiteness assumptions on the flags) and \cref{villains,villainss} (with the finiteness assumptions removed).
\end{itemize}
For experts, there are probably no surprises in the above statements, but it is remarkable that it is possible to develop this theory so fully
given that we have imposed very 
mild finiteness assumptions on $A$.
In fact, in the 
motivating examples, the algebra $A$ fails
to be locally Noetherian---finitely generated projectives often have submodules
that are not themselves finitely generated.
To deal with this, our notion of $\Delta$-flag in this setting allows sections of such filtrations to be 
infinite direct sums of standard modules; see \cref{defup,upity}.
Accordingly, the Grothendieck group
of the exact category of modules with $\Delta$-flags
is a free $\Z\lround q \rround$-module 
(rather than merely a $\Z[q,q^{-1}]$-module)
with basis given by the isomorphism classes
of the standard modules. This is consistent with the completions that
are needed in order to work with the bases from \cite{PBW} integrally
rather than over $\Q(q)$.

There are two more results we would like to summarize here, both of which require some additional hypothesis.
\begin{itemize}
\item
Assuming that $A$ is unital
rather than merely locally unital,
graded Noetherian (both left and right), and that each of the algebras
$A_\lambda$ has finite graded global dimension, 
the algebra $A$ has finite graded global dimension;
see \cref{birdie}.
\end{itemize}
The strong finiteness assumptions in the statement just made are satisfied in many more classical examples.
When they hold,
the category of finitely generated graded left $A$-modules is an example of an
affine properly stratified category in the sense of
\cite[Def.~5.1]{AHW}, and this result
about global dimension can also be deduced from \cite[Cor.~5.25]{AHW}.
Our final observation is as follows:
\begin{itemize}
\item
If each of the algebras $A_\lambda$ has additional structure making them into based affine quasi-hereditary algebras, then there are also {\em pure standard} and {\em pure proper standard}
modules $\pureDelta(b), \bar\pureDelta(b)\:(b \in \B)$ obtained by applying the standardization functors to the standard and proper standard module of each $A_\lambda$,
and {\em pure costandard} and {\em pure proper costandard modules} $\purenabla(b), \bar \purenabla(b)\:(b \in \B)$ obtained by applying the costandardization functor to the costandard and proper costandard modules of each $A_\lambda$.
These
satisfy analogous homological properties to the standard, proper standard, costandard and proper costandard modules in an affine highest weight category; see \cref{breathe,pureexts}.
In this refined setting, there are canonical homomorphisms
$$
P(b)\twoheadrightarrow 
\Delta(b)\twoheadrightarrow \pureDelta(b)
\twoheadrightarrow \bar\pureDelta(b) \twoheadrightarrow \bar\Delta(b)\twoheadrightarrow L(b) \hookrightarrow
\bar\nabla(b) \hookrightarrow\bar\purenabla(b)
\hookrightarrow \purenabla(b) \hookrightarrow \nabla(b)
\hookrightarrow I(b).
$$\end{itemize}
To explain the significance of this last point, 
we say a little more about the application of graded triangular bases to the categorification of PBW bases of the modified form $\dot U$ of 
a quantized enveloping algebra.
The algebra $\dot U$ is obtained by glueing together  
$U^+$ and $U^-$, both of which are isomorphic to Lusztig's algebra $\mathbf{f}$ which is categorified by certain quiver Hecke algebras according to \cite{KL3}.
In \cite{PBW}, {\em two} new families of bases
for $\dot U$ are discussed, one called 
the {\em fused canonical basis}, which exists in general, and the other, called the {\em PBW basis},
which exists in all finite types.
The fused canonical basis
is categorified by 
standard modules arising from a graded triangular basis whose Cartan algebras are tensor products of two quiver Hecke algebras, one arising from the categorification of $U^+$ and the other
from the categorification of $U^-$.
In finite type, these quiver Hecke algebras are based affine quasi-hereditary algebras 
thanks to \cite{abcdefg,AHW}, and it is the pure standard modules resulting from this extra structure 
which categorify Wang's PBW bases.

To conclude the introduction, we make further historical remarks, with apologies to
many contributions in the same spirit which we have surely missed.
\begin{itemize}
\item
The antecedant for this genre is the notion of {\em cellular algebra} formulated by
Graham and Lehrer \cite{GL}.  
There are many other variations in the literature, including cellular categories \cite{W},
graded cellular algebras \cite{HM}, 
affine cellular algebras \cite{KX}, skew cellular algebras \cite{MHR}, and
sandwich cellular algebras \cite{Tubbenhauer}.
However, algebras with triangular bases have more in common with the quasi-hereditary algebras of \cite{CPS} than cellular algebras---our standard modules 
always have a unique irreducible quotient unlike the situation for cellular algebras where there can be strictly more 
cell modules than isomorphism classes of irreducible modules.
\item
Another influential contribution 
is the definition \cite[Def.~2.17]{EL} of 
{\em fibered object-adapted cellular basis}.
Our primary motivating examples, the morphism categories of Kac-Moody 2-categories, were also one of the motivations behind \cite{EL}.
In \cref{raspberries}, we have weakened some of the hypotheses compared to \cite{EL} but strengthened some others. Most significant, in \cite{EL} the algebras $A_\lambda$ are required to be (commutative) subalgebras of $e_\lambda A e_\lambda$, whereas for us they are subquotients.
However, the novelty of the present article 
compared to \cite{EL}
lies in the subsequent theory that we are able to develop, rather than in the definition itself.
\item
Also providing motivation for us was
the definition of {\em based quasi-hereditary algebra} from \cite{KM}, and the older
notion of
{\em standardly based algebra} from \cite{DR}.
However, \cite{KM} and \cite{DR} only consider
finite-dimensional algebras, in particular, the poset $\Lambda$ is finite rather than merely being lower finite.
In \cite[Def.~5.1]{BS},
we simplified the
definition of based quasi-hereditary algebra
and upgraded it
from unital to locally unital algebras.
The result is equivalent
to the notion of
{\em strictly object-adapted cellular basis}
from \cite[Def.~2.4]{EL}, a definition which was designed to capture the properties of 
Libedinsky's double leaf basis for the diagrammatic Hecke category as studied in \cite{EW}.
In \cite[Ch.~5]{BS}, 
we used semi-infinite Ringel duality together with some arguments involving tilting modules adapted from \cite{AST}
to show that all upper finite 
highest weight categories can be realized in terms of based quasi-hereditary algebras. 
Thus, there are already many important examples in the ungraded setting.
\item
In \cite[Def.~5.20]{BS},
the definition of based quasi-hereditary algebra was weakened
to the notion of a
{\em based stratified algebra}. This is almost the same as an algebra with a triangular basis but with one extra axiom requiring that the idempotents $\bar 1_s (s\in S_\lambda)$ are primitive in $A_\lambda$; see also \cref{nice} below. Upper finite 
fully stratified categories whose tilting modules satisfy some additional axioms
can be realized in terms of 
based stratified algebras;
see \cite[Th.~5.24]{BS}.
\item Finally we would like to mention that there is a stronger notion of
{\em triangular decomposition} formalized in \cite[Def.~5.31]{BS}, which is closely related to the notion of {\em triangular category} introduced in \cite{SS}. The latter is particularly
useful in when there is also some  monoidal structure, i.e., one has what Sam and Snowden call a {\em triangular monoidal category}.
Examples include various sorts of Brauer category (both oriented and unoriented) arising from Schur-Weyl dualities, but the notion is too restrictive to capture 
examples like the ones coming from Kac-Moody 2-categories.
\end{itemize}

One unusual feature of the remainder of the text is that we have not included any {\em examples}. 
The historical discussion above points to many classical examples, but really the present setup was developed specifically to treat the examples arising from Kac-Moody 2-categories, and the nil-Brauer category studied in \cite{BWWbasis, BWW2}. The latter is a particularly good example since closed formulae exist for the 
graded composition multiplicities of proper standard modules, making their infinite nature clear; see especially \cite[Sec.~5]{BWW2} which discusses the graded triangular basis explicitly for this example.
We encourage the reader to have this example in plain view when working through the subsequent definitions and proofs in the present paper.
Some familiarity with the general theory of highest weight categories (e.g., see \cite{BS}) is also assumed.

\section{Locally unital graded algebras and their modules}

Throughout the article, we will work 
over an algebraically closed field $\kk$. All algebras, categories, functors, etc. will be assumed to be $\kk$-linear.
We write $\GVec$ for the closed symmetric monoidal category of $\Z$-graded vector spaces with morphisms that are degree-preserving linear maps. 
The upward
degree shift functor\footnote{In the official published version of this text
  the opposite convention is used---there, $q$ is the downward degree
  shift.} 
is denoted by $q$, i.e.,
for a graded vector space 
$V=\bigoplus_{d \in \Z} V_d$
its degree shift $qV$ is the same underlying vector space
with grading defined via
$(q V)_d := V_{d-1}$ for each $d \in \Z$.
For any sort of formal series
$f = \sum_{d \in \Z} r_d q^d$ with each $r_d \in \N$, we write 
$V^{\oplus f}$
for $\bigoplus_{d \in \Z} q^d V^{\oplus r_d}$.
The conjugate series $\overline{f}$ is $\sum_{d \in \Z} r_d q^{-d}$.
For a graded vector space $V = \bigoplus_{d \in \Z} V_d$ with finite-dimensional graded pieces, we write
$$
\dim_q V := \sum_{d \in \Z} (\dim V_d) q^{d}.
$$
Usually for this,  $V$ will be 
finite-dimensional so that $\dim_q V \in \N[q,q^{-1}]$, or
{\em bounded below} in the sense that $V_d = 0$ for $d \ll 0$ so that
$\dim_q V \in \N\lround q\rround$, or
{\em bounded above} in the sense that $V_d = 0$ for $d \gg 0$ so that
$\dim_q V \in \N\lround q^{-1}\rround$.

By a {\em locally unital graded algebra}
we mean a graded associative (but not necessarily unital) algebra
$A$ equipped with a distinguished system
$1_i\:(i \in \I)$ of mutually orthogonal
homogeneous idempotents such that
\begin{equation}\label{thedec}
A = \bigoplus_{i,j \in \I} 1_i A 1_j.
\end{equation}
By a {\em graded left $A$-module}, we mean
a {\em locally unital} graded left $A$-module $V$, i.e.,
$V = \bigoplus_{i \in \I} 1_i V$. 

For graded left $A$-modules $V$ and $W$
and $d \in \Z$,
we write $\Hom_{A}(V,W)_d$
for the vector space of all ordinary $A$-module
homomorphisms $f:V \rightarrow W$
such that $f(V_n) \subseteq W_{n+d}$ for each $n \in \Z$.
Then
$$
\Hom_{A}(V,W) := \bigoplus_{d \in \Z} \Hom_{A}(V,W)_d
$$
is a morphism space in 
the $\GVec$-enriched 
category of graded left $A$-modules. 
We 
denote the underlying category 
consisting of the same objects with
morphism spaces $\Hom_{A}(V,W)_0$
by $A\gmod$.
This is the usual Abelian category of graded modules and degree-preserving module homomorphisms.
It has enough injectives and projectives, indeed, it is a Grothendieck category,
so that homological algebra makes sense in
$A\gmod$. We define
$\Ext^n_{A}(V,W)$ so that it is
 naturally graded just like $\Hom_{A}(V,W)$:
$$
\Ext^n_{A}(V,W) = \bigoplus_{d \in \Z}
\Ext^n_{A}(V,W)_d\qquad\text{ with }\qquad
\Ext^n_{A}(V,W)_d = 
\Ext^n_{A}(q^{d} V,W)_0=\Ext^n_{A}(V,q^{-d} W)_0.
$$
We use $V \cong W$ for isomorphism in $A\gmod$ and $V \simeq W$
if $V \cong q^d W$ for some $d \in \Z$.

We write $A\pgmod$ (resp., $A\igmod$) for the full subcategory
of $A\gmod$ consisting of finitely generated
projective (resp., finitely cogenerated injective) graded modules.
These are additive Karoubian categories equipped with the downward degree shift functor $q$.
We say that a graded left $A$-module $V$ is
{\em locally finite-dimensional}
if $\dim 1_i V_d < \infty$ for all $i \in \I$ and $d \in \Z$. Also it is {\em bounded below}
(resp., {\em bounded above})
if for each $i \in \I$ there exists $N_i\in \Z$ such that
$1_i V_d = 0$ for $d < N_i$ 
(resp., $d > N_i$).
We denote the Abelian category of locally finite-dimensional graded left 
$A$-modules by $A\gmodlfd$.
There are also graded right $A$-modules,
which are of course the same thing as
graded left $A^\op$-modules.
The various categories of graded right $A$-modules are $\domg A$, $\domgp A$, $\domgi A$ and $\dfldomg A$.

For any locally finite-dimensional graded $A$-module $V$
and an irreducible graded $A$-module $L$,
the {\em graded multiplicity}
of $L$ in $V$
is the following formal series with coefficients in $\N$:
\begin{equation}\label{tech}
[V:L]_q := \sum_{d \in \Z} \max 
\bigg(
\left|\{r=1,\dots,n\:|\:V_r / V_{r-1} \cong q^d L \}\right|
\:\bigg|\:
\begin{array}{ll}
\text{for all finite graded filtrations}\\
0 = V_0 \subseteq \cdots \subseteq V_n = V
\end{array}
\bigg)
 q^d.
\end{equation}
If $V$ is bounded below (resp., above) then this is a formal Laurent series
in $\N\lround q \rround$ (resp., $\N\lround q^{-1}\rround$).
For example, taking $A$ to be $\kk$ itself and writing also $\kk$ for the ground field viewed as a one-dimensional graded vector space concentrated in degree zero,
we have that $[V:\kk]_q = \dim_q V$.

There are  exact contravariant functors
\begin{align}\label{politicians}
?^\circledast:&\domg A \rightarrow A \gmod,
&
?^\circledast:&A\gmod \rightarrow \domg A.
\end{align}
The first of these takes a graded left
module $V$ to
$V^\circledast :=
\bigoplus_{i \in \I}\bigoplus_{d \in \Z} (1_i V_{-d})^*$ viewed as a graded right module
with the natural action.
The second functor is defined similarly. 
If $V$ is locally finite-dimensional
 then $(V^{\circledast})^\circledast \cong V$
naturally.
So $?^\circledast$ restricts to quasi-inverse contravariant equivalences
\begin{align}\label{circledast1}
?^\circledast:\dfldomg A&\rightarrow A\gmodlfd,&
?^\circledast:A\gmodlfd&\rightarrow \dfldomg A.
\end{align}
There is a natural isomorphism
\begin{equation}\label{theduals}
\Hom_A(V, W^\circledast)
\cong \Hom_A(W,V^\circledast)
\end{equation}
for any graded left (resp., right) $A$-module $V$ (resp., $W$). 
This implies that
 $?^\circledast:
\domg A \rightarrow (A\gmod)^\op$
is left adjoint to the exact functor
$?^\circledast:(A\gmod)^\op \rightarrow
\domg A$.
Hence, by properties of adjunctions, 
$?^\circledast$ takes projectives 
in $\domg A$ to projectives in $(A\gmod)^\op$,
i.e., injectives in $A\gmod$.
It then follows that
\begin{equation}\label{hands}
\Ext^n_A(V, W^\circledast) \cong \Ext_A^n(W, V^\circledast)
\end{equation}
for a graded left (resp., right) $A$-module $V$ (resp., $W$) and $n \geq 0$. Indeed, we can compute $\Ext^n_A(V,W^\circledast)$
from a projective resolution of $V$. Applying $?^\circledast$ gives an injective resolution of $V^\circledast$, which can be used to compute
$\Ext^n_A(W, V^\circledast)$.
Then \cref{hands} follows using \cref{theduals}.

It will always be the case for us that $A$ itself is locally finite-dimensional and bounded below, by which we mean that 
all of the right $A$-modules $1_i A \:(i \in \I)$ and all of the left $A$-modules $A 1_j\:(j \in \I)$ are locally finite-dimensional and bounded below in the earlier sense.
Assuming this, finitely generated (resp., finitely cogenerated) graded $A$-modules
are locally finite-dimensional and bounded below (resp., above).
In particular, if $L$ is an irreducible graded left $A$-module, it is both finitely generated
and finitely cogenerated, so it is locally finite-dimensional and it is bounded both below
and above. This proves that
\begin{equation}\label{finitude}
\dim 1_i L < \infty
\end{equation}
for any $i \in \I$.
Using also the assumption that $\kk$ is algebraically closed, one
deduces that 
\begin{equation}\label{schurslemma}
\End_A(L) = \kk.
\end{equation}
The functor $?^\circledast$ restricts
to contravariant functors
\begin{align}\label{circledast2}
?^\circledast:\domgp A &\rightarrow A\igmod,&
?^\circledast:A\pgmod &\rightarrow \domgi A.
\end{align}
For this assertion, we have used that the dual of a finitely generated projective is a finitely cogenerated injective, as follows from the discussion in the previous paragraph.
It is also true that the dual of a finitely 
cogenerated injective is a finitely generated projective, so that restrictions
of $?^\circledast$ also give functors
\begin{align}\label{circledast3}
?^\circledast:A\igmod&\rightarrow \domgp A,&
?^\circledast:\domgi A&\rightarrow A\pgmod,
\end{align}
which are quasi-inverses
of the ones in \cref{circledast2}, i.e., these are all contravariant equivalences.
The proof of this needs some further
 argument which will be explained in the proof of the first lemma.

\begin{lemma}\label{tech2}
Suppose that $A$ is a locally unital graded algebra which is locally finite-dimensional and bounded below.
Let $V$ be any graded left $A$-module.
\begin{enumerate}
\item The module $V$ is finitely cogenerated if and only if $\soc V$, the sum of its irreducible graded submodules, is an essential submodule of finite length.
It always has an injective hull $I_V$
in $A\gmod$. When $V$ is finitely cogenerated, $I_V$ is also finitely cogenerated and coincides with the injective hull of $\soc V$ in $A\gmod$.
\item The module $V$ is finitely generated if and only if $\rad V$, the intersection of its maximal graded submodules,
is a superfluous submodule and $\head V := V / \rad V$ is of finite length.
In that case, it has a projective cover 
$P_V$ 
in $A\gmod$, which is itself finitely generated and coincides with the projective
cover of $\head V$ in $A\gmod$.
\item 
The module $V$ is locally finite-dimensional if and only if
$\Hom_A(P_L, V) \cong 
\Hom_A(V, I_L)^\circledast$
is locally finite-dimensional for all irreducible graded left $A$-modules $L$.
When this holds, the graded dimension of this 
morphism space is equal to the graded multiplicity
$[V:L]_q$ defined by \cref{tech}.
\end{enumerate}
\end{lemma}

\begin{proof}
(1) This follows from general principles since $A\gmod$ is a Grothendieck category.

\vspace{1mm}
\noindent
(2)
We have already noted that finitely generated (resp., finitely cogenerated) modules are locally finite-dimensional and bounded below (resp., above).
Consequently, if $V$ is finitely generated
we can apply $?^\circledast$ then
the first part of (1)
 (with $A$ replaced by $A^\op$) then $?^\circledast$ again to deduce that 
 $\rad V$ is superfluous and $\head V$ is of finite length.
 Conversely, if $\rad V$ is superfluous and $\head V$ is of finite length then
 it is clear that $V$ is finitely generated since it is generated by pre-images of generators of $\head V$.
 
 To complete the proof, it suffices to show that any irreducible graded left $A$-module $L$ has a projective cover $P_L$.
 To see this, we pick $i \in \I$ such that $1_i L \neq 0$,
so that $L$ is a quotient of $q^d A 1_i$
for some $d \in \Z$.
Since $q^d A 1_i$ is a finitely generated projective graded left $A$-module,
its dual $(q^d A 1_i)^\circledast$
is finitely cogenerated and injective.
So by (1), $(q^d A 1_i)^\circledast
= I_1 \oplus \cdots \oplus I_n$ 
with each $I_r$ being the injective hull
of an irreducible graded right $A$-module.
We deduce that $q^d A 1_i \cong
P_1 \oplus\cdots \oplus P_n$
for $P_r := I_r^\circledast$.
Since $q^d A 1_i$ is projective, so is each summand $P_r$,
and duality then gives that $P_r$ is the projective cover of its head
which is an irreducible graded left $A$-module.
One of these summands is a projective cover
of the irreducible $L$,
completing the proof.
This argument shows moreover that the duals of finitely cogenerated injective graded right $A$-modules are projective, something which was promised just before the statement of the lemma.

\vspace{1mm}
\noindent
(3) We just prove the assertions involving $P_L$; the ones involving $I_L$
follow by the dual argument.
If $V$ is locally finite-dimensional then
$\Hom_A(P_L, V)$
is locally finite-dimensional since $P_L$
is finitely generated. Also its graded dimension is equal to $[V:L]_q$
by Schur's Lemma \cref{schurslemma} and the definition \cref{tech}.
Conversely, suppose that 
$\Hom_A(P_L, V)$ is locally finite-dimensional for all $L$.
We need to show that $1_i V$ is locally finite-dimensional for $i \in \I$.
Since $A 1_i$ is finitely generated, (2) implies that 
there are irreducible graded left $A$-modules
$L_1,\dots,L_n$ with $L_r \not\simeq L_s$
for $r \neq s$ 
and $f_1,\dots,f_n \in \N[q,q^{-1}]$ 
such that
$$
A 1_i \cong P_1^{\oplus f_1} \oplus \cdots \oplus P_n^{\oplus f_n},
$$
where $P_r$ is a projective cover of $L_r$.
We deduce that $1_i V \cong
\Hom_A(A 1_i, V)$ is locally
finite-dimensional since
each $\Hom_A(P_r, V)$ is locally finite-dimensional by assumption.
\end{proof}

The locally unital algebra $A$ is {\em unital} if and only if $|\{i \in \I\:|\:1_i \neq 0\}| < \infty$.
Then $1_A = \sum_{i \in \I} 1_i$.
More can be said when this holds.
To start with, \cref{finitude} implies that
all irreducible graded $A$-modules 
are actually finite-dimensional.
Moreover,
there are only finitely many of them up to isomorphism and degree shift;
see \cite[Lemma 2.2(i)]{Klecturenotes}
for the proof.
The following is a graded version of the Nakayama Lemma.

\begin{lemma}\label{nakayama}
Suppose that $A$ is a {\em unital} graded algebra which is locally finite-dimensional and
bounded below.
Let $V$ be a graded left $A$-module
which is bounded below. 
If $\Hom_A(V,L) = 0$ for all
irreducible graded left $A$-modules $L$
then $V = 0$.
\end{lemma}

\begin{proof}
Let $N = N(A)$ be the graded Jacobson radical of $A$. The quotient algebra $A / N$ is a 
finite direct product of graded matrix algebras over $\kk$. In particular, it is semisimple. Suppose that $V$ is a non-zero
graded module that is bounded below. 
Let $m \in \Z$ be minimal such that
$V_m \neq 0$. By \cite[Lem.~2.7]{AHW},
there exists $r \geq 1$ such that
$N^r \subseteq \bigoplus_{d \geq 1} A_d$.
We have that
$N^r V \subseteq \bigoplus_{d \geq 1} A_d V
\subseteq \bigoplus_{d \geq 1} V_{m+d}$.
Hence, $N^r V \neq V$, so $N V \neq V$.
As $A/ N$ is graded semisimple, $V / NV$ is a completely reducible graded module, so there exists an irreducible graded left $A$-module $L$
with $\Hom_A(V / NV, L) \neq 0$.
This implies that $\Hom_A(V,L) \neq 0$ as required.
\end{proof}

\begin{lemma}\label{tough}
Suppose that $A$ is a {\em unital} graded algebra 
that is locally finite-dimensional and bounded below.
Any finitely generated (resp., finitely cogenerated) graded left $A$-module $V$
 has a graded filtration
$V = V_0 \supseteq V_1 \supseteq V_2 \supseteq \cdots$
(resp., $0 = V_0 \subseteq V_1 \subseteq \cdots$)
which is exhaustive in the sense that
$\bigcap_{r \geq 0} V_r = 0$
(resp., $\bigcup_{r \geq 0} V_r = V$)
and has sections are irreducible or zero.
\end{lemma}

\begin{proof}
We just prove the result in the finitely generated case,
the other case following by duality.
Let $A_{\geq r} := \bigoplus_{s \geq r} A_s$.
Let $X$ be a finite set of homogeneous generators for $V$.
Since $A A_{\geq r} / A A_{\geq (r+1)}$
is spanned by the image of $\sum_{s \leq r} A_s$, which is finite-dimensional, the sections 
of the exhaustive filtration
$$
V= A X \supseteq A A_{\geq 1} X 
\supseteq A A_{\geq 2} X \supseteq \cdots
$$
are all finite-dimensional. Then each section can be refined to a composition series to obtain a filtration of the desired form.
\end{proof}

The following lemma is stronger than Lemma~\ref{tech2}(1)--(2) since there is no assumption on finite generation or cogeneration here.

\begin{lemma}\label{tech1}
Suppose that $A$ is a {\em unital} graded algebra that is locally finite-dimensional and bounded below.
Let $\{L(b)\:|\:b \in \B\}$
be a full set of irreducible graded left $A$-modules up to isomorphism and degree shift. Let $P(b)$ and $I(b)$ 
be a projective cover and an injective hull of $L(b)$ in $A\gmod$, respectively.
\begin{enumerate}
\item Any graded left $A$-module $V$
that is locally finite-dimensional and bounded below has a projective cover $P_V$
in $A\gmod$, which is itself locally finite-dimensional and bounded below.
Moreover, we have that
\begin{equation}\label{pone}
P_V \cong \bigoplus_{b \in \B}
P(b)^{\oplus\overline{\dim_q \Hom_A(V, L(b))}}
\end{equation}
as a graded left $A$-module.
\item Any graded left $A$-module $V$
that is locally finite-dimensional and bounded above  has an injective hull $I_V$
in $A\gmod$, which is itself both locally finite-dimensional and bounded above.
Moreover, we have that
\begin{equation}\label{ione}
I_V \cong \bigoplus_{b \in \B}
I(b)^{\oplus\dim_q \Hom_A(L(b), V)}
\end{equation}
as a graded left $A$-module.
\end{enumerate}
\end{lemma}

\begin{proof}
(1) 
Let $V$ be a graded left $A$-module
which is locally finite-dimensional and bounded below. The multiplication map
$A \otimes_\kk V \twoheadrightarrow V,
a\otimes v \mapsto av$
is a surjective graded left $A$-module
homomorphism. Also $A \otimes_\kk V$
is a 
projective graded left $A$-module for the action coming from left multiplication on the first tensor factor.
It is locally finite-dimensional and bounded below
since both $A$ and $V$ are.
Thus, 
we have constructed 
$f:P \twoheadrightarrow V$ for $P \in \ob A\gmodlfd$ 
that is bounded below and projective in $A\gmod$. 
Next we apply the functor \cref{circledast2}
to obtain $f^\circledast:V^\circledast \hookrightarrow P^\circledast$ with
$V^\circledast$ and $P^\circledast$ being locally finite-dimensional and bounded above,
and $P^\circledast$ being injective in $\domg A$.

Let $i:V^\circledast
\hookrightarrow I$ be an injective hull of 
$V^\circledast$ in $\domg A$, which exists by general principles
because $\domg A$ is a Grothendieck category.
Using that $P^\circledast$ is injective,
we extend $f^\circledast:V^\circledast \hookrightarrow P^\circledast$ to $g:I \rightarrow P^\circledast$ so that 
the following diagram commutes:
$$
\begin{tikzcd}
\arrow[dr,hookrightarrow,"i" below left,]
V^\circledast\arrow[rr,hookrightarrow,"f^\circledast"]&&P^\circledast\\
&I\arrow[ur,hookrightarrow,"g" below right]&
\end{tikzcd}
$$
Thus $I$ embeds into $P^\circledast$.
It follows that $I$ is locally finite-dimensional and bounded above.
Also $I$ is injective in $\domg A$
so it is certainly injective in the
Abelian subcategory $\dfldomg A$,
and $V^\circledast$ is an essential submodule of $I$.
Finally we dualize again, making some natural identifications
to get a commuting diagram
$$
\begin{tikzcd}
V&&\arrow[ll,twoheadrightarrow,"f"]P\arrow[dl,twoheadrightarrow,"g^\circledast" right]\\
&\arrow[ul,twoheadrightarrow,"i^\circledast" below,]
I^\circledast&
\end{tikzcd}
$$
By duality, $I^\circledast$ is projective in $A\gmodlfd$,
but we do not immediately know that it is injective in $A\gmod$. This follows because the surjection $g^\circledast$ splits to reveal that $I^\circledast$ is a graded summand of $P$, so it is projective in $A\gmod$ as $P$ is so.
Also $\ker i^\circledast$ is a superfluous submodule of
$I^\circledast$ since
$\im i$ was an essential submodule of $I$.
So $I^\circledast$ is a projective cover of $V$ in
$A\gmod$, and it is locally finite-dimensional and bounded below as required.

It remains to prove \cref{pone}.
Take $b \in \B$ and pick a homogeneous basis
$\Theta$
for $\Hom_A(V, L(b))$.
For each $\theta \in \Theta$, we 
use projectivity to construct homogeneous maps $\hat\theta$ making the following diagram commute:
$$
\begin{tikzcd}
& V\arrow[dr,"\theta" above right,twoheadrightarrow]\\
P(b)\arrow[ur,"\hat\theta" above left]\arrow[rr,twoheadrightarrow]&& L(b)
\end{tikzcd}
$$
Note $\deg(\hat \theta) = - \deg(\theta)$.
Let $\theta^\vee\:(\theta \in \Theta)$
be the basis for $\Hom_A(V, L(b))^\circledast$ that is dual to $\Theta$.
We obtain a graded left $A$-module homomorphism
$f_b:P(b)\otimes\Hom_A(V,L(b))^\circledast
\rightarrow V,p \otimes \theta^\vee
\mapsto \hat\theta(p)$.
These homomorphisms for all $b$ combine
to define a graded $A$-module homomorphism
$$
f:\bigoplus_{b \in \B} 
P(b) \otimes \Hom_A(V, L(b))^\circledast
\rightarrow V.
$$
This is surjective by construction.
Moreover, the module $P$ appearing on the left hand side is locally finite-dimensional, bounded below and projective in $A\gmod$.
It follows that there is a surjection
$P \twoheadrightarrow P_V$ from $P$ to the projective cover,
i.e., we have a short exact sequence
$0 \rightarrow K \rightarrow P \rightarrow P_V\rightarrow 0$ for some graded submodule $K$ of $P$.
To complete the proof, we show that $K = 0$.
Applying $\Hom_A(-,L(b))$ to the short exact sequence gives
$0 \rightarrow \Hom_A(P_V,L(b))
\rightarrow \Hom_A(P, L(b))
\rightarrow \Hom_A(K, L(b)) \rightarrow 0$.
As we have that
$\dim_q \Hom_A(P, L(b)) = \dim_q \Hom_A(V, L(b))
=\dim_q \Hom_A(P_V, L(b))$
by the construction, we deduce that 
$\Hom_A(K, L(b)) = 0$
 for all $b \in \B$.
This implies that $K=0$
by \cref{nakayama}.

\vspace{1mm}
\noindent
(2) This follows from (1) (with $A$ replaced by $A^\op$) by applying $?^\circledast$.
\end{proof}

\begin{corollary}\label{noname}
Suppose once again that $A$ is {\em unital}, locally finite-dimensional and
bounded below. Let
 $V$ be a graded left $A$-module
which is locally finite-dimensional and bounded below.
If $\Ext^1_A(V,L) = 0$ for all
irreducible graded left $A$-modules $L$
then $V$ is projective in $A\gmod$.
\end{corollary}

\begin{proof}
By \cref{tech1}(1), $V$ has a projective cover
$P_V$ in $A\gmod$ which is locally finite-dimensional and bounded below.
Moreover, 
$\Hom_A(P_V, L) \cong \Hom_A(V, L)$
for all irreducible graded modules $L$.
We apply $\Hom_A(-,L)$ to 
the short exact sequence
$0 \rightarrow K \rightarrow P_V \rightarrow V \rightarrow 0$
using the assumption that
$\Ext^1_A(V,L) = 0$ to get a short exact sequence
$0 \rightarrow \Hom_A(V,L)
\rightarrow \Hom_A(P_V, L) 
\rightarrow \Hom_A(K,L) \rightarrow 0$.
We have already observed that the first map is an isomorphism. It follows that
$\Hom_A(K, L) = 0$.
By \cref{nakayama}, this implies that $K = 0$,
so $V \cong P_V$ as required.
\end{proof}

\section{First properties of graded triangular bases}\label{sfirst}

Throughout the section, we assume that $A$ has a 
graded triangular basis in the sense of \cref{raspberries}.
We will use obvious notations like
$\S_{\leq \lambda}$ for
$\{s \in \S\:|\:\dot s \leq \lambda\}$,  $\S_{\ngeq \lambda}$ for $\{s \in \S\:|\:\dot s \ngeq \lambda\}$, etc.
Before we do anything interesting with the axioms, we make some general remarks.
\begin{itemize}
\item
The axiom (A1)
implies that $A = \sum_{s \in \S} A 1_s A$.
It follows that $A$ is graded Morita equivalent to the idempotent
truncation $\bigoplus_{s,t \in \S} 1_s A 1_t$.
This algebra also has a graded triangular basis that is the obvious subset of the one for $A$.
In this way, one can always reduce to the case that 
$\I=\S$, at the price of replacing $A$ by a Morita equivalent algebra.
\item
Without changing the algebra $A$, merely contracting its distinguished idempotents, one can always reduce to 
a situation in which $\S = \Lambda$.
To do this
starting from the general setup of
\cref{raspberries}, we 
first replace $\Lambda$ by the image of the function
$\S \rightarrow \Lambda, s \mapsto \dot s$. Assuming also
that the sets $\Lambda$ and $\I - \S$ are disjoint,
we 
define $\tilde \X(i,\lambda) := \bigcup_{s \in \S_\lambda} \X(i,s)$
and $\tilde \Y(\lambda,j) := \bigcup_{t \in \S_\lambda} \Y(t,j)$
for $\lambda \in \Lambda$, $i,j \in \I-\S$.
Also for $\lambda,\mu \in \Lambda$, we let
$\tilde \H(\lambda,\mu) := \bigcup_{s \in \S_\lambda, t \in \S_\mu} \H(s,t)$,
and we set 
$\tilde \X(\lambda,\mu) := \bigcup_{s \in \S_\lambda, t \in \S_\mu} \X(s,t)$ and
$\tilde \Y(\lambda,\mu) := \bigcup_{s \in \S_\lambda, t \in \S_\mu} \Y(s,t)$
assuming that $\lambda \neq \mu$.
Finally, let $\tilde \I := (\I-\S) \cup\Lambda$
and define $\tilde 1_i$ to be $1_i$ for $i \in \I - \S$
or $\sum_{s \in \S_\lambda} 1_s$ for $i=\lambda \in \Lambda$,
then set
$\tilde \X(\lambda,\lambda) = \tilde \Y(\lambda,\lambda) := \{\tilde 1_\lambda\}$.
This data gives a new graded triangular basis for $A = \bigoplus_{i,j \in \tilde\I} \tilde 1_i A \tilde 1_j$ 
with special idempotents indexed by
the weight poset $\Lambda \subset \tilde\I$,
which is what we wanted.
\end{itemize}
Taken together, these reductions 
reduce to the case that $\S = \I = \Lambda$.
Although harmless, we have not assumed this since it
is not so convenient in the motivating examples discussed in the introduction.

Returning to the general setup, we proceed to develop some basic consequences
of \cref{raspberries}.
For $\lambda \in \Lambda$,
let $e_\lambda := \sum_{s \in \S_\lambda} 1_s$.
Note it is perfectly possible that $e_\lambda=0$,
indeed, the idempotents $1_s$ can already be zero,
and also it could be that $\S_\lambda = \varnothing$ since we did not assume that the function
$\S \rightarrow \Lambda$ is surjective.
The {\em $\lambda$-weight space} of 
a graded left $A$-module $V$ is the subspace $e_\lambda V$. Then the {\em set of
 weights} of $V$ is 
\begin{equation}
\Lambda(V):=\{\lambda \in \Lambda\:|\:e_\lambda V \neq 0\}.
\end{equation}
Let $A_{\geq \lambda}$
be the quotient of $A$ by the two-sided ideal generated by the idempotents
$\{e_\mu\:|\:\mu\ngeq\lambda\}$.
We often use the notation
$\bar a$ to denote the image of $a \in A$
in $A_{\geq \lambda}$.
The algebra $A_{\geq \lambda}$ is another 
locally unital graded algebra with distinguished idempotents $\bar 1_i\:(i \in \I)$, and it is locally finite-dimensional and bounded below since $A$ is so by assumption.
Let $A_\lambda := \bar e_\lambda A_{\geq \lambda} \bar e_\lambda$. This is a {\em unital} graded algebra which is locally finite-dimensional and bounded below; its
identity element is $\bar e_\lambda$.

\begin{lemma}\label{parasites}
Any
element $f$ of the
two-sided ideal $A e_\lambda A$ 
can be written as a linear combination of
elements of the form $x hy$ for $
(x,h,y) \in \bigcup_{s,t \in \S_{\leq \lambda}}
 \X(s)\times \H(s,t)\times \Y(t)$.
\end{lemma}

\begin{proof}
We argue by induction up the poset.
By (A1) and (A3), we may assume that
$f = x_1 h_1 y_1 x_2 h_2 y_2$
for $x_1 \in \X(s_1), h_1\in \H(s_1,t_1),
y_1 \in \Y(t_1,u), 
x_2 \in \X(u,t_2), h_2 \in \H(t_2,s_2),
y_2 \in \Y(s_2)$,
$s_1,t_1,t_2,s_2,u \in S$
with 
$\dot u = \lambda \in \Lambda$
and $\dot s_1 = \dot t_1 \leq 
\lambda \geq \dot t_2 = \dot s_2$.
If $\dot t_1 < \lambda$ or $\lambda > \dot t_2$
we get done by induction, so we may assume
that $\dot t_1 = \lambda = \dot t_2$.
But then by (A2) we must have that 
$t_1 = u = t_2$ and $y_1 = 1_u = x_2$.
So $f = x_1h_1 h_2 y_2$.
Then we expand $h_1 h_2$
in terms of the basis to get a linear combination of terms
$x_3 h_3 y_3$
for $x_3 \in \X(s_1,s_3),
h_3\in \H(s_3,t_3), y_3 \in \Y(t_3,s_2)$ 
for $s_3,t_3 \in S_{\mu}$ and $\mu \leq \lambda$. 
It remains to show that the resulting 
$x_1 x_3 h_3 y_3 y_2$ can be written in the desired form.
If $\mu < \lambda$ this follows by induction, so assume that $\mu=\lambda$.
Then we must have $s_1 = s_3$
and $x_3 = 1_{s_1}$, and $t_3 = s_2$
and $y_3 = 1_{s_2}$.
The term simplifies to $x_1 h_3 y_2$,
which is of the desired form.
\end{proof}

\begin{corollary}\label{climbing}
Suppose we are given a partition $\Lambda = \hat\Lambda\sqcup\check\Lambda$ 
with $\hat\Lambda$ being an upper set, equivalently, $\check\Lambda$ being a lower set.
The quotient $\hat A$ of $A$ by the 
two-sided ideal $I$ generated
by $\{e_{\lambda}\:|\:\lambda \in \check\Lambda\}$ has
basis given by the images of of all
$x h y$ for $(x,h,y) \in \bigcup_{s,t \in \hat\S}
\X(s)\times \H(s,t)\times \Y(t)$
where $\hat \S := \{s \in \S\:|\:\dot s \in \hat \Lambda\}$.
\end{corollary}

\begin{proof}
In view of (A1), it suffices to show that $I$
is spanned by all
$x h y$ for $(x,h,y) \in \bigcup_{s,t \in \check\S}
\X(s)\times \H(s,t)\times \Y(t)$
where
$\check \S := \{s \in \S\:|\:\dot s \in \check\Lambda\} = \S - \hat\S$.
This follows from \cref{parasites}.
\end{proof}

\begin{corollary}\label{thimbles}
For $\lambda \in \Lambda$, 
$A_{\geq\lambda}$ has a basis
given by all $\bar x \bar h \bar y$
for
$(x,h,y) \in \bigcup_{s,t \in \S_{\geq \lambda}}\X(s)\times \H(s,t)\times \Y(t)$.
Hence,
$A_\lambda = \bar e_{{\lambda}} A_{\geq\lambda} \bar e_{{\lambda}}$ has
basis consisting of all $\bar h$ for $h \in \bigcup_{s,t \in \S_\lambda} \H(s,t)$.
\end{corollary}

In the setup of \cref{climbing},
we will always {\em identify}
$\hat A\gmod$ with the full subcategory of $A\gmod$
consisting of the $A$-modules annihilated by all
$e_\lambda\:(\lambda \in \check\Lambda)$.
There is an adjoint triple of functors
$(i^*, i, i^!)$
with 
\begin{equation}
i:\hat A\gmod \rightarrow A \gmod
\end{equation}
being the (often omitted) natural inclusion functor
and
\begin{align}\label{func1}
i^*:= \hat A \otimes_A -:&A\gmod \rightarrow \hat A \gmod,\\
i^! := \bigoplus_{i \in \I}
\Hom_A(\hat A 1_i, -):&A\gmod\rightarrow \hat A\gmod.\label{func2}
\end{align}
We clearly have that $i^* \circ i =i^! \circ i= \id_{\hat A\gmod}$.
In the special case that $\hat A = A_{\geq \lambda}$,
we denote the adjoint triple
$(i^*, i, i^!)$ instead by
$(i^*_{\geq \lambda}, i_{\geq \lambda}, i^!_{\geq \lambda})$:
\begin{equation}\label{recollement1}
\begin{tikzpicture}[anchorbase]
\node(a) at (0,0) {$A_{\geq \lambda}\gmod$};
\node(b) at (4,0) {$A\gmod$};
\draw[-latex] (a) -- (b);
\draw[bend right=40,-latex] (b.north west) to (a.north east);
\draw[bend left=40,-latex] (b.south west) to (a.south east);
\node at (2,0.2) {$\scriptstyle i_{\geq \lambda}$};
\node at (2,1) {$\scriptstyle i^!_{\geq \lambda}$};
\node at (2,-1) {$\scriptstyle i^*_{\geq \lambda}$};
\end{tikzpicture}
\end{equation}
More explicitly,
for a graded left $A$-module $V$,
$i_{\geq \lambda}^* V$ is the largest graded quotient 
and $i_{\geq \lambda}^! V$ is the largest
graded submodule of $V$ all of whose weights
are $\geq \lambda$.

\begin{lemma}\label{green}
Let $V$ be a graded left $A$-module
and $\lambda$ be minimal in $\Lambda(V)$.
\begin{enumerate}
\item
We have that $e_\mu A e_\lambda V = 0$
unless $\mu \geq \lambda$. Hence, the natural inclusion
$\bar e_\lambda\big(i^!_{\geq \lambda} V\big) \hookrightarrow e_\lambda V$
is an isomorphism of graded vector spaces.
\item
We have that $e_\lambda A e_\mu V = 0$
unless $\mu \geq \lambda$. 
Hence, the natural quotient map $e_\lambda V \twoheadrightarrow
\bar e_\lambda\big(i^*_{\geq \lambda} V\big)$
is an isomorphism of graded vector spaces.
\end{enumerate}
\end{lemma}

\begin{proof}
(1) The subspace 
$e_\mu A e_\lambda V$ is spanned
by vectors $xhy v$ for 
$x \in \X(s_1,s_2), h \in \H(s_2,t_2), y \in \Y(t_2,t_1)$
and
$v \in 1_{t_1} V$
with $\dot s_1 = \mu,
\dot s_2 = \nu = \dot t_2, \dot t_1 = \lambda$ and $\mu \geq \nu \leq \lambda$. The minimality of $\lambda$ implies that
$xhyv = 0$ unless $\nu = \lambda$, in which case $\mu \geq \lambda$.
It follows that the submodule $A e_\lambda V$ is contained in $i^!_{\geq \lambda} V$, so their $\lambda$-weight spaces coincide.

\vspace{1mm}
\noindent
(2) The proof that
$e_\lambda A e_\mu V = 0$
unless $\mu \geq \lambda$ is similar to the proof in (1). 
To deduce that
$e_\lambda V\cong \bar e_\lambda\big( i^*_{\geq \lambda} V\big)$, note that
$i^*_{\geq \lambda} V = V / \sum_{\mu\ngeq \lambda} A e_\mu V$.
We have shown that the $\lambda$-weight
space of each $A e_\mu V$ appearing here is zero, so the quotient map 
restricts to an isomorphism between
the $\lambda$-weight spaces of $V$
and $i^*_{\geq \lambda} V$.
\end{proof}

\section{Standard modules and the classification of irreducible modules}\label{saturday}

Suppose to start with that $A$ is any locally unital graded algebra as in \cref{thedec}.
Let $e$ be an idempotent in $A$
that is a finite sum of the distinguished idempotents
$1_i\:(i \in \I)$. Then $eAe$ is a unital graded algebra.
Truncating a module with the idempotent $e$
defines an exact 
functor 
\begin{equation}
j:A\gmod \rightarrow eAe\gmod, V \mapsto eV.
\end{equation}
It is well known that $j$ takes irreducible graded $A$-modules to irreducible graded $eAe$-modules or to zero, 
and all irreducible graded left $eAe$-modules arise in this way. 
Moreover, $j$ satisfies the universal property of 
quotient functor: any exact functor
from $A\gmod$ to an Abelian category which takes
all of the irreducibles annihilated by $j$ to zero
factors uniquely through $j$.
The functor $j$ has a left adjoint $j_!$ and a right adjoint $j_*$ defined by
\begin{align}\label{pinkeye}
j_! := A e \otimes_{eAe} ?&:
eAe\gmod \rightarrow A_{\geq\lambda}\gmod,\\
j_* := \bigoplus_{i \in \I} \Hom_{eAe}(e A1_i,?)&:
              eAe\gmod\rightarrow
              A\gmod.\label{blueye}
\end{align}
Neither $j_!$ nor $j_*$ is exact in general.
We obviously have that $j \circ j_! 
\cong j\circ j_* \cong \id_{eAe\gmod}$.
If $P\twoheadrightarrow L$
(resp., $L \hookrightarrow I$) is a projective
cover (resp., an injective hull) of an irreducible graded left $A$-module $L$ such that $j L \neq 0$
then $jP$ (resp., $jI$) is a projective cover
(resp., an injective hull) of $jL$ in $eAe\gmod$.
Using properties of adjunctions, it 
follows that $j_! j P \cong P$ and $j_* j I \cong I$.

Now return to the setup of the previous section, so that $A$ has a graded triangular basis.
Take any $\lambda \in \Lambda$.
Applying the constructions just explained
to the idempotent $\bar e_\lambda$ in the algebra $A_{\geq \lambda}$ produces an adjoint triple
of functors which we denote by $(j^\lambda_!, j^\lambda,
j^\lambda_*)$:
\begin{equation}\label{recollement2}
\begin{tikzpicture}[anchorbase]
\node(a) at (0,0) {$A_{\geq \lambda}\gmod$};
\node(b) at (4,0) {$A_{\lambda}\gmod$};
\draw[-latex] (a) -- (b);
\draw[bend right=40,-latex] (b.north west) to (a.north east);
\draw[bend left=40,-latex] (b.south west) to (a.south east);
\node at (2,0.2) {$\scriptstyle j^\lambda$};
\node at (2,1) {$\scriptstyle j^\lambda_*$};
\node at (2,-1) {$\scriptstyle j^\lambda_!$};
\end{tikzpicture}
\end{equation}
We call $j^\lambda_!$ and $j^\lambda_*$ 
the {\em standardization}
and {\em costandardization} functors, respectively.
We are in a special situation so that these functors
have additional favorable properties:

\begin{lemma}\label{xmas}
For $\lambda \in \Lambda$,
the functor $j^\lambda_!$ (resp., $j^\lambda_*$)
is exact and it takes modules that are locally finite-dimensional and bounded below (resp., bounded above) to modules that are locally finite-dimensional and bounded below (resp., bounded above).
\end{lemma}

\begin{proof}
The functor
$j^\lambda_!$ is exact because
$\bar 1_i A_{\geq \lambda} \bar e_\lambda$ is a projective graded right
$A_\lambda$-module for each $i \in \I$.
Indeed, by Corollary~\ref{thimbles},
$\bar 1_i A_{\geq \lambda} \bar e_\lambda$ has
basis $\bar x\bar h$ for $(x,h) \in \bigcup_{s,t \in \S_\lambda} \X(i,s) \times \H(s,t)$.
Hence we have that
\begin{equation}\label{thisgivesabasis}
\bar 1_i A_{\geq \lambda} \bar e_\lambda = 
\bigoplus_{s \in \S_\lambda} \bigoplus_{x \in \X(i,s)}
\bar x A_{\lambda}
\end{equation}
with the summand $\bar x A_{\lambda}$ here being isomorphic
to
$q^{\deg(x)}\bar 1_s A_{\lambda}$
as a graded right $A_\lambda$-module, which is projective.
Similarly,
\begin{equation}\label{anotherbasis}
\bar e_\lambda A_{\geq \lambda} \bar 1_i
= \bigoplus_{s \in \S_\lambda} \bigoplus_{y \in \Y(s,i)} A_\lambda \bar y
\end{equation}
with the summand $A_\lambda \bar y$
being isomorphic
to
$q^{\deg(y)}A_{\lambda}\bar 1_s$
as a graded left $A_\lambda$-module.
So
$\bar e_\lambda A_{\geq\lambda}\bar 1_i$ is a projective graded left
$A_\lambda$-module, hence, $j^\lambda_*$ is exact.

Now let $V$ be a graded left $A_{\lambda}$-module 
and let $V(s)$ be a homogeneous basis for $1_s V$
for $s \in \S_\lambda$.
The decomposition \cref{thisgivesabasis} implies that 
$1_i \left(j^\lambda_! V\right) = \bar 1_i A_{\geq \lambda} \bar e_\lambda \otimes_{A_\lambda} V$
has homogeneous basis given by the vectors
\begin{equation}\label{thisisabasis}
\bar x \otimes v
\qquad\text{for }
(x,v) \in \bigcup_{s \in \S_\lambda} \X(i,s) \times V(s).
\end{equation}
The vector $\bar x\otimes v$ is of degree $\deg(x)+\deg(v)$.
Since $A$ is locally finite-dimensional and bounded below
and $\S_\lambda$ is finite,
there are only finitely many $x \in \bigcup_{s \in \S_\lambda}
\X(i,s)$ of any given degree, and these degrees are bounded below.
This implies that
$j^\lambda_! V$ is locally finite-dimensional and bounded below assuming $V$ has these properties.
Similarly, from \cref{anotherbasis},
we deduce that $1_i \left(j^\lambda_* V\right) = \Hom_{A_\lambda}(\bar e_\lambda 
A_{\geq \lambda}\bar 1_i, V)$
has basis 
\begin{equation}\label{theotherbasis}
\delta_{y,v}\qquad
\text{for }
(y,v) \in \bigcup_{s \in \S_\lambda} \Y(s,i) \times V(s),
\end{equation}
where $\delta_{y,v}$ is the unique left $A_\lambda$-module
homomorphism that takes $\bar y \in \Y(s,i)$ 
to $v$ and all other elements of 
$\bigcup_{t \in \S_\lambda} \Y(t,i)$
to zero. Since $\deg(\delta_{y,v}) = \deg(v)-\deg(y)$,
it is easy to deduce that $j^\lambda_* V$ is locally finite-dimensional and bounded above assuming that $V$ has these properties.
\end{proof}

Next, we fix a set $\B = \coprod_{\lambda \in \Lambda} \B_\lambda$ 
 such that $\B_\lambda$ parametrizes 
a full set
$L_\lambda(b)\:(b \in \B_\lambda)$
of irreducible
graded left $A_\lambda$-modules up to isomorphism and degree shift.
Given $b \in \B$, we use the notation
$\dot b$ to denote the unique $\lambda \in \Lambda$ such that $b \in \B_\lambda$.
For this notation to be unambiguous, one should assume that the sets are chosen so that
$\B_\lambda \cap \S = \B \cap \S_\lambda$.
As we did with $\S$, we also use notations 
like $\B_{\leq \lambda}, \B_{\geq \lambda}$, etc.
Since $A_\lambda$ is a unital graded algebra which is locally finite-dimensional and bounded below, the set $\B_\lambda$ is finite and 
each $L_\lambda(b)$ is finite-dimensional.
Also let $P_\lambda(b)$ (resp., $I_\lambda(b)$) be a projective cover 
(resp., injective hull) 
of $L_\lambda(b)$ in $A_\lambda\gmod$; these modules may 
be infinite-dimensional. 
For any $b \in \B$, 
we let
\begin{align}\label{standards}
\Delta(b)& :=  j^\lambda_! P_\lambda(b),&
\bar\Delta(b) &:= j^\lambda_! L_\lambda(b),&
\bar\nabla(b)& := j^\lambda_* L_\lambda(b),&
\nabla(b) &:= j^\lambda_* I_\lambda(b),
\end{align}
where $\lambda := \dot b$.
We view all of these as graded left 
$A$-modules via the natural inclusion $i_{\geq \lambda}$.
We call them the {\em standard,
proper standard, proper costandard} and {\em costandard} modules, respectively. 
If one knows bases for $P_\lambda(b),
L_\lambda(b)$ and $I_\lambda(b)$,
one obtains bases for $\Delta(b)$ and $\bar\Delta(b)$ from \cref{thisisabasis},
and bases for $\nabla(b)$ and $\bar\nabla(b)$
from \cref{theotherbasis}.

In general, there is no reason for any of the modules \cref{standards} to have finite length.
However, by \cref{tough}, each
$P_\lambda(b)\:(b \in \B_\lambda)$
admits an exhaustive descending filtration 
with irreducible sections. By exactness of $j^\lambda_!$, it follows
that $\Delta(b)$ has an exhaustive descending filtration
with top section $\bar\Delta(b)$ and other sections that
are degree shifts of $\bar\Delta(c)$ for $c \in \B_\lambda$.
Similarly, $\nabla(b)$ has an exhaustive ascending filtration with bottom section $\bar\nabla(b)$ and other sections that are degree shifts of $\bar\nabla(c)$ for $c \in \B_\lambda$.

\begin{remark}\label{nice}
It is especially convenient when the sets $\B$ and $\S$ are naturally identified.
We record here two special cases of \cref{raspberries} where this can be achieved.
\begin{itemize}
\item
We call $A$ a {\em based affine quasi-hereditary algebra} if
$\S =\Lambda$ with the map $\S\rightarrow\Lambda, s\mapsto \dot s$ being the identity,
and each $A_\lambda\:(\lambda \in \Lambda)$ is graded local, i.e., 
the quotient of $A_\lambda$ by its graded Jacobson radical 
$N(A_\lambda)$ is $\kk$.
In this situation, 
$\B_\lambda$ is a singleton. Then
one can choose notation so that $\S = \Lambda=\B$ and $P_\lambda(\lambda) = A_\lambda$
for each $\lambda \in \Lambda$.
When the grading is concentrated in degree zero,
this setup recovers the based quasi-hereditary algebras of
\cite{KM} if $\Lambda$ is finite, or their semi-infinite analog from \cite[Def.~5.1]{BS} when $\Lambda$ is infinite.
\item
We call $A$ a {\em based affine stratified algebra} if 
$A_\lambda / N(A_\lambda) \cong \prod_{s \in \S_\lambda} \kk$
for each $\lambda \in \Lambda$.
In this situation, we can 
choose notation so that $\S_\lambda = \B_\lambda$ 
for each $\lambda \in \Lambda$ and
$P_\lambda(b) = A_\lambda \bar 1_b$ for each $b \in \B_\lambda$.
When the grading is concentrated in degree zero,
this setup recovers
the based stratified algebras of \cite[Def.~5.20]{BS}.
\end{itemize}
\end{remark}

If $V$ is any graded left $A$-module and $\lambda$ is minimal in $\Lambda(V)$,
the weight space $e_\lambda V$ is naturally an $A_\lambda$-module,
with the basis vector $\bar h$ of $A_\lambda$ 
acting simply by multiplication by $h \in \bigcup_{s,t \in \S_\lambda} \H(s,t)$.
This follows from \cref{green}(1).
Clearly, both of the isomorphisms
$e_\lambda V \cong j^\lambda i_{\geq \lambda}^! V$ 
and 
$e_\lambda V \cong j^\lambda i_{\geq \lambda}^* V$ from \cref{green}
are $A_\lambda$-module homomorphisms.
If we take $V$ here to be one of the modules $\Delta(b), \bar\Delta(b),\bar\nabla(b)$ or $\nabla(b)$ for $b \in \B_\lambda$
then $\lambda$ is the {\em lowest weight}
of $V$, i.e., it is the unique minimal weight in $\Lambda(V)$.
Moreover, 
in view of the bases \cref{thisisabasis,theotherbasis},
the lowest weight space $e_\lambda V$ 
simply recovers the $A_\lambda$-module
from which $V$ was constructed in the first place in \cref{standards}. This is a familiar situation since it is entirely analogous to the 
construction of Verma and dual Verma modules for semisimple Lie algebras.
In view of this, the following theorem (and its proof) should come as no surprise.

\begin{theorem}[Classification of irreducible modules]\label{irrclass}
For $b \in \B$, the module $\Delta(b)$ has a unique irreducible quotient denoted $L(b)$.
This is also the unique irreducible submodule of $\nabla(b)$. Moreover, the modules
$L(b)\:(b \in \B)$
give a full set of irreducible graded left $A$-modules up to isomorphism and degree shift.
\end{theorem}

\begin{proof}
Take $b \in \B$ and let $\lambda := \dot b$.
The indecomposable projective $A_\lambda$-module
$P_\lambda(b)$ has a unique maximal graded submodule
$\rad P_\lambda(b)$.
Since $e_\lambda \Delta(b) = \bar e_\lambda \otimes P_\lambda(b) \cong P_\lambda(b)$
as an $A_\lambda$-module,
and $\Delta(b)$ is generated as an $A$-module
by its lowest weight space $e_\lambda \Delta(b)$,
it follows that $\Delta(b)$ has a unique maximal graded submodule, namely,
$\left(\bar e_\lambda \otimes \rad P_\lambda(b)\right)
\oplus \bigoplus_{\mu > \lambda} e_\mu \Delta(b)$.
Hence, $\Delta(b)$ has a unique irreducible quotient $L(b)$. Moreover, $\lambda$ is lowest weight of $L(b)$, and $e_\lambda L(b) \cong L_\lambda(b)$ as a graded $A_\lambda$-module. This implies that
$L(a) \not\simeq L(b)$ for $a \neq b$.

Now we show that any irreducible graded $A$-module $L$ is isomorphic to $L(b)$ for some $b \in B$.
Pick $\lambda$ minimal in $\Lambda(L)$,
so that $e_\lambda L$ is naturally 
an $A_\lambda$-module.
There is a non-zero homogeneous $A_\lambda$-module homomorphism
$f:P_\lambda(b) \rightarrow e_\lambda L
$ for some $b \in \B_\lambda$.
Since $e_\lambda L \cong j^\lambda i_{\geq \lambda}^! L$ as an $A_\lambda$-module
and $\Delta(b) = i_{\geq \lambda} j^\lambda_! P_\lambda(b)$,
the adjunctions produce a non-zero homogeneous
$A$-module homomorphism
$\Delta(b) \rightarrow L$, which is necessarily surjective.
We deduce that $L \simeq L(b)$.
The classification of irreducible modules is now proved.

It remains to show that $\nabla(b)$ has irreducible socle 
$L(b)$. For this, we take $a \in \B$ and compute:
$$
\Hom_A(L(a), \nabla(b))
=\Hom_A(L(a), i_{\geq \lambda} j^\lambda_* I_\lambda(b))
\cong \Hom_{A_\lambda}(j^\lambda i_{\geq \lambda}^* L(a), I_\lambda(b)).
$$
Since $j^\lambda i_{\geq \lambda}^* L(a) = 0$ unless $a \in \B_\lambda$,
in which case it is $L_\lambda(a)$,
we deduce that 
$\Hom_A(L(a), \nabla(b))$
is zero unless $a=b$, when it is $\kk$.
This proves that $\soc \nabla(b) = L(b)$.
\end{proof}

For $b \in \B$, we let $P(b)$ be a projective cover and
$I(b)$ be an injective hull of the irreducible module $L(b)$ in $A\gmod$.
For $b \in \B_\lambda$, $\Delta(b)$ (resp.,
$\nabla(b)$) can also be described as the projective cover (resp., injective hull) of $L(b)$
in $A_{\geq \lambda}\gmod$, and we
have that $P(b) \twoheadrightarrow \Delta(b)$
and $\nabla(b) \hookrightarrow I(b)$.
The following lemma gives characterizations of
$\bar\Delta(b)$ and $\bar\nabla(b)$ in a similar vein.

\begin{lemma}\label{eat}
Suppose that $b \in \B$ and let 
$\lambda := \dot b$.
\begin{enumerate}
\item
The proper standard module $\bar \Delta(b)$
is the largest graded quotient of 
$\Delta(b)$ with the properties $[\bar \Delta(b):L(b)]_q = 1$
and $[\bar \Delta(b):L(c)]_q = 0$ for $b \neq c \in \B_{\ng \lambda}$.
\item
The proper costandard module $\bar \nabla(b)$
is the largest graded submodule of 
$\nabla(b)$ with the properties $[\bar \nabla(b):L(b)]_q = 1$
and $[\bar \nabla(b):L(c)]_q = 0$ for $b \neq c \in \B_{\ng \lambda}$.
\end{enumerate}
\end{lemma}

\begin{proof}
(1)
Let $\lambda := \dot b$.
As noted earlier, $\Delta(b)$ has an exhaustive descending
filtration $V = V_0 \supset V_1 \supseteq \cdots$ with top section $V_0 / V_1 = \bar\Delta(b)$ and other
sections $\simeq \bar\Delta(c)$ for $c \in \B_\lambda$.
It follows that any strictly larger quotient $Q$
of $\Delta(b)$ than $\bar\Delta(b)$
has an irreducible quotient of the form $q^d L(c)$ 
for some $c \in \B_\lambda$. 
Hence, either $[Q:L(b)]_q \neq 1$ or
$[Q:L(c)]_q \neq 0$ for $b \neq c \in \B_{\ng \lambda}$,
violating the properties we wanted.
It remains to see that the quotient $\bar\Delta(b)$
does have these properties. 
We certainly have that $[\bar\Delta(b):L(c)]_q = 0$
if $\dot c \ngeq \lambda$ since $\lambda$ is the lowest weight of $\bar\Delta(b)$.
If $\dot c = \lambda$ then $L(c)$ can be viewed an irreducible
$A_{\geq \lambda}$-module with $j^\lambda 
L(c) \cong L_\lambda(c)$, 
and we have by exactness of $j^\lambda$ that
$$
[\bar\Delta(b):L(c)]_q=
[j^\lambda_! L_\lambda(b):L(c)]_q=
[j^\lambda j^\lambda_! L_\lambda(b):j^\lambda L(c)]_q
= 
[L_\lambda(b): L_\lambda(c)]_q = \delta_{b,c}.
$$

\vspace{1mm}
\noindent
(2) Similar.
\end{proof}

\begin{corollary}\label{up}
For $b,c \in \B$,
we have that
$\dim_q \Hom_A(\Delta(b), \bar\nabla(c))=
\dim_q \Hom_A(\bar\Delta(b), \nabla(c))= \delta_{b,c}.$
\end{corollary}

\begin{proof}
We just explain for $\Hom_A(\Delta(b), \bar\nabla(c))$.
If $b = c$ there is by Schur's Lemma a unique (up to scalars) non-zero 
homomorphism taking the irreducible head of $\Delta(b)$ to the irreducible socle of $\bar\nabla(c)$. Any
non-zero homogeneous homomorphism
$\Delta(b) \rightarrow \bar\nabla(c)$ that is not of this form
 takes the head $L(b)$ of $\Delta(b)$ to an irreducible subquotient of $\bar\nabla(c)$ different from $L(c)$, so we get
that $\dot b > \dot c$ thanks to the lemma.
Also we have that $\dot c \geq \dot b$ since 
there must be an irreducible subquotient of $\Delta(b)$
isomorphic to the socle $L(c)$ of $\bar\nabla(c)$.
This contradiction shows that there are no such homomorphisms.
\end{proof}

\begin{corollary}\label{upup}
For $b, c \in \B$, $f \in \N\lround q \rround$ and $g \in \N\lround q^{-1}\rround$,
we have that
$$
\dim_q \Hom_A\big(\Delta(b)^{\oplus f}, \bar\nabla(c)^{\oplus g}\big)=
\dim_q \Hom_A\big(\bar\Delta(b)^{\oplus f}, \nabla(c)^{\oplus g}\big)= \delta_{b,c}\,\overline{f}\,g \in \N \lround q^{-1} \rround.
$$
\end{corollary}

\begin{proof}
Again we just treat
$\Hom_A\big(\Delta(b)^{\oplus f}, \bar\nabla(c)^{\oplus g}\big)$.
Say $f = \sum_{m \in \Z} r_m q^{m}$
and $g = \sum_{n \in \Z} s_n q^{-n}$.
We need to show that
$$
\dim \Hom_A\Big(\bigoplus_{m \in \Z} q^{m} \Delta(b)^{\oplus r_m},
\bigoplus_{n\in\Z} q^{-n} \bar\nabla(c)^{\oplus s_n}\Big)_{-d} = \delta_{b,c}\sum_{m+n=d} r_m s_n,
$$
which makes sense because $r_m=s_n=0$ for $m,n \ll 0$.
Using that $\Delta(b)$ is finitely generated, we have that
\begin{align*}
\Hom_A\Big(\bigoplus_{m \in \Z} q^{m} 
\Delta(b)^{\oplus r_m},
\bigoplus_{n \in \Z} q^{-n} \bar\nabla(c)^{\oplus s_n}\Big)_{-d}
&\cong \prod_{m \in \Z} \Hom_A\Big(\Delta(b),
\bigoplus_{n \in \Z}q^{-n}\bar\nabla(c)^{\oplus s_n}\Big)^{\oplus r_m}_{m-d}\\
&\cong \prod_{m \in \Z} 
\bigoplus_{n \in \Z} \Hom_A(\Delta(b),
\bar\nabla(c)\big)^{\oplus r_m s_n}_{m+n-d}.
\end{align*}
By \cref{up}, the $\Hom$ space here is zero unless $b=c$ and $m+n=d$, when it is 1-dimensional. So the dimension is $\delta_{b,c} \sum_{m+n=d} r_m s_n$ as required.
\end{proof}

We record also a useful consequence of the Nakayama Lemma
for the algebras $A_\lambda$.

\begin{lemma}\label{nakcor}
Let $V$ be a graded left $A$-module
that is bounded below.
If $\Hom_A(V, \bar\nabla(b)) = 0$
for all $b \in \B$ 
then $V = 0$.
\end{lemma}

\begin{proof}
Suppose that $V \neq 0$.
Let $\lambda$ be minimal in $\Lambda(V)$.
By \cref{nakayama},
there is a non-zero $A_\lambda$-module homomorphism
$e_\lambda V \rightarrow L_\lambda(b)$
for some $b \in \B_\lambda$.
Since $e_\lambda V \cong j^\lambda i_{\geq \lambda}^* V$
and $\bar\nabla(b) = i_{\geq \lambda} j^\lambda_* L_\lambda(b)$,
 we get induced a non-zero
homomorphism 
$V \rightarrow \bar\nabla(b)$.
So $\Hom_A(V, \bar\nabla(b)) \neq 0$.
\end{proof}

\section{Duality}

The definition of graded triangular basis is
symmetric in the sense that if we are given a graded
triangular basis of $A$, then it also gives one for
$A^{\op}$. One just has to swap the sets $\X(s)$ and 
$\Y(s)$. Clearly the algebras $(A^\op)_\lambda$ arising from this new triangular basis for $A^\op$ are the opposites 
$(A_\lambda)^\op$
of the algebras $A_\lambda$ from before.
Letting $L_\lambda^\op(b) := L_\lambda(b)^\circledast$ 
for each $b \in \B_\lambda$, we obtain a full set of irreducible graded right $A_\lambda$-modules up to isomorphism and degree shift.
Then one can apply the general theory to this basis of $A^{\op}$ to obtain graded right $A$-modules
$P^\op(b)$, $\Delta^\op(b)$, $\bar\Delta^\op(b)$, $L^\op(b)$, $\bar\nabla^\op(b)$, $\nabla^\op(b)$ 
and $I^\op(b)$ indexed by $b \in \B$. 
By properties of adjunctions, we have that
\begin{align}
j_!^\lambda \circ ?^\circledast &\cong
?^\circledast \circ j_*^\lambda,&
j_*^\lambda \circ ?^\circledast &\cong
?^\circledast \circ j_!^\lambda.
\end{align}
Also duality obviously commutes with
the inclusion functor $i_{\geq \lambda}$.
It follows that $?^\circledast$ 
takes $\Delta^\op(b)$, $\bar\Delta^\op(b)$,  $\bar\nabla^\op(b)$ and $\nabla^\op(b)$ 
to
$\nabla(b)$, $\bar\nabla(b)$, $\bar\Delta(b)$ and $\Delta(b)$, respectively.
By \cref{irrclass}, we deduce that 
$L^\op(b)^\circledast \cong L(b)$, so
$I^\op(b)^\circledast \cong P(b)$ and
$P^\op(b)^\circledast \cong I(b)$.

In examples, it is often the case that 
$A$ admits a graded algebra 
anti-automorphism $\tau:A \rightarrow A$
fixing each $1_i\:(i \in \I)$.
We say that the graded triangular basis
admits a {\em duality} $\tau$ when this 
holds. If in addition $\tau$
can be chosen so that it takes each $\X(i,s)$ to $\Y(s,i)$ and each $\H(s,t)$ to $\H(t,s)$, we say that the graded triangular basis is {\em symmetric}.
Assuming that $A$ admits a duality $\tau$,
we can compose the functor
$?^\circledast:A\gmod \rightarrow \domg A$ 
from \cref{politicians} with 
restriction along $\tau$ to obtain a 
contravariant functor
\begin{equation}\label{taudual}
?^\taudual: A\gmod \rightarrow A \gmod.
\end{equation}
This restricts to a contravariant graded auto-equivalence on $A\gmodlfd$.
It is easy to see that $\tau$ descends to anti-automorphisms $\tau:A_{\geq \lambda}
\rightarrow A_{\geq \lambda}$
and
$\tau:A_\lambda\rightarrow A_\lambda$
for each $\lambda \in \Lambda$. 
Using this, we define
dualities $?^\taudual$ on $A_{\geq \lambda}\gmod$ and $A_\lambda\gmod$ too. 
If it happens that 
$L_\lambda(b)^\taudual \cong L_\lambda(b)$
for each $b \in \B_\lambda$, we get that
\begin{align}\label{selfduality}
\Delta(b)^\taudual &\cong \nabla(b),
&\bar\Delta(b)^\taudual&\cong \bar\nabla(b),&
L(b)^\taudual &\cong L(b),&
P(b)^\taudual &\cong I(b),&
I(b)^\taudual &\cong P(b).
\end{align}
It follows that
\begin{equation}
[\bar\Delta(a):L(b)]_q
= \overline{[\bar\nabla(a):L(b)]_q}
\end{equation}
for any $a,b \in \B$.

\section{Good filtrations}

Continue with $A$ being an algebra with a graded triangular basis.
The next important theorem is similar to \cite[Th.~5.28]{BS},
which treats the ungraded setting. A key difference in the graded case is that
the direct sums in the sections of the filtration  may be infinite.

\begin{theorem}\label{ohno}
Take any $b \in \B$ and let $\lambda := \dot b$.
Let $\lambda = \lambda_1,\dots,\lambda_n$
be $\{\mu \in \Lambda\:|\:\mu \leq \lambda\}$
ordered so that $\lambda_p < \lambda_q \Rightarrow p > q$.
\begin{enumerate}
\item
There exists a (non-unique) module $Q(b)\in \ob A\pgmod$ with a graded filtration
$$
Q(b) =  Q_0(b) \supset Q_1(b)\supseteq\cdots\supseteq Q_n(b) = 0$$ 
such that each
$Q_{r-1}(b) / Q_{r}(b)$ is
a (possibly infinite) direct sum of 
degree-shifted copies of standard modules 
$\Delta(a)$ for $a \in \B_{\lambda_r}$.
Moreover, the top section
$Q_0(b) / Q_1(b)$ is actually a finite direct sum
of these standard modules, with one of them being $\Delta(b)$.
\item
There exists a (non-unique) module
$J(b) \in \ob A\igmod$ with a graded filtration
$$0 = J_{0}(b) \subset J_1(b)\subseteq
\cdots\subseteq J_n(b)=J(b)$$ 
such that each
$J_{r}(b) / J_{r-1}(b)$ 
is a (possibly infinite) direct sum of 
degree-shifted copies of costandard modules
$\nabla(a)$ for $a \in \B_{\lambda_r}$.
Moreover, the bottom section
$J_1(b) / J_0(b)$ is actually a finite direct sum
of these costandard modules, with one of them being $\nabla(b)$.
\end{enumerate}
\end{theorem}

\begin{proof}
We just explain the proof of (1). Then (2) follows 
by applying $?^\circledast$
to the conclusion of
 (1) for $A^\op$ .
Pick $u \in \S_\lambda$ 
and $d \in \Z$ such that $\bar 1_u L_\lambda(b)_{d} \neq 0$. Equivalently, $P_\lambda(b)$ is a summand of 
$q^{d} A_\lambda \bar 1_u$.
We define $Q(b)$ to be the finitely generated projective graded left $A$-module
$q^{d} A 1_u$. It has basis
$xhy$ for $(x,h,y) \in \bigcup_{s,t \in \S_{\leq \lambda}} \X(s) \times \H(s,t) \times \Y(t,u)$.
Let $Q_r(b)$ be the subspace of $Q(b)$
spanned by all $xhy$ for $(x,h,y) \in \bigcup_{r< f\leq n} \bigcup_{s,t \in \S_{\lambda_f}}
\X(s) \times \H(s,t) \times \Y(t,u)$.

We show in this paragraph that $Q_r(b)$ is an $A$-submodule of $Q(b)$.
It suffices to see 
that $a xhy \in Q_r(b)$ for any $i \in \I$,
$r < f \leq n$,
$s,t \in \S_{\lambda_f},
a \in A 1_i$,
 and 
 $(x,h,y) \in 
\X(i,s) \times \H(s,t) \times \Y(t,u)$.
This follows by applying Lemma~\ref{parasites}
to get that
$a xhy$ is a linear combination of elements of the form
$x'h'y'$ for $x' \in \X(s') \times \H(s',t') \times \Y(t',u)$
and $s',t' \in \S_{\lambda_g}\:(g \geq f)$. Hence, we have constructed a filtration of $Q(b)$.

Consider some $1 \leq r \leq n$.
The module $Q_{r-1}(b) / Q_{r}(b)$
has basis given by the canonical images
of the vectors $xhy$ for
$(x,h,y) \in \bigcup_{s,t \in \S_{\lambda_r}}
\X(s) \times \H(s,t) \times \Y(t,u)$.
By
\cref{thisisabasis},
the vectors
$\bar x \otimes\bar h \bar y$
for
$(x,h,y) \in \bigcup_{s,t \in \S_{\lambda_r}}
\X(s) \times \H(s,t) \times \Y(t,u)$
give a basis for
$A_{\geq \lambda_j}\bar e_{\lambda_r} \otimes_{A_{\lambda_r}}
\bar e_{\lambda_r} 
A_{\geq \lambda_r} \bar 1_u$.
It follows that there is a degree-preserving isomorphism of graded vector spaces
$$
f:q^d A_{\geq \lambda_r}\bar e_{\lambda_r} \otimes_{A_{\lambda_r}}
\bar e_{\lambda_r} 
A_{\geq \lambda_r} \bar 1_u
\stackrel{\sim}{\rightarrow} Q_{r-1}(b) / Q_{r}(b),\qquad
\bar x\otimes \bar h \bar y\mapsto xhy + Q_{r}(b).
$$
This is actually an isomorphism of $A$-modules.
To see this, we take $s,t \in \S_{\lambda_r}$,
$a \in A 1_i$ and 
$(x,h,y) \in \X(i,s) \times \H(s,t) \times \Y(t)$ then apply \cref{parasites} again to
write $a xhy$
as a linear combination of basis elements
$x'h'y'$ for $(x',h',y') \in \bigcup_{
s',t' \in \S_{\geq \lambda_r}}
\X(i,s') \times \H(s',t') \times \Y(t',s)$.
It just remains to observe that when $s',t' \in \S_{> \lambda_r}$
both $\bar x' \otimes \bar h' \bar y'$
and $xhy + Q_r(b)$ are zero.

We have now proved that 
$Q_{r-1}(b) / Q_{r}(b) \cong
q^d A_{\geq \lambda_r}\bar e_{\lambda_r} \otimes_{A_{\lambda_r}}
\bar e_{\lambda_r} 
A_{\geq \lambda_r} \bar 1_u$
as graded left $A$-modules.
The basis implies that
$\bar e_{\lambda_r} A_{\geq \lambda_r}
\bar 1_u =
\bigoplus_{s \in \S_{\lambda_r}}
\bigoplus_{y \in \Y(s,u)}
A_{\lambda_r} \bar y$ 
as a graded left $A_{\lambda_r}$-module,
and $A_{\lambda_r} \bar y
\cong q^{\deg(y)} A_{\lambda_r} \bar 1_s$
for $s \in \S_{\lambda_r}$
and $y \in \Y(s,u)$.
We deduce that
$$
q^d A_{\geq \lambda_r}\bar e_{\lambda_r} \otimes_{A_{\lambda_r}}
\bar e_{\lambda_r} 
A_{\geq \lambda_r} \bar 1_u
\cong 
\bigoplus_{s \in \S_{\lambda_r}}
\bigoplus_{y \in \Y(s,u)}
q^{d+\deg(y)} j_!^{\lambda_r} A_{\lambda_r} \bar 1_s
$$
as graded left $A$-modules.
Since $A_{\lambda_r} \bar 1_s$ is a finitely generated projective, it is a finite direct sum of
degree-shifted copies of $P_{\lambda_r}(a)$ for $a \in \B_{\lambda_r}$.
Now our decomposition implies that $Q_{r-1}(b) / Q_{r}(b)$ is a (possibly infinite) direct sum of degree-shifted copies of $\Delta_{\lambda_r}(a)$
for $a \in \B_{\lambda_r}$.
In the case $r=1$, the argument shows further that
$Q_0(b) / Q_1(b)
\cong q^d j_!^\lambda A_\lambda \bar 1_u$,
which is a {\em finite} direct sum of 
degree-shifted standard modules since $A_\lambda \bar 1_u$ is a finitely generated projective, and it
contains $j_!^\lambda
P_\lambda(b) \cong \Delta(b)$
as a summand by the choice of $u$.
\end{proof}

\begin{corollary}\label{gym}
Suppose that $\lambda$ is minimal in $\Lambda$.
Then $\Delta(b) = P(b)$ and $\nabla(b) = I(b)$
for any $b \in \B_\lambda$.
\end{corollary}

\begin{proof}
We just prove the first statement.
 It suffices to show that $\Delta(b)$ is projective. This follows from \cref{ohno}(1): the filtration of 
$Q(b)$ constructed there has just one layer by the minimality of $\lambda$
so it shows that $\Delta(b)$ is a summand of the projective module $Q(b)$.
\end{proof}

 \cref{ohno} reveals that we are in a situation which 
 is similar in some respects to the semi-infinite fully stratified categories of \cite{BS}, and in other respects to the affine highest weight categories of \cite{AHW}. However, in
 \cite{BS}, there is no grading and the algebras $A_\lambda$ are assumed to be finite-dimensional, while in \cite{AHW} the graded algebra $A$ is assumed to be both unital 
 and Noetherian. In the examples of interest to us, the sections of the filtration constructed  in Theorem~\ref{ohno}
 usually involve infinite direct sums, so that 
 our indecomposable projectives 
 $P(b)\:(b \in \B)$ are seldom
 Noetherian. So we need to develop some new theory to proceed.
 
\begin{definition}\label{defup}
By a {\em $\Delta$-layer} (resp., a {\em $\bar\Delta$-layer}) of type $\lambda$,
we mean a graded $A$-module that is isomorphic to
$j^\lambda_! \bar V$
for 
a projective (resp., an arbitrary)
graded left $A_\lambda$-module $\bar V$
that is locally finite-dimensional and bounded below.
We say that $V \in \ob A\gmod$ has a 
{\em $\Delta$-flag} (resp., a {\em $\bar\Delta$-flag}) if for some $n \geq 0$
there is a graded
filtration 
$$
0 = V_0 \subset V_1 \subset\cdots \subset V_n = V
$$
and distinct weights $\lambda_1,\dots,\lambda_n \in \Lambda$
such that $V_{r} / V_{r-1}$ is a $\Delta$-layer (resp., a $\bar\Delta$-layer) of type
$\lambda_r$ for each $r = 1,\dots,n$.
\end{definition}

\begin{definition}\label{defdown}
By a {\em $\nabla$-layer} (resp., a {\em $\bar\nabla$-layer}) of type $\lambda$,
we mean a graded $A$-module
that is isomorphic to $j^\lambda_* \bar V$
for 
an injective (resp., an arbitrary)
graded left $A_\lambda$-module $\bar V$ that is locally finite-dimensional and bounded above.
We say that $V \in \ob A\gmod$ has a
{\em $\nabla$-flag} (resp., a {\em $\bar\nabla$-flag}) if for some $n \geq 0$
there is a graded filtration
$$
V = V_0 \supset V_1 \supset \cdots\supset V_n = 0
$$
and distinct weights $\lambda_1,\dots,\lambda_n \in \Lambda$
such that $V_{r-1} / V_{r}$ is a $\nabla$-layer (resp., a $\bar\nabla$-layer) of type
$\lambda_r$ for each $r= 1,\dots,n$.
\end{definition}

\begin{remark}\label{jam}
Our $\Delta$-layers of type $\lambda$
can be defined equivalently as
modules of the form $\bigoplus_{b \in \B_\lambda}
\Delta(b)^{\oplus f_b}$ for $f_b \in \N\lround q\rround$.
Similarly, 
$\nabla$-layers of type $\lambda$ are
modules of the form $\bigoplus_{b \in \B_\lambda}
\nabla(b)^{\oplus f_b}$ for $f_b \in \N\lround q^{-1}\rround$.
Using this, it follows that the module $Q(b)$ in \cref{ohno}(1) has a $\Delta$-flag, and the module $J(b)$ in \cref{ohno}(2)
has a $\nabla$-flag.
\end{remark}

The full subcategory of $A\gmod$
consisting of modules with $\Delta$-flags
(resp., $\bar\Delta$-flags, $\nabla$-flags, $\bar\nabla$-flags) will be denoted
$A\gmoddelta$ (resp., $A\gmoddeltabar$,
$A\gmodnabla$, $A\gmodnablabar$).
Evidently, 
\cref{defup,defdown}
are dual to each other. Note also by \cref{xmas}
that modules in $A\gmoddelta$ or 
$A\gmoddeltabar$ are locally finite-dimensional and bounded below, and modules in $A\gmodnabla$ 
or $A\gmodnablabar$ are locally finite-dimensional and bounded above.
Consequently, all subsequent results about $\Delta$-
or $\bar\Delta$-flags have dual 
formulations involving $\nabla$- or $\bar \nabla$-flags.

Noting that $\Delta$-layers are $\bar\Delta$-layers,
  $A\gmoddelta$ is a subcategory of
$A\gmoddeltabar$.
The next result allows sections
in $\bar\Delta$-flags, hence, in $\Delta$-flags, to be reordered so that the biggest weights are at the top.

\begin{lemma}\label{archers}
If $V$ is a $\bar\Delta$-layer
of type $\lambda$ and $W$ is a $\bar\Delta$-layer of type $\mu$
for $\lambda \ngeq \mu$ then $\Ext^1_A(V,W) = 0$.
\end{lemma}

\begin{proof}
We have that $V = j^\lambda_! \bar V$
for some $\bar V \in \ob A_\lambda\gmod$ that is locally finite-dimensional and bounded below.
Take the start of a projective
resolution of $\bar V \in \ob A_\lambda\gmod$:
$\bar P_1 \rightarrow \bar P_0 \rightarrow \bar V \rightarrow 0$.
Then apply the exact functor $j^\lambda_!$
to deduce that there is an exact sequence
$P_1 \stackrel{f}{\rightarrow} P_0 \rightarrow V \rightarrow 0
$
in $A\gmod$ such that $P_0, P_1$ both
(possibly infinite) are direct sums of degree-shifted modules 
of the form $\Delta(b)$ for $b \in \B_\lambda$.
Next, we apply \cref{ohno}(1) to construct 
projective resolutions $\cdots \rightarrow P_{0,1}\rightarrow P_{0,0} \rightarrow P_0 \rightarrow 0$ and
$\cdots \rightarrow P_{1,0} \rightarrow P_1 \rightarrow 0$ 
such that $P_{0,0}$ and $P_{1,0}$ are (possibly infinite)
direct sum of degree-shifted copies of
$Q(b)$ for $b \in \B_\lambda$ and 
$P_{0,1}$ is a 
(possibly infinite)
direct sum of degree-shifted copies of
$Q(a)$ for 
$a \in \B_{\leq \lambda}$.
By the nature of the filtration
from \cref{ohno}(1),
it follows that all irreducible quotients of $P_{0,0}, P_{0,1}$ and $P_{1,0}$ are degree-shifted copies of $L(a)$
for $a \in \B_{\leq \lambda}$.
We lift $f:P_1 \rightarrow P_0$ to these resolutions then take the total complex to
obtain the beginning of a projective resolution
$$
P_{1,0} \oplus P_{0,1} \rightarrow P_{0,0} \rightarrow V \rightarrow 0
$$
of $V$. Then apply $\Hom_{A}(-,W)$
and take homology to deduce that $\Ext^1_A (V, W)$ is 
a subquotient of 
$\Hom_A(P_{1,0} \oplus P_{0,1}, W)$.
But the module $W$ has lowest weight $\mu$,
while all non-zero quotients of $P_{1,0}\oplus P_{0,1}$ 
have a weight that is $\leq \lambda$.
Since $\lambda\ngeq \mu$ this means that
$\Hom_A(P_{1,0} \oplus P_{0,1}, W) = 0$, so 
$\Ext^1_A (V, W) = 0$ too.
\end{proof}

\begin{corollary}\label{archersc}
If $V$ is a $\Delta$-layer
of type $\lambda$ and $W$ is a $\Delta$-layer of type $\mu$
for $\lambda \ngeq \mu$ then $\Ext^1_A(V,W) = 0$.
\end{corollary}

\begin{proof}
This follows immediately from the lemma 
since $\Delta$-layers are $\bar\Delta$-layers.
\end{proof}

\begin{remark}\label{archersr}
In fact, the following 
slightly stronger statement than
\cref{archersc} is true:
if $V$ is a $\Delta$-layer
of type $\lambda$ and $W$ is a $\Delta$-layer of type $\mu$
for $\lambda \ng \mu$ then $\Ext^1_A(V,W) = 0$.
We are not in a position to be able to prove this yet, but it follows from \cref{bggproj,eat}
since they imply that $P(b)$ has a $\Delta$-flag with
top section 
$\Delta(b)$ and other sections
that are $\Delta$-layers of type $\mu$
for $\mu< \dot b$.
In view of this, if $V$ is a $\Delta$-layer
of type $\lambda$, we can construct a projective resolution $\cdots \rightarrow 
P_1 \rightarrow P_0 \rightarrow V$
such that $P_0$ is a direct sum of degree-shifted
copies of $P(b)\:(b \in \B_\lambda)$
and $P_1$ is a direct sum of degree-shifted copies of $P(c)\:(c \in \B_{<\lambda}$).
We deduce that $\Ext^1_A(V,W) = 0$
for a $\Delta$-layer $W$ of type $\mu \nl\lambda$
since we have that $\Hom_A(P_1, W) = 0$ like at the end of the proof of \cref{archers}.
\end{remark}

Later on, the next lemma (which is really two lemmas since there are two cases in the statement) will be used at a crucial point in an inductive argument; see the proof of \cref{citizens}.

\begin{lemma}\label{inductionbase}
Suppose that 
$\lambda \in \Lambda$ is minimal
and $V \in A\gmod$ has the following properties:
\begin{enumerate}
\item
$V$ is locally finite-dimensional and bounded below;
\item
$V = A e_\lambda V$;
\item
$\Ext^1_A(V, \bar\nabla(b)) = 0$ 
(resp.,
$\Ext^1_A(V, \nabla(b)) = 0$) for all
$b \in \B$.
\end{enumerate}
Then $V$ is a $\Delta$-layer (resp., a $\bar\Delta$-layer) of type $\lambda$.
\end{lemma}

\begin{proof}
The assumption (2) plus
\cref{green}(1) implies that all weights of $V$ are $\geq \lambda$, hence, $V$ is an $A_{\geq \lambda}$-module.
The counit of adjunction
gives a homomorphism $\eps^\lambda_V:j^\lambda_! j^\lambda V
\rightarrow V$. This becomes an isomorphism when we
apply $j^\lambda$, so the $\lambda$-weight space of $\coker f$ is zero. But by (2) we know that every quotient of $V$ is generated by its $\lambda$-weight space,
so this implies that $\coker \eps^\lambda_V = 0$.
Thus, we have proved that $\eps^\lambda_V$ is surjective.

Let $K := \ker \eps^\lambda_V$ so that there is a 
short exact sequence
$0 \rightarrow K \rightarrow j^\lambda_! j^\lambda V \rightarrow V \rightarrow 0$.
Let $Y := \bar\nabla(b)$ (resp., $\nabla(b)$) 
for some $b \in \B$.
We claim that $\Hom_A(K, Y) = 0$.
To see this, we apply $\Hom_A(-,Y)$ to the short exact sequence and use (3) to get another short exact sequence
$$
0 \longrightarrow \Hom_A(V,Y)
\longrightarrow 
\Hom_A(j^\lambda_! j^\lambda V, Y)
\longrightarrow \Hom_A(K,Y) \rightarrow 0
$$
If $\dot b \ngeq\lambda$ then $\Hom_A(K, Y) = 0$
because $\dot b$ is a weight of $\soc Y$ but it is not a weight of $K$.
If $\dot b > \lambda$ then $\Hom_A(j^\lambda_! j^\lambda V, Y) \cong \Hom_{A_\lambda}(j^\lambda V, j^\lambda Y)$,
which is zero
since $j^\lambda Y = 0$. Hence, 
$\Hom_A(K,Y) = 0$ in this case.
If $\dot b = \lambda$, both 
$\Hom_A(j^\lambda_! j^\lambda V, Y)$
and
$\Hom_A(V,Y)$
are isomorphic to $\Hom_{A_\lambda}(j^\lambda V, j^\lambda Y)$; in the second case this follows because 
$Y = j^\lambda_* j^\lambda Y$. Hence, they have  the same graded dimensions. It follows that the first map in the displayed short exact sequence is an isomorphism. Again, this gives that $\Hom_A(K, Y) = 0$, and the claim is proved.

By the claim, we deduce in either case that 
$\Hom_A(K, \bar\nabla(b)) = 0$ for all $b\in \B$.
So $K=0$ thanks to \cref{nakcor}.
Now we have proved that
$V \cong j^\lambda_! j^\lambda V$.
This already shows that $V$ is a $\bar\Delta$-layer.
To complete the proof, we need to show
that $j^\lambda V$ is projective in $A_\lambda\gmod$
in the case that
$\Ext^1_A(V, \bar\nabla(b)) = 0$ for all $b \in \B$.
The functor $j^\lambda_*$ is right adjoint to an exact functor, so it takes injective graded $A_\lambda$-modules to injective graded $A_{\geq \lambda}$-modules.
It is also exact by \cref{xmas}.
Since $\Hom_{A_\lambda}(j^\lambda V, -)\cong 
\Hom_{A_{\geq \lambda}}(V, -) \circ j^\lambda_*$,
a standard degenerate Grothendieck spectral sequence argument gives that
\begin{equation}\label{ssarg}
\Ext^n_{A_\lambda}(j^\lambda V, -)
\cong \Ext^n_{A_{\geq \lambda}}(V, j^\lambda_* -)
\end{equation} 
for any $V \in A_{\geq \lambda} \gmod$ and $n \geq 0$.
Using this, we deduce that
$$
\Ext^1_{A_\lambda}(j^\lambda V, L_\lambda(b))
\cong
\Ext^1_{A_{\geq \lambda}}(V, j^\lambda_* L_\lambda(b))=
\Ext^1_A(V, \bar\nabla(b))=0
$$
for all $b \in \B_\lambda$.
This implies that $j^\lambda V$ is projective
according to \cref{noname}.
\end{proof}

\section{Truncation to upper sets}\label{tuper}

In this section, 
we assume that $\hat\Lambda$ is an upper set in $\Lambda$. Let $\hat\S := \{s \in \S\:|\:\dot s \in \hat\Lambda\}$,
$\hat\B := \{b \in \B\:|\:\dot b \in \hat \Lambda\}$
and 
$\check\Lambda := \Lambda - \hat\Lambda$, $\check\S := \S - \hat\S$,
$\check\B := \B-\hat\B$.
Let $\hat A$ be the quotient of $A$ by the
two-sided ideal generated by the idempotents
$e_\lambda\:(\lambda \in \check\Lambda)$.
Let $$
i:\hat A\gmod \rightarrow A\gmod
$$ 
be the canonical
inclusion functor, which is part of the adjoint triple
$(i^*,i,i^!)$ discussed in \cref{func1,func2}.
So $i^* = \hat A \otimes_A-$
and $i^! = \bigoplus_{i \in \I} \Hom_{A}(\hat A \bar 1_i,-)$.

By \cref{climbing}, $\hat A = \bigoplus_{i,j \in \I} 1_i \hat A 1_j$ has a graded triangular basis
with special idempotents indexed by 
$\hat \S \subseteq \I$, weight poset
$\hat\Lambda$, and bases arising from
the sets $\hat \X(i,s), \hat \H(s,t),
\hat \Y(s,j)$ that are the canonical images
of $\X(i,s), \H(s,t), \Y(t,j)$
for $i,j \in \I, s,t \in \hat\S$.
Using the decoration ``$\wedge$" in other notation related to $\hat A$ in the obvious way,
the algebras $\hat A_{\geq \lambda}\:(\lambda \in \hat\Lambda)$ 
are naturally identified with the algebras $A_{\geq \lambda}$. So we also have that $\hat A_\lambda = A_\lambda$,
and the adjoint triple $(\hat \jmath^\lambda_!, \hat \jmath, \hat \jmath^\lambda_*)$ defined for $\hat A$
is just the same triple of functors  $(j^\lambda_!,j^\lambda, j^\lambda_*)$ as for $A$,
still assuming that $\lambda \in \hat\Lambda$.

The various modules
for $\hat A$ arising from the triangular basis are 
\begin{align}\label{standardsagain}
\hat\Delta(b)& := 
j^\lambda_! P_\lambda(b),&
\widehat{\bar\Delta}(b) &:= 
j^\lambda_! L_\lambda(b),&
\widehat{\bar\nabla}(b)& := 
j^\lambda_* L_\lambda(b),&
\hat\nabla(b) &:= 
j^\lambda_* I_\lambda(b)
\end{align}
for $b \in \hat \B$ and $\lambda := \dot b$.
Then the modules
$\hat L(b) := \head \hat\Delta(b) = \soc \hat\nabla(b)$ for $b \in \hat \B$ 
give a complete set of irreducible graded
left $\hat A$-modules up to isomorphism and degree shift. We denote a projective cover
and an injective hull of $\hat L(b)$
by $\hat P(b)$ and $\hat I(b)$, respectively.

\begin{lemma}\label{workout}
For $b \in \hat\B$, we have that 
$\hat \Delta(b) = \Delta(b)$,
$\widehat{\bar\Delta}(b) = \bar\Delta(b)$,
$\hat L(b) = L(b)$,
$\widehat{\bar \nabla}(b) = \bar\nabla(b)$ and
$\hat \nabla(b) = \nabla(b)$.
Also
$i^* P(b) \cong \hat P(b)$, 
$i^! I(b) \cong \hat I(b)$
if $b \in \hat\B$, and
$i^* P(b) = i^* \Delta(b) = i^* \bar\Delta(b)
=i^* L(b) = i^! L(b)= i^! \bar\nabla(b) = i^! \nabla(b) = i^! I(b)= 0$ if $b \in \check \B$.
\end{lemma}

\begin{proof}
We have that $i_{\geq \lambda} = i \circ \hat 
\imath_{\geq \lambda}$, which implies the
assertions about
$\Delta(b)$,
$\bar\Delta(b)$,
$\bar\nabla(b)$ and
$\nabla(b)$ for $b \in \hat \B$. Clearly we also have that $\hat L(b) = L(b)$ since it is the irreducible head of $ \hat \Delta(b) = \Delta(b)$.
To see that $i^* P(b) \cong \hat P(b)$,
note that $i^*$ is left adjoint to an exact functor,
so $i^* P(b)$ is a finitely generated
projective for any $b \in \B$. It remains to
observe for $c \in \hat\B$ that $\Hom_{\hat A}(i^* P(b), L(c)) \cong
\Hom_A(P(b), i L(c))$, which is zero unless $c=b$.
This gives that $i^* P(b) \cong \hat P(b)$ for $b \in \hat \B$ and it is zero otherwise.
A similar argument proves the assertion about $i^! I(b)$.
Everything else follows by right exactness of $i^*$
and left exactness of $i^!$.
\end{proof}

\begin{lemma}\label{previousl}
For 
$V \in \ob A\gmoddelta$ and $i \in \I$,
we have that
$\Tor^A_m(1_i \hat A, V) = 0$
for all $m \geq 1$.
\end{lemma}

\begin{proof}
In the next paragraph, we show that $\Tor_m^A(1_i \hat A, \Delta(b)) = 0$ for 
$b \in \B$ and $m \geq 1$. 
To deduce the lemma from this, 
$\Delta$-layers are (possibly infinite) direct sums
of standard modules as noted in \cref{jam},
so we get that $\Tor_m^A(1_i \hat A, V) = 0$
for all $\Delta$-layers $V$ and $m \geq 1$.
Then one deduces the result for all $V$ with a $\Delta$-flag by induction on the length of the filtration.

Take $b \in \B$ and let $Q$ be the module $Q(b)$ from \cref{ohno}(1).
There is a short exact sequence
$0 \rightarrow K \rightarrow Q \rightarrow \Delta(b)
\rightarrow 0$ with $K$ and $Q$ having a $\Delta$-flags. Applying $1_i \hat A \otimes_A-$ gives
the long exact sequence
$$
0 \rightarrow \Tor_1^A(1_i \hat A, \Delta(b))  \longrightarrow 1_i \hat A \otimes_A K
 \longrightarrow 1_i \hat A \otimes_A Q
 \longrightarrow 1_i \hat A \otimes_A \Delta(b)
  \rightarrow 0
  $$
  and isomorphisms
  $\Tor_{m+1}^A(1_i \hat A, \Delta(b))\cong
  \Tor_m^A(1_i \hat A, K)$ for $m \geq 1$.
  Now we use \cref{archersc} to see that the $\Delta$-flags of $K$ and $Q$ can be ordered to obtain short exact sequences
  $0 \rightarrow K^- \rightarrow K \rightarrow K^+\rightarrow 0$ and $0 \rightarrow Q^- \rightarrow Q \rightarrow Q^+\rightarrow 0$ so that $K^-$ and $Q^-$ (resp., $K^+$ and $Q^+$) have a $\Delta$-flags
with all sections being $\Delta$-layers of types in $\check\Lambda$ (resp., $\hat\Lambda$).
It is then clear that $1_i \hat A \otimes_A K = 1_i K^+$
and $1_i \hat A \otimes_A Q = 1_i Q^+$,
since $K^+$ and $Q^+$ are the largest quotients of $K$ and $Q$ with all weights in $\hat\Lambda$.
If $b \notin \hat\B$ then $K^+ = 0$, so we have that
$\Tor_1^A(1_i \hat A, \Delta(b)) = 0$ at once.
If $b \in \hat\B$ then there is a short exact sequence
$0 \rightarrow 1_i K^+ \rightarrow 1_i Q^+ \rightarrow 1_i \Delta(b)
\rightarrow 0$. This is just the same
as the rightmost terms  $1_i \hat A \otimes_A K \rightarrow
1_i \hat A \otimes_A Q \rightarrow 1_i \hat A \otimes_A \Delta(b) \rightarrow 0$ of the long exact sequence displayed above. So again we deduce that $\Tor_1^A(1_i \hat A, \Delta(b)) = 0$.
So now we have shown that $\Tor_1^A(1_i \hat A, \Delta(b)) = 0$ for all $b \in \B$.
For $K$ as before, it follows that $\Tor_1^A(1_i \hat A, K) = 0$,
hence, we get that $\Tor_2^A(1_i \hat A, \Delta(b)) = 0$
for all $b \in \B$.
Further degree shifting like this 
gives the conclusion in general.
\end{proof}

\begin{corollary}\label{previousc}
The functor $i^* = \hat A \otimes_A -$
takes short exact sequences of modules with $\Delta$-flags
 to short exact sequences of modules with $\Delta$-flags.
Similarly, the functor $i^! = \bigoplus_{i \in \I} \Hom_A(\hat A 1_i,-)$
takes short exact sequences of modules with $\nabla$-flags
 to short exact sequences of modules with $\nabla$-flags.
\end{corollary}

\begin{proof}
The lemma shows that $i^*$ takes short exact sequences
of modules with $\Delta$-flags to short exact sequences.
Hence, to prove that $i^*$ takes
modules with $\Delta$-flags to $\Delta$-flags, 
it suffices to check that $i^*$ takes $\Delta$-layers
to $\Delta$-layers. This follows from \cref{workout}
since $i^*$ commutes with direct sums.
This proves the first statement.
Then the second statement follows by duality, i.e.,
we apply $?^\circledast$ then the analog of the first statement for the opposite algebras, then apply $?^\circledast$ again.
\end{proof}

\begin{lemma}\label{mercedes}
For $V \in \ob A\gmoddelta$ and $W \in \ob\hat A\gmod$,
we have that
$\Ext^n_{A}(V, i W)
\cong \Ext^n_{\hat A}(i^* V, W)$
for all $n \geq 0$.
Similarly, for $V \in \ob \hat A\gmod$ and $W \in \ob A\gmodnabla$,
we have that
 $\Ext^n_{A}(i V, W)
\cong \Ext^n_{\hat A}(V, i^! W)$
for all $n \geq 0$.
\end{lemma}

\begin{proof}
To prove the first statement, take $W \in \ob\hat A\gmod$.
The adjunction gives an isomorphism of functors
$\Hom_{\hat A}(-, W) \circ i^* \cong \Hom_{A}(-, iW)$. Also the functor $i^* = \hat A \otimes_A-$ 
takes projectives to projectives as it is left adjoint to an exact functor. By a Grothendieck spectral sequence argument, it follows that $\Ext^n_{\hat A}(i^* V, W)
\cong \Ext^n_A(V, i W)$ for all $n \geq 0$ and $V$ 
such that $\Tor^A_m(\hat A, V) = 0$
for all $m \geq 1$. It remains to apply \cref{previousl}.

The second statement follows from the first statement by 
duality. This is a bit more complicated than 
it sounds, so we go through the details.
We show equivalently that
 $\Ext^n_{A}\left(i V, W^\circledast\right)
\cong \Ext^n_{\hat A}\left(V, i^! (W^\circledast)\right)$
for $V \in \ob\hat A\gmod$ and $W \in \ob \domg A$
such that $W^\circledast$ has a $\nabla$-flag
(equivalently, $W$ has a $\Delta^\op$-flag).
We have that $i \circ ?^\circledast \cong ?^\circledast \circ i$ viewed as covariant 
functors from $(\domg \hat A)^\op$ to $A\gmod$. 
Taking left adjoints
gives that $?^\circledast \circ i^* \cong
i^! \circ ?^\circledast$
viewed as 
functors from $A \gmod$ to $(\domg \hat A)^\op$.
So $$
\Ext^n_{\hat A}\left(V, i^! (W^\circledast)\right)
\cong \Ext^n_{\hat A}\left(V, (i^* W)^\circledast\right)
\stackrel{\cref{hands}}{\cong} 
\Ext^n_{\hat A}\left(i^* W, V^\circledast\right).
$$
Then we apply the analog of the first statement for the opposite algebras
to see that
$$
\Ext^n_{\hat A}\left(i^* W, V^\circledast\right)
\cong \Ext^n_A\left(W, i (V^\circledast)\right)
\cong \Ext^n_A\left(W, (i V)^\circledast\right)
\stackrel{\cref{hands}}{\cong} \Ext^n_A\left(i V, W^\circledast\right),
$$
as required.
\end{proof}

Now we can prove the hallmark property
of highest weight categories and their generalizations:

\begin{theorem}\label{extvanishingt}
If $V \in \ob A\gmoddelta$ and
$W \in \ob A\gmodnablabar$,
or if
$V \in \ob A\gmoddeltabar$ and
$W \in \ob A\gmodnabla$,
we have that $\Ext^n_A(V,W) = 0$ for all $n \geq 1$.
\end{theorem}

\begin{proof}
We prove this assuming $V \in \ob A\gmoddelta$ and
$W \in \ob A\gmodnablabar$; the result in the other case then follows by duality.
The proof reduces easily to the case that
$V$ is a single $\Delta$-layer
and $W= j^\lambda_! \bar W$ is a single $\bar\nabla$-layer
of type $\lambda$.
By \cref{jam}, 
$V$ is a (possibly infinite) direct sum
of degree-shifted standard modules, and the proof reduces further to checking that
$\Ext^n_A(\Delta(b), j^\lambda_! \bar W) = 0$ for all $b \in \B$ and $n \geq 1$.
By \cref{mercedes}, we have that
$$
\Ext^n_A(\Delta(b), j^\lambda_! \bar W) \cong \Ext^n_{A_{\geq \lambda}}(
i_{\geq \lambda}^* \Delta(b), j^\lambda_! \bar W).
$$
If $\dot b \ngeq \lambda$ then
$i_{\geq \lambda}^* \Delta(b) = 0$ and the conclusion follows at once. If $\dot b \geq \lambda$ then
we are in the same situation as \cref{ssarg},
and applying that isomorphism gives that
$\Ext^n_{A_{\geq \lambda}}(\Delta(b), j^\lambda_* \bar W)
\cong \Ext^n_{A_\lambda}(j^\lambda \Delta(b), \bar W)$.
This is zero for $n \geq 1$ as required
since $j^\lambda \Delta(b) \cong 
P_\lambda(b)$ is projective in $A_\lambda\gmod$
if $\dot b = \lambda$, and $j^\lambda \Delta(b)=0$ otherwise.
\end{proof}

\section{BGG reciprocity}

Using \cref{extvanishingt},
we can make sense of multiplicities
in $\Delta$- and  $\bar\Delta$-flags.
First, for any $V \in A\gmod$, we define
the {\em $\Delta$-} and  {\em $\bar\Delta$-supports}
of $V$:
\begin{align}\label{supp1}
\supp_\Delta(V) &:= \big\{\dot b\:\big|\:b \in \B\text{ such that }
\Hom_A(V, \bar\nabla(b))\neq 0\big\},\\\label{supp2}
\supp_{\bar\Delta}(V) &:= \big\{\dot b\:\big|\:
b \in \B\text{ such that }\Hom_A(V, \nabla(b))\neq 0\big\}.
\end{align}
Since $\bar\nabla(b) \hookrightarrow \nabla(b)$, we have that
$\supp_\Delta(V) \subseteq\supp_{\bar\Delta}(V)$.
When $A$ is not unital, i.e., infinitely many of the
$e_\lambda\:(\lambda\in\Lambda)$ are non-zero, 
these sets could
be infinite, but they are always finite
if $V$ is finitely generated:

\begin{lemma}\label{supports}
If $V$ is a finitely generated graded left $A$-module
then $\supp_\Delta(V)$ and $\supp_{\bar\Delta}(V)$
are finite.
\end{lemma}

\begin{proof}
Suppose that $V$ is
generated by finitely many weight vectors.
Let $\lambda_1,\dots,\lambda_n$ be their weights.
Then $\Hom_A(V, \nabla(b)) = 0$ unless
one of $\lambda_1,\dots,\lambda_n$ is a weight of $\nabla(b)$. But this implies that $b \in \B_{\leq \lambda_1} \cup\cdots\cup\B_{\leq\lambda_r}$, which is finite.
\end{proof}

For $V$ with a $\Delta$-flag, we define
\begin{equation}\label{seasonalmuffin}
(V:\Delta(b))_q := \overline{\dim_q \Hom_A(V, \bar\nabla(b))} \in \N\lround q \rround.
\end{equation}
This is non-zero if and only if $b \in \supp_\Delta(V)$.
If $0 = V_0 \subseteq \cdots \subseteq V_n = V$
is a $\Delta$-flag,
the section $V_r / V_{r-1}$ being a $\Delta$-layer of type
$\lambda_r$,
we have that
\begin{equation}
(V:\Delta(b))_q = \sum_{r=1}^n (V_r / V_{r-1}:\Delta(b))_q.
\end{equation}
This follows from \cref{extvanishingt}.
Moreover, \cref{upup} implies that
\begin{equation}
V_r / V_{r-1} \cong \bigoplus_{b \in \B_{\lambda_r}} \Delta(b)^{\oplus (V_r / V_{r-1}:\Delta(b))_q}.
\end{equation}
Thus, $(V:\Delta(b))_q$ counts the graded multiplicity
of $\Delta(b)$ as a summand of the layers of the
$\Delta$-flag as one would expect.
Instead, if $V$ has a $\bar\Delta$-flag, we set
\begin{equation}\label{tea}
(V:\bar\Delta(b))_q := \overline{\dim_q \Hom_A(V, \nabla(b))_q}\in \N\lround q \rround,
\end{equation}
which is non-zero if and only if $b \in \supp_{\bar\Delta}(V)$.
Again, we have that 
\begin{equation}
(V:\bar\Delta(b))_q = \sum_{r=1}^n (V_r / V_{r-1}:\bar\Delta(b))_q
\end{equation}
if $0 = V_0 \subseteq \cdots \subseteq V_n = V$
is a $\bar\Delta$-flag; now this follows by
\cref{extvanishingt}.
So $(V:\bar\Delta(b))_q$ computes
the sum of the graded multiplicities of $\bar\Delta(b)$ in each of the $\bar\Delta$-layers,
with the understanding that for a 
single $\bar\Delta$-layer $W \cong j^\lambda_! \bar W$ of type $\lambda$ and $b \in \B_\lambda$
we have that
\begin{equation}\label{wednesdaynight}
(W:\bar \Delta(b))_q = \overline{\dim_q \Hom_A(W, \nabla(b))}
= \overline{\dim_q \Hom_{A_\lambda}(\bar W, I_\lambda(b))} = [\bar W: L_\lambda(b)]_q.
\end{equation}
For example, every $\Delta(a)$ has a $\bar\Delta$-flag,
and we have that
\begin{equation}\label{thursdaynight}
(\Delta(a):\bar\Delta(b))_q = 
\left\{
\begin{array}{ll}
[P_\lambda(a):L_\lambda(b)]_q&\text{if $a,b \in \B_\lambda$
for some $\lambda \in \Lambda$}\\
0&\text{if $\dot a \neq \dot b$.}
\end{array}\right.
\end{equation}

\begin{lemma}\label{fridaynight}
If $V$ has a $\Delta$-flag then
$(V:\bar\Delta(b))_q = \sum_{a \in \B}
(V:\Delta(a))_q(\Delta(a):\bar\Delta(b))_q$.
\end{lemma}

\begin{proof}
It suffices to prove this when $V$ is a $\Delta$-layer of type $\lambda$,
so
 $V \cong \bigoplus_{a \in \B_\lambda} \Delta(a)^{\oplus (V:\Delta(a))_q}$.
 Then $$
 (V:\bar\Delta(b))_q
=[j^\lambda V: L_\lambda(b)]_q  = \sum_{a \in \B_\lambda} (V:\Delta(a))_q [P_\lambda(a):L_\lambda(b)]_q
  = \sum_{a \in \B_\lambda} (V:\Delta(a))_q (\Delta(a):\bar\Delta(b))_q.
  $$
  Here, we used \cref{wednesdaynight,thursdaynight}.
\end{proof}

\begin{theorem}[Homological criteria for 
$\Delta$- and $\bar\Delta$-flags]\label{citizens}
Assume that $V \in \ob A\gmod$ is locally finite-dimensional and bounded below.
\begin{itemize}
\item[(1)]
The following are equivalent:
\begin{enumerate}
\item[(a)] $V$ has a $\Delta$-flag;
\item[(b)] 
$|\supp_\Delta(V)| < \infty$ and
$\Ext^1_A(V, \bar\nabla(b)) = 0$ for all $b \in \B$;
\item[(c)] 
$|\supp_\Delta(V)| < \infty$ and
$\Ext^n_A(V, \bar\nabla(b)) = 0$ for all $b \in \B$ and $n \geq 1$;
\end{enumerate}
\item[(2)]
The following are equivalent:
\begin{enumerate}
\item[(a)] $V$ has a $\bar\Delta$-flag;
\item[(b)] 
$|\supp_{\bar\Delta}(V)| < \infty$ and
$\Ext^1_A(V, \nabla(b)) = 0$ for all $b \in \B$;
\item[(c)] 
$|\supp_{\bar\Delta}(V)| < \infty$ and
$\Ext^n_A(V, \nabla(b)) = 0$ for all $b \in \B$ and $n \geq 1$.
\end{enumerate}
\end{itemize}
\end{theorem}

\begin{proof}
(1)
Clearly (c)$\Rightarrow$(b).
Also
(a)$\Rightarrow$(c) 
by \cref{extvanishingt}.
It remains to prove that (b)$\Rightarrow$(a).
Suppose that (b) holds.
We show that $V$ has a $\Delta$-flag
by induction on the size of the
support $\supp_\Delta(V)$.
If $\supp_\Delta(V) = \varnothing$ then we have that $V = 0$ by \cref{nakcor}, and the conclusion is clear.
Now assume that $\supp_\Delta(V)$ is non-empty and pick
a maximal element $\lambda$.
Let $W := i_{\geq \lambda}^* V$.
We are going to apply \cref{inductionbase} (with
$A$ replaced by $A_{\geq\lambda}$, $\Lambda$
replaced by the upper set generated by $\lambda$
and $\B$ replaced by $\B_{\geq \lambda}$)
to show that $W$ is a $\Delta$-layer of type $\lambda$; it is important to note here that ``$\Delta$-layer of type $\lambda$" means the same thing for $A_{\geq\lambda}\gmod$ as 
it does for $A\gmod$ because 
$\hat\jmath^\lambda_! = j^\lambda_!$, notation as in 
\cref{tuper}.
Since $W$ is a quotient of $V$, the choice of $\lambda$ implies that
$\Hom_A(W, \bar\nabla(b)) = 0$
unless $\dot b = \lambda$.
Since $W / A e_\lambda W$ does not have $\lambda$ as a weight, 
we deduce that
$\Hom_A(W / A e_\lambda W , \bar\nabla(b)) = 0$ for all $b \in \B$. Applying \cref{nakcor} again, it follows that $W = A e_\lambda W = A_{\geq \lambda} e_\lambda W$. Thus $W$ satisfies property (2) from \cref{inductionbase}.
Also, $W$ is finitely generated, so it satisfies property (1).
To show that it satisfies property (3) too,
 let $K$ be the kernel of the quotient map $V \twoheadrightarrow W$ and take any $b \in \B$. Applying
$\Hom_A(-,\bar\nabla(b))$ to the short exact sequence $0 \rightarrow K \rightarrow V \rightarrow W \rightarrow 0$ gives the long exact sequence
$$
0 \longrightarrow \Hom_A(W, \bar\nabla(b))
\longrightarrow \Hom_A(V, \bar\nabla(b))
\longrightarrow \Hom_A(K, \bar\nabla(b))
\longrightarrow \Ext^1_A(W, \bar\nabla(b))
\longrightarrow 0,
$$
plus an isomorphism $\Ext^1_A(K, \bar\nabla(b))\cong \Ext^2_A(W, \bar\nabla(b))$.
Now suppose that $b \in \B_{\geq \lambda}$, so that
all weights of $\bar\nabla(b)$ are $\geq \lambda$ too.
By the definition of $W$, $K$ does not have a proper
quotient whose weights are all $\geq \lambda$,
so $\Hom_A(K,\bar\nabla(b)) = 0$.
We deduce that $\Ext^1_A(W, \bar\nabla(b))= \Ext^1_{A_{\geq\lambda}}(W, \bar\nabla(b)) = 0$ 
for all $b \in \B_{\geq \lambda}$.
Now we have checked all of the properties, so we can now apply \cref{inductionbase} to deduce that $W$ is indeed a $\Delta$-layer of type $\lambda$. 

From \cref{extvanishingt}, it follows
that $\Ext^2_A(W, \bar\nabla(b)) = 0$, hence, we get
also that $\Ext^1_A(K, \bar\nabla(b)) = 0$ for all $b \in \B$. Also $|\supp_\Delta(K)| < |\supp_\Delta(V)|$
since $\Hom_A(K, \bar\nabla(b))$ is a quotient of
$\Hom_A(V, \bar\nabla(b))$ for all $b \in \B$, and 
$\Hom_A(K, \bar\nabla(b)) = 0$ for 
$b \in \B_\lambda$ so $\lambda \notin \supp_\Delta(K)$.
This means that we can apply the induction hypothesis
to the module $K$ to deduce that it has a $\Delta$-flag.
Also none of the layers in such a flag 
are of type $\lambda$, again
because $\Hom_A(K, \bar\nabla(b)) = 0$ for $b \in \B_\lambda$. Now we have in our hands a $\Delta$-flag of $V$ coming from the $\Delta$-flag of $K$ plus the top section that is the $\Delta$-layer $W$ of type $\lambda$.
Thus, (a) is proved.

\vspace{1mm}
\noindent
(2)
This is a very similar argument. For the hardest implication
(b)$\Rightarrow$(a),
one proceeds by induction on the size of the set
$\supp_{\bar\Delta}(V)$.
Noting that $\supp_{\Delta}(V) \subseteq 
\supp_{\bar\Delta}(V)$, 
we are done trivially in case $\supp_{\bar\Delta}(V) = \varnothing$
as before.
Then we repeat the arguments in (a)
replacing $\supp_\Delta(V)$ and $\bar\nabla(b)$ with
$\supp_{\bar\Delta}(V)$ and $\nabla(V)$.
\end{proof}

\begin{corollary}[BGG reciprocity for projectives]\label{bggproj}
For $b \in \B$, the indecomposable
projective $P(b)$ has a $\Delta$-flag
with
$(P(b):\Delta(a))_q = \overline{[\bar\nabla(a):L(b)]_q}$
for all $a \in \B$.
If the graded triangular basis admits a duality
then $(P(b):\Delta(a))_q = [\bar\Delta(a):L(b)]_q$.
\end{corollary}

\begin{proof}
The fact that $P(b)$ has a $\Delta$-flag
follows from \cref{supports} and
the homological criterion of \cref{citizens}.
For the multiplicities, we compute
from the definition \cref{seasonalmuffin}:
$$
(P(b):\Delta(a))_q = \overline{\dim_q \Hom_A(P(b), \bar\nabla(a))}
= \overline{[\bar\nabla(a):L(b)]_q}.
$$
\end{proof}

\begin{corollary}\label{sesdelta}
Suppose that $0 \rightarrow U \rightarrow V \rightarrow W 
\rightarrow 0$ is a short exact sequence of 
graded left $A$-modules. 
Assuming that $W$ has a $\bar\Delta$-flag, $U$ has a $\bar\Delta$-flag if and only if $V$ has a $\bar\Delta$-flag.
Similarly for $\Delta$-flags.
\end{corollary}

\begin{proof}
We explain for $\bar\Delta$-flags, the case of $\Delta$-flags being similar.
Since $W$ is locally finite-dimensional 
bounded below as it has a 
$\bar\Delta$-flag, it is clear that
 $U$ is locally finite-dimensional and bounded below if and only if $V$ has these properties. Also, this is the case if either $U$ or $V$ has a $\bar\Delta$-flag.
Applying $\Hom_A(-,\nabla(b))$ to the short exact sequence 
using the vanishing of $\Ext^n_A(W, \nabla(b))$ for $n \geq 1$
gives short exact sequences
$$
0 \longrightarrow \Hom_A(W, \nabla(b))
\longrightarrow \Hom_A(V, \nabla(b))
\longrightarrow \Hom_A(U, \nabla(b))
\longrightarrow 0
$$
and isomorphisms
$\Ext^1_A(V,\nabla(b))
\cong
\Ext^1_A(U,\nabla(b))$
for all $b \in \B$.
The short exact sequences imply that
\begin{equation}\label{supportsum}
\supp_{\bar\Delta}(V) = \supp_{\bar\Delta}(U) \cup \supp_{\bar\Delta}(W).
\end{equation}
Hence, $\supp_{\bar\Delta}(U)$ is finite if and only if
$\supp_{\bar\Delta}(V)$ is finite. Now we can apply the
homological criterion for ${\bar\Delta}$-flags from \cref{citizens}
to deduce the result.
\end{proof}

\begin{corollary}
The categories $A\gmoddelta$ and $A\gmoddeltabar$
are closed under degree shift, finite direct sum and passing to graded direct summands.
\end{corollary}

Since they are often useful, we take the time to formulate the dual results too.
The {\em $\nabla$-} and {\em $\bar\nabla$-supports}
of $V \in \ob A\gmod$ are
\begin{align}\label{supp3}
\supp_\nabla(V) &:= \big\{\dot b\:\big|\:
b\in\B\text{ such that }\Hom_A(\bar\Delta(b), V) \neq 0\big\},\\\label{supp4}
\supp_{\bar\nabla}(V) &:= \big\{\dot b\:\big|\:
b \in \B\text{ such that }
\Hom_A(\Delta(b),V) \neq 0\big\}.
\end{align}
We have that $\supp_\nabla(V) \subseteq \supp_{\bar\nabla}(V)$. These sets are necessarily finite if $A$ is unital, or if $V$ is finitely cogenerated (this statement is dual to \cref{supports}).
Multiplicities in $\nabla$- and $\bar\nabla$-flags are
defined by
\begin{align}\label{pea}
(V:\nabla(b))_q &:= \dim_q \Hom_A(\bar\Delta(b), V) 
\in \N\lround q^{-1}\rround,\\\label{nut}
(V:\bar \nabla(b))_q &:= \dim_q \Hom_A(\Delta(b), V)
\in\N\lround q^{-1}\rround,
\end{align}
with interpretations similar to the ones explained
for $\Delta$- and $\bar\Delta$-flags.
For example, every $\nabla(a)$ has a $\bar\nabla$-flag
with
\begin{equation}
(\nabla(a):\bar\nabla(b))_q = 
\left\{
\begin{array}{ll}
[I_\lambda(a):L_\lambda(b)]_q&\text{if $a,b \in \B_\lambda$
for some $\lambda \in \Lambda$}\\
0&\text{if $\dot a \neq \dot b$.}
\end{array}\right.
\end{equation}
The dual results to \cref{fridaynight}, \cref{citizens} and its corollaries are as follows:

\begin{lemma}\label{saturdaynight}
If $V$ has a $\nabla$-flag then
$(V:\bar\nabla(b))_q = \sum_{a \in \B}
(V:\nabla(a))_q(\nabla(a):\bar\nabla(b))_q$.
\end{lemma}

\begin{theorem}[Homological criteria for 
$\nabla$- and $\bar\nabla$-flags]\label{citizenss}
Assume that $V \in \ob A\gmod$ is locally finite-dimensional and bounded above.
\begin{itemize}
\item[(1)]
The following are equivalent:
\begin{enumerate}
\item[(a)] $V$ has a $\nabla$-flag;
\item[(b)] 
$|\supp_\nabla(V)| < \infty$ and
$\Ext^1_A(\bar\Delta(b), V) = 0$ for all $b \in \B$;
\item[(c)]
$|\supp_\nabla(V)| < \infty$ and
$\Ext^n_A(\bar\Delta(b),V) = 0$ for all $b \in \B$ and $n \geq 1$.
\end{enumerate}
\item[(2)]
The following are equivalent:
\begin{enumerate}
\item[(a)] $V$ has a $\bar\nabla$-flag;
\item[(b)] 
$|\supp_{\bar\nabla}(V)| < \infty$ and
$\Ext^1_A(\Delta(b),V) = 0$ for all $b \in \B$;
\item[(c)] 
$|\supp_{\bar\nabla}(V)| < \infty$ and
$\Ext^n_A(\Delta(b),V) = 0$ for all $b \in \B$ and $n \geq 1$.
\end{enumerate}
\end{itemize}
\end{theorem}

\begin{corollary}[BGG reciprocity for injectives]\label{bgginj}
For $b \in \B$, the indecomposable
injective $I(b)$ has a $\nabla$-flag
with
$(I(b):\nabla(a))_q = \overline{[\bar\Delta(a):L(b)]_q}$
for all $a \in \B$.
If the graded triangular basis admits a duality
then $(I(b):\nabla(a))_q = [\bar\nabla(a):L(b)]_q$.
\end{corollary}

\begin{corollary}
Suppose that $0 \rightarrow U \rightarrow V \rightarrow W 
\rightarrow 0$ is a short exact sequence of 
graded left $A$-modules. 
Assuming that $U$ has a $\bar\nabla$-flag, $V$ has a $\bar\nabla$-flag if and only if $W$ has a $\bar\nabla$-flag.
Similarly for $\nabla$-flags.
\end{corollary}

\begin{corollary}\label{12ozofcoffee}
The categories $A\gmodnabla$ and $A\gmodnablabar$
are closed under degree shift, finite direct sum and passing to graded direct summands.
\end{corollary}

We record one more lemma which will be needed in the next section.

\begin{lemma}\label{more}
If $V$ has a $\bar\Delta$-flag (resp., 
a $\bar\nabla$-flag) then
$[V:L(b)]_q = \sum_{a \in \B} (V:\bar\Delta(a))_q [\bar\Delta(a):L(b)]_q$
(resp., $\sum_{a \in \B} (V:\bar\nabla(a))_q [\bar\nabla(a):L(b)]_q$).
\end{lemma}

\begin{proof}
We just prove the result when $V$ has a $\bar\nabla$-flag, the other case being the dual statement.
We may assume that $V$ is a single $\bar\nabla$-layer,
so
$V \cong j^\lambda_* \bar V$
for a graded left $A_\lambda$-module $\bar V$ that is locally finite-dimensional and bounded above.
By \cref{bggproj}, $P(b)$ has a $\Delta$-flag
with sections given by the $\Delta$-layers 
$\bigoplus_{a \in \B_\mu} \Delta(a)^{\oplus \overline{[\bar\nabla(a):L(b)]_q}}$ of type $\mu$
for all $\mu \in \Lambda$ (this being zero unless $\mu \leq \dot b$).
Using \cref{extvanishingt}, we deduce that
\begin{align*}
[V:L(b)]_q &= \dim_q \Hom_A(P(b),V)
= \sum_{\substack{\mu \in \Lambda\\a \in \B_\mu} }
\dim_q \Hom_A\Big(\Delta(a)^{\oplus \overline{[\bar\nabla(a):L(b)]_q}}, j^\lambda_* \bar V
\Big)\\
&= \sum_{\substack{\mu \in \B_{\geq \lambda}\\a \in \B_\mu}}
\dim_q \Hom_{A_{\geq \lambda}}\Big(
\Delta(a)^{\oplus \overline{[\bar\nabla(a):L(b)]_q}},
j^\lambda_* \bar V\Big)
= 
\sum_{a \in \B_\lambda}
\dim_q \Hom_{A_{\lambda}}\Big(P_\lambda(a)^{\oplus\overline{[\bar\nabla(a):L(b)]_q}}, \bar V\Big).
\end{align*}
To complete the proof, we
show that the $q^d$-coefficient 
of $\dim_q \Hom_{A_{\lambda}}\Big(P_\lambda(a)^{\oplus\overline{[\bar\nabla(a):L(b)]_q}}, \bar V\Big)$
is equal to the $q^d$-coefficient of $(V:\bar\nabla(a))_q[\bar\nabla(a):L(b)]_q$
for each $a \in \B_\lambda$ and $d \in \Z$.
Like in \cref{wednesdaynight}, we have that
 \begin{equation}\label{tuesdaynight}
(V:\bar \nabla(a))_q = \dim_q \Hom_A(\Delta(a), V)
= 
\dim_q \Hom_{A_\lambda}\big(P_\lambda(a), \bar V\big)
= [\bar V:L_\lambda(a)]_q.
\end{equation}
Assuming that  $(V:\bar\nabla(a))_q = \sum_{m \in \Z} r_m q^m$ and
$[\bar\nabla(a):L(b)]_q
 = \sum_{n \in \Z} s_n q^n$,
we deduce that 
$$
\dim\Hom_{A_{\lambda}}\Big(P_\lambda(a)^{\oplus\overline{[\bar\nabla(a):L(b)]_q}}, \bar V\Big)_{d}
=
\dim \prod_{n \in \Z} \Hom_{A_\lambda}\big(P_\lambda(a), \bar V\big)^{\oplus s_n}_{d-n}
=
\sum_{n \in \Z} r_{d-n} s_{n},
$$
which is the $q^d$-coefficient of $(V:\bar\nabla(a))_q[\bar\nabla(a):L(b)]_q$
as we wanted.
\end{proof}

\section{Truncation to finite lower sets}\label{trunctolower}

Now let $\Gamma$ be a {\em finite} lower set in $\Lambda$
and set $\S_\Gamma := \{s \in \S\:|\:\dot s \in \Gamma\}$, $\B_\Gamma := \{b \in \B\:|\:\dot b \in \Gamma\}$.
Let $e_\Gamma := \sum_{\lambda \in \Gamma} e_\lambda$.
Then $A_\Gamma := e_\Gamma A e_\Gamma = \bigoplus_{s,t \in \S_\Gamma} 1_s A 1_t$ is a unital graded algebra which is locally finite-dimensional and bounded below.
We let 
\begin{equation}
j^\Gamma: A\gmod \rightarrow A_\Gamma\gmod
\end{equation}
be the quotient functor defined by truncating with the  idempotent $e_\Gamma$. As explained at the start of
\cref{saturday}, $j^\Gamma$ fits into an 
adjoint triple $(j^\Gamma_!, j^\Gamma, j^\Gamma_*)$.

The algebra $A_\Gamma$ has a graded triangular basis
with special idempotents
$1_s\:(s \in \S_\Gamma)$,
the finite weight poset $(\Gamma,\leq)$, and basis elements arising from
the sets $\X(s,t), \H(s,t)$ and $\Y(s,t)$
for all $s, t \in \S_\Gamma$.
For $\lambda\in\Gamma$, 
it is clear by considering the bases that
the quotient algebra
$(A_\Gamma)_{\geq\lambda}$ 
of $A_\Gamma$ may be identified with the
idempotent truncation $(A_{\geq \lambda})_\Gamma = \bar e_\Gamma A_{\geq \lambda} \bar e_\Gamma$ of $A_{\geq \lambda}$. 
Hence, $(A_\Gamma)_\lambda$ is 
identified with exactly the same algebra $A_\lambda = \bar e_\lambda A_{\geq \lambda} \bar e_\lambda$ as before.
The analog of the 
adjoint triple $(j^\lambda_!, j^\lambda, j^\lambda_*)$
for $A_\Gamma$ will be denoted $(j^{\Gamma,\lambda}_!,
j^{\Gamma,\lambda}, j^{\Gamma,\lambda}_*)$.
So 
\begin{equation}
j^{\Gamma,\lambda}:(A_{\Gamma})_{\geq \lambda}
\gmod \rightarrow A_\gamma\gmod
\end{equation}
is the idempotent
truncation functor defined by $\bar e_\lambda$, and
$j^{\Gamma,\lambda}_!$ and $j^{\Gamma,\lambda}_*$
are its left and right adjoints.

The standard, proper standard, costandard and proper
costandard modules for $A_\Gamma$ 
arising from the graded triangular basis are 
\begin{align}\label{standardsyetagain}
\Delta_\Gamma(b)& := 
\jmath^{\Gamma,\lambda}_! P_\lambda(b),&
\bar\Delta_\Gamma(b) &:= \jmath^{\Gamma,\lambda}_! L_\lambda(b),&
\bar\nabla_\Gamma(b)& := \jmath^{\Gamma,\lambda}_* L_\lambda(b),&
 \nabla_\Gamma(b) &:= \jmath^{\Gamma,\lambda}_* I_\lambda(b)
\end{align}
for $b \in \B_\Gamma$ and $\lambda := \dot b$.
Then, by \cref{irrclass}, the modules
$L_\Gamma(b) := \head \Delta_\Gamma(b) = \soc \nabla_\Gamma(b)$ for $b \in \B_\Gamma$ 
give a complete set of irreducible graded
left $A_\Gamma$-modules up to isomorphism and degree shift. We denote a projective cover
and an injective hull of $L_\Gamma(b)$
by $P_\Gamma(b)$ and $I_\Gamma(b)$, respectively.

\begin{lemma}\label{easy}
For $b \in \B_\Gamma$,  we have that
$j^\Gamma_! P_\Gamma(b) \cong P(b)$,
$j^\Gamma_! \Delta_\Gamma(b) \cong \Delta(b)$,
$j^\Gamma_! \bar\Delta_\Gamma(b) \cong \bar\Delta(b)$,
$j^\Gamma_* \bar\nabla_\Gamma(b) \cong \bar\nabla(b)$.
$j^\Gamma_* \nabla_\Gamma(b) \cong \nabla(b)$,
$j^\Gamma_* I_\Gamma(b) \cong I(b)$.
Also 
$j^\Gamma L(b) \cong L_\Gamma(b)$
for $b \in \B_\Gamma$, and $j^\Gamma \Delta(b) = j^\Gamma\bar\Delta(b) = j^\Gamma L(b) = j^\Gamma \bar\nabla(b) =j^\Gamma \nabla(b) = 0$ for $b \in \B-\B_\Gamma$.
\end{lemma}

\begin{proof}
The functor $j^\lambda:A_{\geq \lambda}\gmod
\rightarrow A_\lambda\gmod$
is 
the composition of 
$j^\Gamma:A_{\geq \lambda}\gmod
\rightarrow \bar e_\Gamma A_{\geq \lambda} \bar e_\Gamma \gmod$
followed by 
$j^{\Gamma,\lambda}:\bar e_\Gamma A_{\geq \lambda} \bar e_\Gamma \gmod
\rightarrow A_\lambda\gmod$.
Hence, $j^\lambda_! \cong j^\Gamma_! \circ j^{\Gamma,\lambda}_!$, giving that
$j^\Gamma_! \Delta_\Gamma(b) \cong \Delta(b)$ and
$j^\Gamma_! \bar\Delta_\Gamma(b) \cong \bar\Delta(b)$.
Similarly, $j^\lambda_* \cong j^\Gamma_* \circ j^{\Gamma,\lambda}_*$, giving that
$j^\Gamma_* \bar\nabla_\Gamma(b) \cong \bar\nabla(b)$.
$j^\Gamma_* \nabla_\Gamma(b) \cong \nabla(b)$.

Next we show that $j L(b) \cong L_\Gamma(b)$
for $b \in \B_\lambda$ and $\lambda \in \Gamma$.
This follows
because $j^{\Gamma,\lambda} \left(j^\Gamma L(b)\right)  = j^\lambda L(b)
\cong L_\lambda(b)$ as $A_\lambda$-modules.
Then we deduce that $j^\Gamma_! P_\Gamma(b)
\cong P(b)$ and $j^\Gamma_* I_\Gamma(b) \cong I(b)$
for $b \in \B_\Gamma$ using adjunction properties.

Finally, it is clear that $j \Delta(b) = j \bar\Delta(b) = j L(b) = j \bar\nabla(b) =j \nabla(b) = 0$ for $b \in 
\B- \B_\Gamma$,
since all these have lowest weight $\dot b$,
hence, they have no weights that are in $\Gamma$.
\end{proof}

\begin{corollary}\label{easyc}
For $b \in \B_\Gamma$, we have that
$j^\Gamma P(b)\cong 
P_\Gamma(b) $,
$j^\Gamma \Delta(b)\cong \Delta_\Gamma(b)$,
$j^\Gamma \bar\Delta(b)\cong \bar\Delta_\Gamma(b)$,
$j^\Gamma \bar\nabla(b)\cong \bar\nabla_\Gamma(b)$,
$j^\Gamma \nabla(b)\cong \nabla_\Gamma(b)$ and
$j^\Gamma I(b)\cong I_\Gamma(b)$.
\end{corollary}

\begin{proof}
This follows from the lemma since $j^\Gamma \circ j^\Gamma_! \cong \id_{A_\Gamma\gmod} \cong j^\Gamma \circ j^\Gamma_*$.
\end{proof}

\begin{lemma}\label{anothergss}
For $V \in \ob A\gmod$ and $W \in \ob A_\Gamma \gmodnablabar$, we have that
$\Ext^n_{A_\Gamma}\left(j^\Gamma V, W\right) \cong \Ext^n_A\left(V, j^\Gamma_* W\right)$
for all $n \geq 0$.
\end{lemma}

\begin{proof}
This is another Grothendieck spectral sequence argument.
We have that $\Hom_{A_\Gamma}(-,W) \circ j^\Gamma
\cong \Hom_A(-,j^\Gamma_* W)$.
Also $j^\Gamma$ is exact.
To deduce that
$\Ext^n_{A_\Gamma}(j^\Gamma -, W) \cong \Hom_A(-, j^\Gamma_* W)$,
it remains to show that $j^\Gamma$ sends projectives
in $A\gmod$
to modules that are acyclic
for $\Hom_{A_\Gamma}(-,W)$.
Since any projective in $A \gmod$ is a summand of a direct
sum of degree-shifts of the projective modules $Q(b)$ from 
\cref{ohno}(1), and $j^\Gamma$ commutes with direct sum and with $Q$, the proof of this reduces to checking that
$\Ext^n_{A_\Gamma}(j^\Gamma Q(b), W) = 0$ for 
all $b \in \B$ and $n \geq 1$.
Since $j^\Gamma$ is exact and $j^\Gamma \Delta(b)$ is either zero or
a standard module for $A_\Gamma$ by \cref{easy,easyc},
we deduce that $j^\Gamma Q(b)$ has a $\Delta$-flag.
Hence, 
$\Ext^n_{A_\Gamma}(j^\Gamma Q(b), W) = 0$ for $n \geq 1$
thanks to \cref{extvanishingt}.
\end{proof}

The dual
result to \cref{anothergss} will be formulated and 
proved in \cref{anothergss2} below. It does not follow immediately at this point
since we have not included any
assumption of locally finite-dimensionality on $W$ in
the statement.

\begin{lemma}\label{lunch}
If $V \in \ob A_\Gamma\gmoddelta$
then $j^\Gamma_! V \in \ob A \gmoddelta$
with 
$\supp_{\Delta}\left(j_!^\Gamma V\right) = \supp_{\Delta}(V)$,
indeed, 
we have
$\left(j_!^\Gamma V: \Delta(b)\right) = \left(V:\Delta_\Gamma(b)\right)$ for 
$b \in \B_\Gamma$.
The same statement with $\Delta$ replaced by $\bar\Delta$ everywhere
also holds.
Similarly, if $V \in \ob A_\Gamma\gmodnabla$
then $j^\Gamma_* V \in \ob A \gmodnabla$
with 
$\supp_{\nabla}\left(j_*^\Gamma V\right) = \supp_{\nabla}(V)$,
indeed, 
we have
$\left(j_*^\Gamma V: \nabla(b)\right) = \left(V:\nabla_\Gamma(b)\right)$ for 
$b \in \B_\Gamma$.
The same statement with $\nabla$ replaced by $\bar\nabla$ everywhere also holds.
\end{lemma}

\begin{proof}
The statements for $\Delta$ and $\bar \Delta$
follow from the ones for $\nabla$ and $\bar\nabla$
by the usual duality argument.
Now we proceed to prove the statement for $\nabla$, with
a similar argument proving the one for $\bar\nabla$.
For $i \in \I$, we have that $1_i (j^\Gamma_* V) = \Hom_{A_\Gamma}(e_\Gamma A 1_i, V) \subseteq \Hom_\kk(e_\Gamma A 1_i, V)$.
Since $e_\Gamma A 1_i$ is locally finite-dimensional and bounded below and $V$ is locally finite-dimensional and bounded above, we deduce that $j^\Gamma_* V$
is locally finite-dimensional and bounded above.

We have that $\Hom_A\left(\bar\Delta(b), j^\Gamma_* V\right)
\cong \Hom_{A_\Gamma}\left(j^\Gamma \bar\Delta(b), V\right)$, and
$j^\Gamma \bar\Delta(b) = \bar\Delta_\Gamma(b)$ if $b \in \B_\Gamma$
or $0$ otherwise thanks to \cref{easy,easyc}. 
Once we have proved that $j^\Gamma_* V$ has a $\nabla$-flag,
this will imply the statement about the multiplicities $\left(j^\Gamma_* V: \nabla(b)\right)$.
It shows already that
$\left|\supp_{\nabla}(j^\Gamma_* V)\right| 
= |\supp_{\nabla}(V)| < \infty$.
Also $\Ext^1_A\left(\bar\Delta(b),j^\Gamma_* V\right)
\cong \Ext^1_{A_\Gamma}\left(j^\Gamma\bar\Delta(b), V\right)$
by \cref{anothergss}, which is zero for all $b \in \B$
as $V$ has a $\bar\Delta$-flag.
It remains to apply \cref{citizenss}.
\end{proof}

\begin{lemma}\label{dinner}
For $V \in \ob A_\Gamma\gmoddeltabar$ and $i \in \I$,
we have that
$\Tor_m^{A_\Gamma}(1_i Ae_\Gamma, V) = 0$
for all $m \geq 1$.
\end{lemma}

\begin{proof}
Consider the short exact sequence
$0 \rightarrow K \rightarrow P \rightarrow V$
where $P := P_V$ is the projective cover of $V$ 
in $A_\Gamma\gmod$ from \cref{tech1}(1).
We note that $P$ has a $\bar\Delta$-flag.
This follows from \cref{citizens} 
(the condition
$|\sup_{\bar\Delta}(P)| < \infty$ holds automatically
since $\Gamma$ is finite).
Also $V$ has a $\bar\Delta$-flag by assumption.
Hence, $K$ has a $\bar\Delta$-flag by \cref{sesdelta}.
Now \cref{lunch} implies that $j^\Gamma_! V, j^\Gamma_! P$ and $j^\Gamma_! L$
all have $\bar\Delta$-flags,
and moreover
$\left(j^\Gamma_! P: \bar\Delta(b)\right) = 
\left(j^\Gamma_! K: \bar\Delta(b)\right)+
\left(j^\Gamma_! V: \bar\Delta(b)\right)$
for all $b \in \B$.
Applying \cref{more}, we deduce that
$$
\left[j^\Gamma_! P: L(b)\right]_q = \left[j^\Gamma_! K:L(b)\right]_q+\left[j^\Gamma_! V: L(b)\right]_q
$$
for each $b \in \B$. 
It follows that
$\dim_q 1_i A e_\Gamma \otimes_{A_\Gamma} P=
\dim_q 1_i A e_\Gamma \otimes_{A_\Gamma} K+
\dim_q 1_i A e_\Gamma \otimes_{A_\Gamma} V$.
Applying $1_i A e_\Gamma \otimes_{A_\Gamma}-$ to the short exact sequence gives
the long exact sequence
$$
0 \longrightarrow \Tor_1^{A_\Gamma}(1_i A e, V)
\longrightarrow 1_i A e_\Gamma \otimes_{A_\Gamma} K \longrightarrow 1_i A e_\Gamma\otimes_{A_\Gamma} P
\longrightarrow 1_i A e_\Gamma\otimes_{A_\Gamma} V \longrightarrow 0
$$
and isomorphisms $\Tor_m^{A_\Gamma}(1_i Ae, K) 
\cong \Tor_{m+1}^{A_\Gamma}(1_i Ae, V)$ for all $m \geq 1$.
From this long exact sequence and the equality of dimensions already established, we deduce that
$\Tor_1^{A_\Gamma}(1_i Ae, V) = 0$.
This applies equally well to $K$, so
we get that $\Tor_1^{A_\Gamma}(1_i Ae, K) = 0$,
hence, $\Tor_2^{A_\Gamma}(1_i Ae, V) = 0$.
Further degree shifting like this completes the proof.
\end{proof}

\begin{corollary}
The functor $j^\Gamma_! = Ae_\Gamma\otimes_{\check A_\Gamma} -$
takes short exact sequences of modules with
$\Delta$-flags 
(resp., $\bar\Delta$-flags)
to short exact sequences
of modules with $\Delta$-flags (resp., $\bar\Delta$-flags).
Similarly,
The functor $j^\Gamma_* = \bigoplus_{i \in \I} \Hom_{A_\Gamma} (e_\Gamma A 1_i, )$
takes short exact sequences of modules with
$\nabla$-flags 
(resp., $\bar\nabla$-flags)
to short exact sequences
of modules with $\nabla$-flags (resp., $\bar\nabla$-flags).
\end{corollary}

\begin{proof}
The results for $\nabla$-flags and $\bar\nabla$-flags
follow for the ones for $\Delta$-flags and $\bar\Delta$-flags by duality.
The proofs of $\Delta$-flags and $\bar\Delta$-flags are
similar. In the case of $\Delta$-flags,
the functor $j_!$ takes modules with $\Delta$-flags
to modules with $\Delta$-flags by \cref{lunch}.
It is exact on $A_\Gamma\gmoddeltabar$ by \cref{dinner},
hence, it is exact on $A_\Gamma\gmoddelta$ too since this is a subcategory.
\end{proof}

The next theorem will be useful in the next section.
For $V \in A\gmod$, we let
\begin{align}\label{tools}
V_\Gamma &:= A e_\Gamma V,&
V^\Gamma &:= \{v \in V\:|\:e_\Gamma A v = 0 \}.
\end{align}
The counit of adjunction for the adjoint pair $(j^\Gamma_!, j^\Gamma)$
defines a homomorphism
$\eps^\Gamma_V: j^\Gamma_! j^\Gamma V \rightarrow V$.
This is just the natural multiplication
map $A e_\Gamma \otimes_{A_\Gamma} e_\Gamma V \rightarrow V$,
so its image is the submodule $V_\Gamma$ just defined.
Also the unit of adjunction for
the adjoint pair $(j^\Gamma, j^\Gamma_*)$ defines a homomorphism
$\eta^\Gamma_V: V \rightarrow j^\Gamma_* j^\Gamma V$.
This takes $v \in V$ to
the element of $j^\Gamma_* j^\Gamma V = \bigoplus_{i \in \I} \Hom_{A_\Gamma} (e_\Gamma A 1_i, V)$
that maps $e_\Gamma a 1_i \in e_\Gamma A 1_i$ 
to $e_\Gamma a 1_i v$.
From this, we see that $\ker \eta^\Gamma_V = V^\Gamma$.

\begin{theorem}\label{aha}
Suppose that $V \in A\gmod$.
\begin{enumerate}
\item 
If $V_\Gamma$ has a $\bar\Delta$-flag then 
the counit of adjunction
defines an isomorphism
$\eps^\Gamma_V:j^\Gamma_! j^\Gamma V \stackrel{\sim}{\rightarrow} V_\Gamma$.
\item
If $V / V^\Gamma$ has a $\bar\nabla$-flag then 
the unit of adjunction defines an isomorphism
$\eta^\Gamma_V: V / V^\Gamma \stackrel{\sim}{\rightarrow}
j^\Gamma_* j^\Gamma V$.
\end{enumerate}
\end{theorem}

\begin{proof}
(1)
Suppose that $V_\Gamma = A e_\Gamma V$ has a 
$\bar\Delta$-flag.
Let $K := \ker \eps^\Gamma_V$ 
so that there is a short exact
sequence $0 \rightarrow K \rightarrow j^\Gamma_! j^\Gamma V \rightarrow
V_\Gamma \rightarrow 0$.
We need to show that $K = 0$.
Note that the second map in this short exact sequence
becomes an isomorphism when we apply $j^\Gamma$, so 
we have that $j^\Gamma K = 0$.
Since $j^\Gamma$ is exact, it is clear 
from \cref{easy} that $j^\Gamma V_\Gamma$ 
has a $\bar\Delta$-flag.
Since $j^\Gamma V_\Gamma = e_\Gamma Ae_\Gamma V = j^\Gamma V$, we deduce that $j^\Gamma V$ has a $\bar\Delta$-flag.
Now \cref{lunch} gives that $j^\Gamma_! j^\Gamma V$ has a $\bar\Delta$-flag
with $\supp_{\bar\Delta}\left(j^\Gamma_! j^\Gamma V\right) \subseteq \Gamma$. 
By \cref{sesdelta,supportsum},
we deduce that
$K$ has a $\bar\Delta$-flag
with $\supp_{\bar\Delta}(K) \subseteq \Gamma$ too.
Since $j^\Gamma K = 0$ 
and $j^\Gamma$ is non-zero on any $\bar\Delta(b)$,
we must have that $K = 0$.

\vspace{1mm}
\noindent
(2) This is quite similar.
Start from the short exact sequence
$0 \rightarrow V / V^\Gamma \rightarrow j^\Gamma_*
j^\Gamma V \rightarrow Q \rightarrow 0$.
We must show that $Q = 0$. The first map becomes an isomorphism when we apply $j^\gamma$, so $j^\gamma Q = 0$.
It remains to show that $Q$ has a $\bar\nabla$-flag
with $\supp_{\bar\nabla}(Q) \subseteq \Gamma$.
This follows from \cref{12ozofcoffee} and the obvious analog of \cref{supportsum} 
because $V / V^\Gamma$ has a $\bar\nabla$-flag by assumption
and $j^\Gamma_* j^\Gamma V \cong j^\Gamma_* j^\Gamma (V / V^\Gamma)$
has a $\bar\nabla$-flag with the appropriate support by \cref{lunch}.
\end{proof}

The final lemma is the dual version of \cref{anothergss}
promised earlier.

\begin{lemma}\label{anothergss2}
For $V \in \ob A_\Gamma\gmoddeltabar$ and $W \in \ob A \gmod$, we have that
$\Ext^n_{A_\Gamma}\left(V, j^\Gamma W\right) \cong \Ext^n_A\left(j^\Gamma_! V, W\right)$
for all $n \geq 0$.
\end{lemma}

\begin{proof}
We have that $\Hom_{A}(-,W) \circ j^\Gamma_!
\cong \Hom_{A_\Gamma}(-, j^\Gamma W)$.
Also $j^\Gamma_!$ takes projectives to projectives 
since it is left adjoint to an exact functor.
Therefore, by the usual argument,
we have that
$\Ext^n_A(j^\Gamma_! V, W) \cong \Ext^n_{A_\Gamma}(V, j^\Gamma W)$
for all $n \geq 0$ and $V \in A_\Gamma\gmod$ 
such that $\Tor_m^{A_\Gamma}(Ae_\Gamma, V) = 0$ for $m \geq 1$. 
This holds for $V \in A_\Gamma\gmoddeltabar$ by
\cref{dinner}.
\end{proof}

\section{Semi-infinite flags}

When the algebra
$A$ (still possessing 
a graded triangular basis) is not unital, 
it also makes sense to consider
certain semi-infinite 
$\Delta$-flags, 
$\bar\Delta$-flags,
$\nabla$-flags and $\bar\nabla$-flags.
These were introduced in \cite[Def.~3.35]{BS} in the ungraded setting, and then they were there used 
to introduce tilting modules.
In this section, we make some first steps 
in this direction in the graded setting
by setting up the basic facts about semi-infinite flags.
Throughout the section, we will make use of the notation
from the previous section
for a finite lower set $\Gamma \subseteq \Lambda$, especially \cref{tools}.

\begin{definition}\label{upity}
We say that a graded left $A$-module $V$ has an
{\em ascending $\Delta$-flag} 
(resp., an {\em ascending $\bar\Delta$-flag})
if the $A$-submodule $V_\Gamma$ has a $\Delta$-flag (resp., a $\bar\Delta$-flag)
for all finite 
lower sets $\Gamma \subseteq \Lambda$.
\end{definition}

\begin{definition}\label{downity}
We say that a graded left $A$-module $V$ has a
{\em descending $\nabla$-flag} 
(resp., a {\em descending $\bar\nabla$-flag})
if the quotient module $V / V^\Gamma$
has a $\nabla$-flag (resp., a $\bar\nabla$-flag)
for all finite lower sets ${\Gamma} \subseteq \Lambda$.
\end{definition}

Our first lemma shows that in order to
check the conditions in \cref{upity,downity},
it suffices just to consider
finite lower sets $\Gamma \subseteq \Lambda$
that are sufficiently large.
In particular, if $A$ is unital
(i.e., $\{\lambda \in \Lambda\:|\:e_\lambda \neq 0\}$
is finite), we deduce that $V$ has an ascending
$\Delta$-flag if and only if $V$ has a $\Delta$-flag in the earlier sense, and similarly for $\bar\Delta$-flags,
$\nabla$-flags and $\bar\nabla$-flags.
So these new notions are only interesting in the non-unital case.

\begin{lemma}\label{pizza}
Let $\Gamma \subseteq \Pi$ be two finite lower sets in
$\Lambda$
and $V \in \ob A\gmod$.
\begin{enumerate}
\item
If
$V_\Pi$ has a $\Delta$-flag (resp., a $\bar\Delta$-flag)
then so do $V_\Gamma$ and $V_\Pi / V_\Gamma$.
\item
If $V / V^\Pi$ has a $\nabla$-flag
(resp., a $\bar\nabla$-flag) then so do 
$V / V^\Gamma$ and $V^\Gamma / V^\Pi$.
\end{enumerate}
\end{lemma}

\begin{proof}
We just go through the details for $\Delta$-flags, the other cases are similar.
Since $e_\Gamma = e_\Pi e_\Gamma = e_\Gamma e_\Pi$,
we have that 
$V_\Gamma \subseteq  V_\Pi$.
We are given that $V_\Pi$ has a $\Delta$-flag.
Clearly its sections are $\Delta$-layers of types
from $\Pi$. 
Using \cref{archersc}, we can arrange the layers
to obtain a short exact sequence
$0 \rightarrow K \rightarrow V_\Pi \rightarrow Q \rightarrow 0$ so that $K$ has a $\Delta$-flag
with $\Delta$-layers of types from $\Gamma$
and $Q$ has a $\Delta$-flag with layers from $\Pi - \Gamma$. But $e_\Gamma$ is zero on $\Delta$-layers
of types from $\Pi - \Gamma$, and any $\Delta$-layer $W$
of type from $\Gamma$ is generated by $e_\Gamma W$.
It follows that $K = V_\Gamma$,
$Q = V_\Pi / V_\Gamma$, so both have $\Delta$-flags.
\end{proof}

For $V$ with an ascending $\Delta$-flag or
an ascending $\bar\Delta$-flag,
we define the multiplicities
$(V:\Delta(b))_q$ and 
$(V:\bar\Delta(b))_q$
by the same formulae \cref{seasonalmuffin,tea}
as before.
They both belong to $\N\lround q^{-1} \rround$ thanks to the next lemma.
Similarly, we define
$(V:\nabla(b))_q$ and $(V:\bar\nabla(b))_q$ for $V$ with
a descending $\nabla$-flag or a descending
$\bar\nabla$-flag by
\cref{pea,nut}; these necessarily belong to $\N\lround q\rround$.

\begin{lemma}\label{wherefinalaxiomisneeded}
Let $V$ be a graded left $A$-module.
\begin{enumerate}
\item
If $V$ has an ascending $\Delta$-flag (resp.,
an ascending $\bar\Delta$-flag) then $V$ is locally finite-dimensional and bounded below.
\item
If $V$ has a descending $\nabla$-flag (resp.,
a descending $\bar\nabla$-flag) then $V$ is locally finite-dimensional and bounded above.
\end{enumerate}
\end{lemma}

\begin{proof}
(1)
It suffices to prove that $V$ is locally finite-dimensional and bounded below assuming if it has an ascending $\bar\Delta$-flag.
Fix a choice of $i \in \I$.
If $1_i (j^\lambda_! \bar V) \neq 0$
for some $\lambda \in \Lambda$ and a graded 
left 
$A_\lambda$-module $\bar V$ that is locally finite-dimensional
and bounded below,
then by \cref{thisisabasis} $\bar x \otimes v \neq 0$
for some $x \in \X(i,s)$, $v \in 1_s \bar V$ 
and $s \in \S_\lambda$.
By axiom (A4),
there are only finitely many possibilities for $\lambda$.
Let $\Gamma$ be the finite lower set in $\Lambda$ generated
by all of them. It follows that
$1_i V = 1_i V_\Gamma$. Since $V_\Gamma$ has a $\bar\Delta$-flag, it is locally finite-dimensional and bounded below by \cref{xmas}. Hence, so is $V$.

\vspace{1mm}
\noindent
(2) This follows by the dual argument.
\end{proof}

Now we are ready for the main results of the section.
These are almost the same as \cref{citizens,citizenss},
it is just that the conditions on finite support 
have been removed.

\begin{theorem}[Homological criteria for 
ascending $\Delta$- and $\bar\Delta$-flags]\label{villains}
Assume that $V \in \ob A\gmod$ is locally finite-dimensional and bounded below.
\begin{itemize}
\item[(1)]
The following are equivalent:
\begin{enumerate}
\item[(a)] $V$ has an ascending $\Delta$-flag;
\item[(b)] $j^\Gamma V$ has a $\Delta$-flag for all finite lower sets $\Gamma\subseteq\Lambda$;
\item[(c)] 
$\Ext^1_A(V, \bar\nabla(b)) = 0$ for all $b \in \B$;
\item[(d)] 
$\Ext^n_A(V, \bar\nabla(b)) = 0$ for all $b \in \B$ and $n \geq 1$.
\end{enumerate}
When this holds, for any finite lower set $\Gamma \subseteq \Lambda$, both $V_\Gamma$ and 
$V / V_\Gamma$ have ascending $\Delta$-flags
with
\begin{align}
(V_\Gamma:\Delta(b))_q &= \left\{
\begin{array}{ll}
(V:\Delta(b))_q&\text{if $b \in \B_\Gamma$}\\
0&\text{otherwise;}
\end{array}
\right.
&
(V / V_\Gamma:\Delta(b))_q &= \left\{
\begin{array}{ll}
0&\text{if $b \in \B_\Gamma$}\\
(V:\Delta(b))_q&\text{otherwise.}
\end{array}
\right.\label{cats}
\end{align}
\item[(2)]
The following are equivalent:
\begin{enumerate}
\item[(a)] $V$ has an ascending $\bar\Delta$-flag;
\item[(b)] $j^\Gamma V$ has a $\bar\Delta$-flag
for all finite lower sets $\Gamma \subseteq \Lambda$;
\item[(c)] 
$\Ext^1_A(V, \nabla(b)) = 0$ for all $b \in \B$;
\item[(d)] 
$\Ext^n_A(V, \nabla(b)) = 0$ for all $b \in \B$ and $n \geq 1$.
\end{enumerate}
When this holds, for any finite lower set $\Gamma \subseteq \Lambda$, both $V_\Gamma$ and 
$V / V_\Gamma$ have ascending $\bar\Delta$-flags
with
\begin{align}\label{catsbar}
(V_\Gamma:\bar\Delta(b))_q &= \left\{
\begin{array}{ll}
(V:\bar\Delta(b))_q&\text{if $b \in \B_\Gamma$}\\
0&\text{otherwise;}
\end{array}
\right.
&
(V / V_\Gamma:\bar\Delta(b))_q &= \left\{
\begin{array}{ll}
0&\text{if $b \in \B_\Gamma$}\\
(V:\bar\Delta(b))_q&\text{otherwise.}
\end{array}
\right.
\end{align}
\end{itemize}
\end{theorem}

\begin{proof}
(1)
It is clear that (d)$\Rightarrow$(c). 

To prove that (a)$\Rightarrow$(d), the canonical map
$\varinjlim_{\Gamma} V_{\Gamma} \rightarrow V$ 
is an isomorphism, where the direct limit is over all finite lower sets $\Gamma \subset \Lambda$ with maps given by the natural inclusions.
This follows because $V$ is generated by all of its weight spaces
$e_\lambda V\:(\lambda \in \Lambda)$,
and the poset is lower finite
so every weight space is a subset of $V_\Gamma$ for some
finite lower set $\Gamma$.
So
$$
\Ext^n_A(V, \bar\nabla(b))
\cong
\Ext^n_A\left(\varinjlim_\Gamma V_\Gamma, \bar\nabla(b)\right)
\cong
\varprojlim_\Gamma \Ext^n_A\left(V_\Gamma, \bar\nabla(b)\right).
$$
This is 0 for $n \geq 1$ 
thanks to \cref{extvanishingt} as
each $V_\Gamma$
has a $\Delta$-flag by the definition of
ascending $\Delta$-flag.

In this paragraph, we prove that (c)$\Rightarrow$(b).
Take a finite lower set $\Gamma$.
Note that $j^\Gamma V$ is locally finite-dimensional and bounded below since $V$ has these properties. Also for $b \in \B_\Gamma$, we have that $\Ext^1_{A_\Gamma}(j^\Gamma V, \bar\Delta_\Gamma(b)) \cong \Ext^1_A(V, j^\Gamma_* \bar\Delta_\Gamma(b))$ by \cref{anothergss}. Since $j^\Gamma_* \bar\Delta_\Gamma(b)
\cong \bar\Delta(b)$ by \cref{easy},
the assumed property (3) gives
that $\Ext^1_{A_\Gamma}(j^\Gamma V, \bar\Delta_\Gamma(b)) = 0$. Now we can apply \cref{citizens} (using that $\Gamma$ is finite so the support condition is automatic)
to establish (2).

For (b)$\Rightarrow$(a), assume that (b) holds.
\cref{lunch} implies that $j^\Gamma_! j^\Gamma V$ has
a $\Delta$-flag. We claim that the
counit of adjunction
$\eps^\Gamma_V: j^\Gamma_! j^\Gamma V \rightarrow V$
is injective. Given this, the image of
$\eps^\Gamma_V$ is $V_\Gamma$, so we deduce that
$V_\Gamma$ has a $\Delta$-flag, as needed to prove (a).
Suppose for a contradiction that $\eps^\Gamma_V$ is not
injective. Then we can find $\lambda \in \Lambda$
such that the restriction of $\eps^\Gamma_V$ to the $\lambda$-weight space is not injective.
Let $\Pi$ be the finite lower set generated by $\Gamma$ and $\lambda$. Consider the following diagram:
$$
\begin{tikzcd}
j^\Pi_! j^\Pi V\arrow[r,"\eps^\Pi_V"]&V\\
\arrow[u,hookrightarrow,"\eps^\Gamma_{\!\scriptscriptstyle j^\Pi_! j^\Pi V}"]j^\Gamma_! j^\Gamma
\left(j^\Pi_! j^\Pi V\right)\arrow[r,"\sim"]&
j^\Gamma_! j^\Gamma V\arrow[u,"\eps^\Gamma_V"]
\end{tikzcd}
$$
The bottom map
is 
$$
A e_\Gamma \otimes_{A_\Gamma} e_\Gamma A e_\Pi
\otimes_{A_\Pi} e_\Pi V \rightarrow
A e_\Gamma \otimes_{A_\Gamma} e_\Gamma V,\qquad
a e_\Gamma \otimes e_\Gamma b e_\Pi \otimes
e_\Pi v \mapsto a e_\Gamma \otimes e_\Gamma b e_\Pi v,
$$
which is clearly an isomorphism.
The left hand map is
$$
A e_\Gamma \otimes_{A_\Gamma} e_\Gamma A e_\Pi
\otimes_{A_\Pi} e_\Pi V \rightarrow
A e_\Pi \otimes_{A_\Pi} e_\Pi V,\qquad
a e_\Gamma \otimes e_\Gamma b e_\Pi \otimes e_\Pi v
\mapsto a e_\Gamma b e_\Pi \otimes e_\Pi v.
$$
This map is injective. To see this, let
$W := j^\Pi_! j^\Pi V$. We know it has a $\bar\Delta$-flag,
so by \cref{pizza}(1) we deduce that $W_\Gamma$ has a
$\bar\Delta$-flag too. Hence, by \cref{aha}(1),
$\eps^\Gamma_W:j^\Gamma_! j^\Gamma W
\rightarrow W_\Gamma$ 
is an isomorphism.
The top and right hand maps in the diagram are the natural multiplication maps, and it is easily checked that the diagram commutes. Finally, the top map becomes an isomorphism when we apply $j^\Pi$, hence, it
is a bijection on $\lambda$-weight spaces.
It follows that $\eps^\Pi_V \circ \eps^\Gamma_W$ is injective on the $\lambda$-weight space. Hence,
$\eps^\Gamma_V$ is injective on the $\lambda$-weight space, contradicting the earlier assumption.

Finally, we assume (a) and deduce \cref{cats}.
Let $\Gamma \subseteq \Lambda$ be a finite lower set.
We have that $V_\Gamma$ has a $\Delta$-flag by the definition. Also $V / V_\Gamma$ has an ascending $\Delta$-flag, as follows 
directly from the definition using \cref{pizza}(1).
Now to establish \cref{cats}, 
one just has to apply
$\Hom_A(-,\bar\nabla(b))$ to the short exact
sequence $0 \rightarrow V_\Gamma \rightarrow V \rightarrow V / V_\Gamma \rightarrow 0$.
Since
$\Ext^1_A(V / V_\Gamma, \bar\nabla(b)) = 0$,
one obtains a short exact sequence
showing that
$$
(V:\Delta(b))_q = (V_\Gamma:\Delta(b))_q + (V / V_\Gamma:\Delta(b))_q.
$$
Defining supports as in \cref{supp1,supp2},
we also have that
$\supp_\Delta(V_\Gamma) \subseteq \Gamma$, e.g., this follows from \cref{lunch} because $V_\Gamma \cong j^\Gamma_! j^\Gamma V$
by \cref{aha}(1) and $j^\Gamma V \in A_\Gamma\gmoddelta$.
Also $\supp_\Delta(V / V_\Gamma) \subseteq \Lambda - \Gamma$ since it has no weights belonging to $\Gamma$.
Now \cref{cats} is clear.

\vspace{1mm}
\noindent
(2) Similar.
\end{proof}

\begin{corollary}
Suppose that $0 \rightarrow U \rightarrow V \rightarrow W 
\rightarrow 0$ is a short exact sequence of 
graded left $A$-modules. 
Assuming that $W$ has an ascending $\Delta$-flag, $U$ has an ascending $\Delta$-flag if and only if $V$ has an ascending $\Delta$-flag.
Similarly for $\bar\Delta$-flags.
\end{corollary}

Finally, we state the dual results which, as usual, follow
from by dualizing the above.

\begin{theorem}[Homological criteria for 
descending $\nabla$- and $\bar\nabla$-flags]\label{villainss}
Assume that $V \in \ob A\gmod$ is locally finite-dimensional and bounded above.
\begin{itemize}
\item[(1)]
The following are equivalent:
\begin{enumerate}
\item[(a)] $V$ has an descending $\nabla$-flag;
\item[(b)] $j^\Gamma V$ has a $\nabla$-flag for all finite
lower sets $\Gamma \subseteq \Lambda$;
\item[(c)] 
$\Ext^1_A(\bar\Delta(b),V) = 0$ for all $b \in \B$;
\item[(d)] 
$\Ext^n_A(\bar\Delta(b),V) = 0$ for all $b \in \B$ and $n \geq 1$.
\end{enumerate}
When this holds, for any finite lower set $\Gamma \subseteq \Lambda$, both $V / V^\Gamma$ and 
$V^\Gamma$ have ascending $\nabla$-flags
with
\begin{align}
(V / V^\Gamma:\nabla(b))_q &= \left\{
\begin{array}{ll}
(V:\nabla(b))_q&\text{if $b \in \B_\Gamma$}\\
0&\text{otherwise;}
\end{array}
\right.
&
(V^\Gamma:\nabla(b))_q &= \left\{
\begin{array}{ll}
0&\text{if $b \in \B_\Gamma$}\\
(V:\nabla(b))_q&\text{otherwise.}
\end{array}
\right.
\end{align}
\item[(2)]
The following are equivalent:
\begin{enumerate}
\item[(a)] $V$ has an descending $\bar\nabla$-flag;
\item[(b)] 
$j^\Gamma V$ has a $\bar\nabla$-flag for all finite
lower sets $\Gamma \subseteq \Lambda$;
\item[(c)] 
$\Ext^1_A(\Delta(b),V) = 0$ for all $b \in \B$;
\item[(d)] 
$\Ext^n_A(\Delta(b),V) = 0$ for all $b \in \B$ and $n \geq 1$.
\end{enumerate}
When this holds, for any finite lower set $\Gamma \subseteq \Lambda$, both $V / V^\Gamma$ and 
$V^\Gamma$ have ascending $\bar\nabla$-flags
with
\begin{align}
(V / V^\Gamma:\bar\nabla(b))_q &= \left\{
\begin{array}{ll}
(V:\bar\nabla(b))_q&\text{if $b \in \B_\Gamma$}\\
0&\text{otherwise;}
\end{array}
\right.
&
(V^\Gamma:\bar\nabla(b))_q &= \left\{
\begin{array}{ll}
0&\text{if $b \in \B_\Gamma$}\\
(V:\bar\nabla(b))_q&\text{otherwise.}
\end{array}
\right.
\end{align}
\end{itemize}
\end{theorem}

\begin{corollary}
Suppose that $0 \rightarrow U \rightarrow V \rightarrow W 
\rightarrow 0$ is a short exact sequence of 
graded left $A$-modules. 
Assuming that $U$ has an ascending $\nabla$-flag, $V$ has an ascending $\nabla$-flag if and only if $W$ has an ascending $\nabla$-flag.
Similarly for $\bar\nabla$-flags.
\end{corollary}

\section{Homological dimensions}

In this section, we give some applications to homological dimensions.
Often these require some Noetherian assumptions (something we have sought to avoid up until now).
Continue with $A$ having a graded triangular basis.
We say that $A$ is {\em locally left} (resp., {\em right}) {\em graded Noetherian} if each finitely generated projective graded left (resp., right) $A$-module
has the descending chain condition (DCC) on graded submodules.
Since $A$ is locally finite-dimensional, 
this is obviously equivalent by duality to 
each finitely cogenerated injective graded right (resp., left) $A$-module having ACC.
If $A$ is both locally left and locally 
right graded Noetherian, then its
(possibly infinite) left and right graded global dimensions
coincide, and we refer to them both just as 
 the graded global dimension of $A$. 
Without this assumption, one must talk about
the left and right graded global dimensions of $A$ separately.
This is the same as for ordinary (graded) algebras,
e.g., see \cite[Ch.~4]{Weibel}.


\begin{lemma}\label{step}
For $\lambda \in \Lambda$, let $\ell(\lambda)$
be the maximal length of a descending chain
$\lambda = \lambda_0 > \lambda_1 >\cdots>\lambda_\ell$
in the poset $\Lambda$.
For any $b \in \B_\lambda$,
the graded projective (resp., injective) dimension
of a $\Delta$-layer (resp., a $\nabla$-layer)
of type $\lambda$ is $\leq \ell(\lambda)$.
\end{lemma}

\begin{proof}
We just explain for $\Delta$-layers; the argument for $\nabla$-layers is similar.
By \cite[Ex.~4.1.3(1)]{Weibel},
it suffices to show that the graded projective
dimension of $\Delta(b)$ is $\leq \ell(\lambda)$
for $b \in \B_\lambda$.
We prove this by induction on $\ell(\lambda)$.
If $\ell(\lambda) = 0$
then $\Delta(b)$ is projective by \cref{gym},
giving the induction base.
Now suppose that $\ell(\lambda) > 0$.
By \cref{bggproj}, 
there is a short exact sequence 
$0 \rightarrow K \rightarrow P(b) \rightarrow \Delta(b)\rightarrow 0$
such that $K$ has a $\Delta$-flag
with sections that are $\Delta$-layers of types
$\mu$ with $\ell(\mu) < \ell(\lambda)$.
By \cite[Ex.~4.1.2(1)]{Weibel} and the induction hypothesis, it follows that the graded projective
dimension of $K$ is $<\ell(\lambda)$.
Another application of \cite[Ex.~4.1.2(1)]{Weibel}
shows that the
graded projective dimension
of $\Delta(b)$ is at most one more than that of $K$.
Hence, the graded projective dimension of $\Delta(b)$
is $\leq \ell(\lambda)$.
\end{proof}

\begin{lemma}\label{logs}
Suppose that we are given $\lambda \in \Lambda$ such that
$A_\lambda$ has finite
left (resp., right) graded global dimension $d(\lambda)$. Then any $\bar\Delta$-layer (resp., $\bar\nabla$-layer)
of type $\lambda$
has finite graded projective (resp., injective) dimension
that is $\leq \ell(\lambda)+d(\lambda)$.
\end{lemma}

\begin{proof}
We go through the argument for a
$\bar\Delta$-layer $V = j^\lambda_! \bar V$ 
of type $\lambda$.
Since $\bar V$ is locally finite-dimensional and bounded below, \cref{tech1}(1) implies that 
it has a projective cover $P_0$ in $A_\lambda\gmod$
which is again locally finite-dimensional and bounded below.
It follows that the kernel of 
$P_0 \twoheadrightarrow \bar V$
is locally finite-dimensional and bounded below.
Repeating the argument, we end up with a minimal
graded projective resolution of $\bar V$
of the form
$0 \rightarrow P_n \rightarrow \cdots \rightarrow P_0
\rightarrow \bar V \rightarrow 0$
for $n \leq d(\lambda)$
and each $P_r$ being a projective graded 
module that is locally finite-dimensional and bounded below. 
Then we apply $j^\lambda_!$ to obtain 
an exact sequence
$0 \rightarrow j^\lambda_! P_n \rightarrow \cdots \rightarrow j^\lambda_! P_0 \rightarrow V \rightarrow 0$
with $n \leq d(\lambda)$ and each $j^\lambda_! P_r$
being $\Delta$-layer of type $\lambda$.
We deduce that $V$ is of finite graded projective dimension $\leq \ell(\lambda)+d(\lambda)$ 
using \cref{step}.
\end{proof}

\begin{lemma}\label{lugs}
If $A$ is locally left (resp., right) graded Noetherian
then all $A_\lambda\:(\lambda\in\Lambda)$
are left (resp., right) graded Noetherian.
\end{lemma}

\begin{proof}
Assume that $A$ is locally left graded Noetherian.
Take $b \in \B_\lambda$ and a descending chain
$P_\lambda(b) = P_0 \supseteq P_1 \supseteq \cdots$
of graded submodules.
Apply the exact functor $j^\lambda_!$
to get a descending chain of graded submodules
of the standard module $\Delta(b)$. Since $A$ is locally left graded Noetherian and this module is finitely generated, it follows that the chain stabilizes. Then apply $j$ using
$j \circ j^\lambda_! \cong \id_{A_\lambda\gmod}$ 
to deduce that the original chain stabilizes too.
This proves that $A_\lambda$ is left graded Noetherian.
A similar argument starting with an ascending chain of graded submodules of $I_\lambda(b)$ and using
$j \circ j^\lambda_* \cong \id_{A_\lambda\gmod}$
proves that $A_\lambda$ is right graded Noetherian
when $A$ has this property.
\end{proof}

\begin{lemma}\label{laurent}
Assume that $A$ is locally left (resp., right) graded Noetherian.
Then $(P(b):\Delta(a))_q$ and $[\bar\nabla(a):L(b)]_q$ 
(resp., 
$(I(b):\nabla(a))_q$ and $[\bar\Delta(a):L(b)]_q$)
are Laurent polynomials in $\N[q,q^{-1}]$
for all $a,b \in \B$. 
\end{lemma}

\begin{proof}
We just explain for the case of left Noetherian.
If $(P(b):\Delta(a))_q = \overline{[\bar\nabla(a):L(b)]_q}$ is {\em not}
in $\N[q,q^{-1}]$ for some $a,b \in \B$
then \cref{bggproj}
implies that there is a $\Delta$-flag
$P(b) = P_0 \supseteq P_1 \supseteq \cdots \supseteq P_n = 0$ such that for some $r$ the section
$P_{r-1} / P_r$ is an infinite direct sum of degree-shifted
standard modules.
This implies that $P_{r-1}$ is not finitely generated, hence, $P(b)$ is not graded Noetherian,
contradicting the assumption that $A$ is locally left graded Noetherian.
\end{proof}

\begin{corollary}\label{laurentc}
If $A$ is {unital} and 
locally left (resp., right) graded Noetherian,
then all of the proper standard modules $\bar\Delta(b)$
(resp., the proper costandard modules $\bar\nabla(b)$)
are of finite length.
\end{corollary}

\begin{theorem}\label{birdie}
Assume $A$ is unital,
both left and right graded Noetherian,
and that each $A_\lambda\:(\lambda \in \Lambda)$ has finite graded global dimension.
Then $A$ has finite
graded global dimension.
\end{theorem}

\begin{proof}
It suffices to 
show that $A$ has finite left graded global dimension.
By \cite[Th.~4.1.2(3)]{Weibel},
we need to show that there is $N \in \N$ such that
$\Ext^n_A(V, W) = 0$
for $n > N$, all finitely generated graded left $A$-modules $V$ and arbitrary graded left $A$-modules $W$.
In fact, we may also assume that $W$ is finitely generated.
To prove this,
we use that $A$ is graded left Noetherian to construct a graded projective resolution
$\cdots \rightarrow
P_{n+1}\stackrel{\partial_n}{\rightarrow} P_n 
\stackrel{\partial_{n-1}}{\rightarrow} P_{n-1}\rightarrow \cdots \rightarrow P_0 \rightarrow V \rightarrow 0$ all of whose terms 
are finitely generated.
Any element of $\Ext^n_A(V,W)$
is 
represented by a homomorphism $f:P_n \rightarrow W$
such that $f \circ \partial_n = 0$.
The image of $f$ is a finitely generated submodule $W'$ 
of $W$.
If we know $\Ext^n_A(V,W') = 0$,
then $f = g \circ \partial_{n-1}$ for some
$g:P_{n-1} \rightarrow W'$, and we deduce that the image of $f$ in 
$\Ext^n_A(V,W)$ is zero, hence, $\Ext^n_A(V,W) = 0$.
So now we have reduced the problem
to showing that there exists $N \in \N$ such that
$\Ext^n_A(V,W) = 0$ for $n > N$
and all finitely generated
graded left $A$-modules $V$ and $W$. By \cite[Lem.~1.1]{abcdefg},
the proof reduces 
further to checking this statement just for all
{\em irreducible} $W$.

Thus, the proof has been reduced 
to showing that all of the
irreducible modules $L(b)\:(b \in \B)$ have finite graded injective dimension.
Replacing $\Lambda$ by $\{\dot b\:|\:b \in \B\}$,
we may assume that the poset $\Lambda$ is finite, and
proceed by downward induction on this poset.
Take any $b \in \B$
and consider the short exact sequence 
$0 \rightarrow L(b) \rightarrow \bar\nabla(b)
\rightarrow Q \rightarrow 0$.
By \cref{laurentc}, 
$Q$ is of finite length. Moreover, all of its composition factors
are degree shifts of $L(c)$ for $c \in \B$ 
with $\dot c > \dot b$.
By induction, 
they are all of finite injective dimension,
hence, $Q$ is of finite injective dimension.
Also $A_\lambda$ is graded right Noetherian by \cref{lugs}, so $\bar\nabla(b)$ is of finite injective dimension
by \cref{logs}. It follows that $L(b)$ has finite 
graded injective dimension.
\end{proof}

\section{Refinement}

What happens if the algebras $A_\lambda$ have additional structure? 
The results in this section address this question in the situation that each $A_\lambda$ is itself a based affine quasi-hereditary algebra in the sense of \cref{nice}.
This means that
we are given partial orders
$\leq_\lambda$ on $\B_\lambda$
and ``local"
graded triangular bases for each $\lambda \in \Lambda$
making the unital graded algebras 
$A_\lambda$ into based affine quasi-hereditary algebras with respect to the posets $(\B_\lambda,\leq_\lambda)$.
Then we can define a refined partial order
$\leq$ on $\B$ by
\begin{equation}\label{neworder}
b \leq c
\qquad \Leftrightarrow 
\qquad 
\dot b < \dot c\text{ or }
(\lambda := \dot b = \dot c
\text{ and }b\leq_\lambda c).
\end{equation}
One might hope to be able to assemble the various local triangular bases into a new global triangular basis making $A$ into a based affine quasi-hereditary algebra with weight poset $(\B, \leq)$.
Unfortunately this seems to be 
difficult to do directly. However, it is still possible to prove the representation theoretical consequences of the existence of such a basis. 

So assume from now on that we are given
a graded triangular basis for $A$ as usual, and additional 
partial orders $\leq_\lambda$ on each of the finite sets $\B_\lambda\:(\lambda \in \Lambda)$.
Define $\leq$ on $\B$ as in \cref{neworder}.
For each $\lambda \in \Lambda$, we assume that 
$A_\lambda$ has some extra structure making it into a based affine quasi-hereditary algebra with respect to the poset $(\B_\lambda,\leq_\lambda)$. We will never refer explicitly to these bases, rather, we will work with them implicitly in terms of the consequences of the existence of these bases for the categories $A_\lambda\gmod$.
We denote the various families of graded modules
for $A_\lambda$ arising from the extra structure
by $P_\lambda(b), \Delta_\lambda(b),
\bar\Delta_\lambda(b), L_\lambda(b),
\bar\nabla_\lambda(b), \nabla_\lambda(b)$
and $I_\lambda(b)$, all for $b \in \B_\lambda$.
There are corresponding 
notions of $\Delta$-layers, 
$\bar\Delta$-layers, $\Delta$-flags, $\bar\Delta$-flags, etc. for $A_\lambda$-modules, which we will call
$\Delta_\lambda$-layers, $\bar\Delta_\lambda$-layers, $\Delta_\lambda$-flags,
$\bar\Delta_\lambda$-flags, etc. for extra clarity.
As well the usual $A$-modules
$\Delta(\lambda), \bar\Delta(\lambda),
\bar\nabla(\lambda)$ and $\nabla(\lambda)$
defined as in \cref{standards}, we also have
\begin{align}
\pureDelta(b) &:= j^\lambda_! \Delta_\lambda(b),&
\bar\pureDelta(b) &:= j^\lambda_! \bar\Delta_\lambda(b),&
\bar\purenabla(b) &:= j^\lambda_* \bar\nabla_\lambda(b),&
\purenabla(b) &:= j^\lambda_* \nabla_\lambda(b)
\end{align}
for $b \in \B_\lambda$.
We call these the {\em pure standard}, {\em pure proper standard}, {\em pure costandard} and {\em pure proper costandard modules}, respectively.

\begin{lemma}\label{upupandaway}
For $b, c \in \B$, $f \in \N\lround q \rround$ and $g \in \N\lround q^{-1}\rround$,
we have that
$$
\dim_q \Hom_A\big(\pureDelta(b)^{\oplus f}, \bar\purenabla(c)^{\oplus g}\big)=
\dim_q \Hom_A\big(\bar\pureDelta(b)^{\oplus f}, \purenabla(c)^{\oplus g}\big)= \delta_{b,c}\,\overline{f}\,g \in \N \lround q^{-1} \rround.
$$
\end{lemma}

\begin{proof}
We just explain for the first space.
Since $\pureDelta(b)$ has lowest weight $\dot b$ 
and $\bar\purenabla(c)$ has lowest weight $\dot c$,
the space is zero unless $\lambda := \dot b = \dot c$.
Assuming this, we have that
\begin{align*}
\Hom_A(\pureDelta(b)^{\oplus f}, \bar\purenabla(c)^{\oplus g})
&=
\Hom_{A_{\geq \lambda}}(\pureDelta(b)^{\oplus f}, \bar\purenabla(c)^{\oplus g})
=\Hom_{A_{\geq \lambda}}(j^\lambda_! \Delta_\lambda(b)^{\oplus f}, j^\lambda_*\bar\nabla_\lambda(c)^{\oplus g})\\
&\cong
\Hom_{A_\lambda}(\Delta_\lambda(b)^{\oplus f},
j^\lambda j^\lambda_* \bar\nabla_\lambda(b)^{\oplus g})
=
\Hom_{A_\lambda}(\Delta_\lambda(b)^{\oplus f},
\bar\nabla_\lambda(b)^{\oplus g}),
\end{align*}
which is of graded dimension $\delta_{b,c} \overline{f} g$ by \cref{upup}.
\end{proof}

\begin{definition}\label{defupnew}
By a {\em $\pureDelta$-layer} (resp., a {\em $\bar\pureDelta$-layer}) of type $b \in \B_\lambda$,
we mean a graded $A$-module that is isomorphic to
$j^\lambda_! \bar V$
for  a
graded left $A_\lambda$-module $\bar V$
which is a $\Delta_\lambda$-layer (resp., a  $\bar\Delta_\lambda$-layer) of type $b$.
We say that $V \in \ob A\gmod$ has a 
{\em $\pureDelta$-flag} (resp., a {\em $\bar\pureDelta$-flag}) if for some $n \geq 0$
there is a graded
filtration $$
0 = V_0 \subset V_1 \subset\cdots \subset V_n = V
$$
and distinct $b_1,\dots,b_n \in \B$
such that $V_{r} / V_{r-1}$ is a $\pureDelta$-layer (resp., a $\bar\pureDelta$-layer) of type
$b_r$ for each $r= 1,\dots,n$.
We say that $V$ has an
{\em ascending $\pureDelta$-flag} 
(resp., an {\em ascending $\bar\pureDelta$-flag})
if the $A$-submodule $V_\Gamma$ defined in \cref{tools} 
has a $\pureDelta$-flag (resp., a $\bar\pureDelta$-flag)
for all finite 
lower sets $\Gamma \subseteq \Lambda$.
\end{definition}

\begin{definition}\label{defdownnew}
By a {\em $\purenabla$-layer} (resp., a {\em $\bar\purenabla$-layer}) of type $b \in \B_\lambda$,
we mean a graded $A$-module that is isomorphic to
$j^\lambda_* \bar V$
for  a
graded left $A_\lambda$-module $\bar V$
which is a $\nabla_\lambda$-layer (resp., a  $\bar\nabla_\lambda$-layer) of type $b$.
We say that $V \in \ob A\gmod$ has a 
{\em $\purenabla$-flag} (resp., a {\em $\bar\purenabla$-flag}) if for some $n \geq 0$
there is a graded
filtration $$
V = V_0 \supset V_1 \supset\cdots \supset V_n = 0
$$
and distinct $b_1,\dots,b_n \in \B$
such that $V_{r-1} / V_{r}$ is a $\purenabla$-layer (resp., a $\bar\purenabla$-layer) of type
$b_r$ for each $r= 1,\dots,n$.
We say that $V$ has an
{\em ascending $\purenabla$-flag} 
(resp., an {\em ascending $\bar\purenabla$-flag})
if the quotient module $V / V^\Gamma$ 
defined in \cref{tools} 
has a $\purenabla$-flag (resp., a $\bar\purenabla$-flag)
for all finite 
lower sets $\Gamma \subseteq \Lambda$.
\end{definition}

\begin{remark}\label{jammier}
A $\pureDelta$-layer of type $b$ means just the same thing as a direct sum
$\pureDelta(b)^{\oplus f}$
for $f \in \N\lround q^{-1}\rround$.
Similarly, a $\purenabla$-layer of type $b$ is a direct sum
$\purenabla(b)^{\oplus g}$
for $g \in \N\lround q\rround$.
\end{remark}

The full subcategory of $A\gmod$
consisting of modules with $\pureDelta$-flags
(resp., $\bar\pureDelta$-flags, $\purenabla$-flags, $\bar\purenabla$-flags) will be denoted
$A\gmodpuredelta$ (resp. $A\gmodpuredeltabar$,
$A\gmodpurenabla$, $A\gmodpurenablabar$).
Since \cref{defupnew,defdownnew} are dual to each other, 
from now on, 
we will explain results just in the case of $\pureDelta$- and $\bar\pureDelta$-flags,
leaving the dual statements for $\purenabla$- and $\bar\purenabla$-flags to the reader.

Noting that $\pureDelta$-layers are $\bar\pureDelta$-layers,
  $A\gmodpuredelta$ is a subcategory of
$A\gmodpuredeltabar$.
It is also the case 
that $A\gmoddelta$ is a subcategory of
$A\gmodpuredelta$ and $A\gmodpuredeltabar$ is a subcategory
of $A\gmoddeltabar$. These statements are not quite obvious; they
are justified by the corollary appearing after the next lemma.

\begin{lemma}\label{archersnew}
In either of the following situations, we
have that $\Ext^1_A(V,W) = 0$:
\begin{enumerate}
\item $V$ is a $\bar\pureDelta$-layer
of type $b$ and $W$ is a $\bar\pureDelta$-layer of type $c$
for $b \ngeq c$;
\item
$V$ is a $\pureDelta$-layer
of type $b$ and $W$ is a $\pureDelta$-layer of type $c$
for $b \ng c$.
\end{enumerate}
\end{lemma}

\begin{proof}
(1) Suppose that $V = j^\lambda_! \bar V$
for a $\bar\Delta_\lambda$-layer $\bar V$ of type $b \in \B_\lambda$ and
$W = j^\mu_! \bar V$
for a $\bar\Delta_\mu$-layer $\bar V$ of type $c \in \B_\mu$.
The hypothesis that $b\ngeq c$ means either that $\lambda \ngeq \mu$, or $\lambda=\mu$ and $b \,\ngeq_\lambda\, c$.
Since $\bar\pureDelta$-layers are $\bar\Delta$-layers,
\cref{archers} gives $\Ext^1_A(V,W) = 0$ if $\lambda\neq \mu$.
Now suppose that $\lambda = \mu$.
We have that $\Ext^1_A(V,W) \cong \Ext^1_{A_{\geq \lambda}}(V,W)$ which, is isomorphic to $\Ext^1_{A_\lambda}(\bar V, \bar W)$
by \cref{anothergss2}.
As $b\,\ngeq_\lambda\,c$, this is zero thanks to \cref{archers}
in $A_\lambda\gmod$.

\vspace{1mm}
\noindent
(2)
This is similar to (1) using \cref{archersr} in place of
\cref{archers}.
\end{proof}

\begin{corollary}
If $V$ has a $\Delta$-flag (resp., a $\bar\pureDelta$-flag) then it has a $\pureDelta$-flag (resp., a $\bar\Delta$-flag).
\end{corollary}

\begin{proof}
First suppose that $V$ has a $\Delta$-flag.
Take $b \in \B_\lambda$. 
Applying $j^\lambda_!$ to a $\Delta_\lambda$-flag
for $P_\lambda(b)$ arising from \cref{bggproj},
we deduce that $\Delta(b)$ has a filtration of finite length
with top section $\pureDelta(b)$ and other sections that are
$\pureDelta$-layers of types $c \in \B_\lambda$ with $c <_\lambda b$. Using \cref{jam}, it
follows easily that any $\Delta$-layer of type $\lambda$
has a filtration of finite length with sections that are
$\pureDelta$-layers of types $c \in \B_\lambda$.
Hence, $V$ itself 
has a filtration of finite length with sections that
are $\pureDelta$-layers of types $c \in \B$. However, this is not yet a $\pureDelta$-flag of $V$ due to the requirement that $b_1,\dots,b_n$ are distinct in \cref{defdownnew}.
To fix the problem, we first use \cref{archersnew}(2) to order the $\pureDelta$-layers in some order refining the order $\leq$ on $\B$ (biggest at the top). It could still be that there are several neighboring layers of the same type, but these
can be combined into a single $\pureDelta$-layer by taking their direct sum. This uses the fact that
 $\Ext^1_A(V, W) = 0$ if $V$ and $W$ are $\pureDelta$-layers of the same type, which is \cref{archersnew}(2) again.

Now assume that $V$ has a $\bar\pureDelta$-flag.
Since $\bar\pureDelta$-layers are $\bar\Delta$-layers,
this means that $V$ has a finite filtration with sections that
are $\bar\Delta$-layers of types $\lambda \in \B$.
However this is not a $\bar\Delta$-flag due to the requirement that $\lambda_1,\dots,\lambda_n$ are distinct in \cref{defup}.
To fix the problem, we first use \cref{archers} to reorder the sections if necessary. Then we have to merge neighboring $\bar\Delta$-layers of the same type into a single $\bar\Delta$-layer. This follows because if
$0 \rightarrow U \rightarrow V \rightarrow W \rightarrow 0$
is an extension of two $\bar\Delta$-layers of type $\lambda$
then $V \cong j^\lambda_! j^\lambda V$, so it is itself a $\bar\Delta$-layer of type $\lambda$. Indeed, the counit of adjunction gives a homomorphism
$j^\lambda_! j^\lambda V \rightarrow V$. This homomorphism is an isomorphism because $j^\lambda_! \circ j^\lambda$ is exact
and the counit of adjunction is an isomorphism 
on $U= j^\lambda_! \bar U$ and $W = j^\lambda_! \bar W$.
\end{proof}

\begin{corollary}\label{breathe}
For $b \in \B$, the indecomposable projective module
$P(b)$ has a $\pureDelta$-flag with top section
$\pureDelta(b)$ and other sections that are $\pureDelta$-layers
of types $a < b$.
\end{corollary}

\begin{proof}
By \cref{bggproj}, we know that $P(b)$
has a $\Delta$-flag with top section $\Delta(b)$
and other sections that are $\Delta$-layers of types
$\mu < \dot b$.
Also $\Delta(b)$ has a $\pureDelta$-flag
with top section $\pureDelta(b)$ and other sections
that are $\pureDelta$-layers of types
$a <_\lambda b$. This filtration can be converted into 
the desired $\pureDelta$-flag by using \cref{archersnew}(2) as in the proof of the previous corollary.
\end{proof}

\cref{breathe} is the key property needed
to upgrade other results about $\Delta$- and $\bar\Delta$-flags to $\pureDelta$- and $\bar\pureDelta$-flags.
To start with, the results about truncation to upper sets from \cref{tuper} carry over to the refined setting.
In particular,
letting $\hat\Lambda$ be an upper set
in $\Lambda$ and $\hat A$ be as in \cref{tuper},
we have the following:
\begin{itemize}
\item
For $V \in \ob A\gmodpuredelta$ and $i \in \I$, we have that
$\Tor^A_m(1_i \widehat A,V) = 0$ for all $m \geq 1$.
\item
The functor $i^* = A \otimes_{\hat A} -$
takes short exact sequences of modules with
$\pureDelta$-flags to short exact sequences of modules with $\pureDelta$-flags.
\item
For $V \in \ob A \gmodpuredelta$
and $W \in \ob \hat A\gmod$,
we have that $\Ext^n_A(V, iW)
\cong \Ext^n_{\hat A}(i^* V, W)$ for all $n \geq 0$.
\end{itemize}
These follow by mimicking the proofs
of \cref{previousl}, \cref{previousc} and
\cref{mercedes}, respectively, using the $\pureDelta$-flag of $P(b)$ from \cref{breathe} in place of the arguments with the $\Delta$-flag of $Q(b)$ given before.

\begin{theorem}\label{pureexts}
In either of the following situations, we
have that $\Ext^n_A(V,W) = 0$ for all $n \geq 1$:
\begin{enumerate}
\item 
$V \in \ob A\gmodpuredelta$ and $W \in \ob A\gmodpurenablabar$; 
\item $V \in \ob A\gmodpuredeltabar$ and $W \in \ob A\gmodpurenabla$.
\end{enumerate}
\end{theorem}

\begin{proof}
(1) The strategy is similar to the proof of \cref{extvanishingt}.
Using \cref{jammier}, we reduce to checking that
$\Ext^n_A(\pureDelta(b), W) = 0$ for $b \in \B$, 
$n \geq 1$ and $W := j^\lambda_! \bar W$ for 
a $\bar\Delta_\lambda$-layer $\bar W$ of 
type $b \in \B_\lambda$. By the third
point noted just before the statement of the theorem, i.e., 
the analog of \cref{mercedes} for $\pureDelta$-flags, we have that
$$
\Ext^n_A\left(\pureDelta(b), W\right)
\cong
\Ext^n_{A_{\geq \lambda}}\left(i^*_{\geq \lambda} \pureDelta(b), j^\lambda_! \bar W\right).
$$
This is clearly zero if $\dot b \not\geq \lambda$.
When $\dot b \geq \lambda$, we apply \cref{ssarg} to get that
$$
\Ext^n_{A_{\geq \lambda}}\left(i^*_{\geq \lambda} \pureDelta(b), j^\lambda_! \bar W\right)
\cong
\Ext^n_{A_\lambda}\left(j^\lambda \pureDelta(b),
\bar W\right).
$$
This is clearly zero if $\dot b \neq \lambda$. Finally,
when $\dot b = \lambda$, it is zero thanks to \cref{extvanishingt}
applied in $A_\lambda\gmod$.

\vspace{1mm}
\noindent
(2) This follows from (1) for $A^\op$ plus \cref{hands}.
\end{proof}

\cref{upupandaway,pureexts} justify the following definitions for $V$ with a $\pureDelta$-flag or a $\bar\pureDelta$-flag, respectively:
\begin{align}
(V:\pureDelta(b))_q &:= \overline{\dim_q \Hom_A(V, \bar\purenabla(b))},&
(V:\bar\pureDelta(b))_q &:= \overline{\dim_q \Hom_A(V, \purenabla(b))_q}.
\end{align}
These are analogous to \cref{seasonalmuffin,tea}.
Now we can strengthen \cref{breathe}:

\begin{corollary}[Pure BGG reciprocity]\label{breathless}
For $a,b \in \B$, we have that $(P(b):\pureDelta(a))_q =
\overline{[\bar\purenabla(a):L(b)]_q}$.
If the graded triangular basis for $A$
admits a duality $\tau$ such that for each $\lambda \in \Lambda$ the induced duality
on $A_\lambda\gmod$ satisfies
$\nabla_\lambda(b)^\taudual \cong \Delta_\lambda(b)$ for all $b \in \B_\lambda$,
this graded multiplicity is also equal to
$[\bar\pureDelta(a):L(b)]_q$.
\end{corollary}

\begin{proof}
We know already from \cref{breathe}
that $P(b)$ has a $\pureDelta$-flag.
We have that
$$
(P(b):\pureDelta(a))_q = 
\overline{
\dim_q \Hom_A(P(b), \bar\purenabla(a))}
=
\overline{[\bar\purenabla(a):L(b)]_q}.
$$
In the presence of the 
duality, for $a \in \B_\lambda$, 
we have that $$
\bar\purenabla(a)^\taudual
= (j^\lambda_* \nabla_\lambda(a))^\taudual
\cong j^\lambda_! (\nabla_\lambda(a)^\taudual)
\cong j^\lambda_! \Delta_\lambda(a)
= \bar\pureDelta(a).
$$
So $\overline{[\bar\purenabla(a):L(b)]_q} = [\bar\pureDelta(a):L(b)]_q$.
\end{proof}

There are also analogs of \cref{citizens,villains} for $\pureDelta$- and $\bar\pureDelta$-flags. The statements are 
almost exactly the same as before, replacing the various standard, proper standard, costandard and proper costandard modules by their pure counterparts. 
The definitions of {\em $\pureDelta$-} and  {\em $\bar\pureDelta$-supports}
of $V\in A\gmod$ needed for the modified statement of \cref{citizens}
are:
\begin{align}\label{supp1pure}
\supp_{\pureDelta}(V) &:= \big\{\dot b\:\big|\:b \in \B\text{ such that }
\Hom_A(V, \bar\purenabla(b))\neq 0\big\},\\\label{supp2pure}
\supp_{\bar\pureDelta}(V) &:= \big\{\dot b\:\big|\:
b \in \B\text{ such that }\Hom_A(V, \purenabla(b))\neq 0\big\}.
\end{align}
We have that
$\supp_\Delta(V) 
\subseteq \supp_{\pureDelta}(V) 
\subseteq \supp_{\bar\pureDelta}(V) 
\subseteq\supp_{\bar\Delta}(V)$.
All of these sets are finite when $V$ is finitely generated
by \cref{supports}.
We leave full 
proofs of the analogs of \cref{citizens,villains} to the reader,
just recording one more lemma here which is the appropriate modification of the key 
\cref{inductionbase} in the new setting---with this in hand, the other modifications to the earlier arguments are straightforward.

\begin{lemma}\label{pureinductionbase}
Suppose that 
$\lambda \in \Lambda$ is minimal
and $V \in A\gmod$ has the following properties:
\begin{enumerate}
\item
$V$ is locally finite-dimensional and bounded below;
\item
$V = A e_\lambda V$;
\item
$\Ext^1_A(V, \bar\purenabla(b)) = 0$ 
(resp.,
$\Ext^1_A(V, \purenabla(b)) = 0$) for all
$b \in \B$.
\end{enumerate}
Then $V$ has a $\pureDelta$-flag (resp., a $\bar\pureDelta$-flag). More precisely, we have that $V \cong j^\lambda_! \bar V$ 
for $\bar V \in A_\lambda\gmod$
with a $\Delta_\lambda$-flag (resp., a $\bar\Delta_\lambda$-flag).
\end{lemma}

\begin{proof}
The same 
arguments as given in the proof of \cref{inductionbase}
show that 
$V \cong j^\lambda_! j^\lambda V$.
It then remains to show that
$\bar V := j^\lambda V$ has a $\Delta_\lambda$-flag (resp., 
a $\bar\Delta_\lambda$-flag). To see this, we can apply the homological criterion from \cref{citizens} 
in $A_\lambda\gmod$.
By \cref{ssarg}, we have that
$$
\Ext^1_{A_\lambda}\left(\bar V, W\right)
\cong \Ext^1_{A_{\geq \lambda}}\left(V, j^\lambda_* W\right)
\cong \Ext^1_{A}\left(V, j^\lambda_* W\right)
$$
for any $W \in A_{\lambda} \gmod$.
We apply this with $W = \bar\nabla_\lambda(c)$
(resp., $\nabla_\lambda(c)$)
for $c \in \B_\lambda$ using (3) to complete the proof.
\end{proof}

\begin{remark}
The principles outlined in this section are sufficiently robust that they can be adapted to various similar situations. For example, there is an analogous theory if we instead 
have that the categories of finitely generated graded
left $A_\lambda$-modules are {\em affine highest weight categories} 
in the sense of \cite[Def.~5.2]{AHW}
with weight posets
that are the opposites of the poset $(\B_\lambda, \leq_\lambda)$.
\end{remark}


\bibliographystyle{alpha}
\bibliography{triangular}
\end{document}